\documentclass[a4paper]{article}
\usepackage{hyperref}
\usepackage{graphicx}
\usepackage{mathrsfs}
\usepackage{amsmath}
\usepackage{enumerate} 
\usepackage{amsxtra,amssymb,latexsym, amscd,amsthm}
\usepackage{indentfirst}
\usepackage{color}
\usepackage[utf8]{inputenc}
\usepackage[mathscr]{eucal}
\usepackage{amsfonts}
\usepackage{graphics}
\usepackage{multirow}
\usepackage{array}
\usepackage{subfigure}

\usepackage{cite}
\usepackage{wrapfig}

\usepackage{algorithm}
\usepackage[noend]{algpseudocode}

\makeatletter
\def\algbackskip{\hskip-\ALG@thistlm}
\makeatother

\usepackage{ushort}
\usepackage{pdflscape}
\usepackage{rotating}
\usepackage{pgfplots}

\usepackage{subfigure}
\usepackage{tikz}
\usepackage{adjustbox}

\usepackage{longtable}

\usepackage{array}
\usepackage{booktabs}
\newcommand{\footremember}[2]{%
    \footnote{#2}
    \newcounter{#1}
    \setcounter{#1}{\value{footnote}}%
}
\newcommand{\footrecall}[1]{%
    \footnotemark[\value{#1}]%
} 

\def\x{\mathbf{x}}

\def\a{\mathbf{a}}

\def\R{{\mathbb R}}

\def\N{{\mathbb N}}

\DeclareMathOperator{\conv}{conv}

\DeclareMathOperator{\diag}{diag}

\DeclareMathOperator{\supp}{supp}
\DeclareMathOperator{\trace}{trace}

\DeclareMathOperator{\s}{s}

\DeclareMathOperator{\minimal}{minimal}

\newcommand{\argmin}{\arg\min}

\newtheorem{theorem}{\bf Theorem}[section]
\newtheorem{lemma}[theorem]{\bf Lemma}
\newtheorem{example}[theorem]{\bf Example}
\newtheorem{proposition}[theorem]{\bf Proposition}
\newtheorem{corollary}[theorem]{\bf Corollary}
\newtheorem{definition}[theorem]{\bf Definition}
\newtheorem{remark}[theorem]{\bf Remark}
\newtheorem{assumption}[theorem]{\bf Assumption}

\providecommand{\keywords}[1]
{
  \small
  \textbf{\textbf{Keywords:}} #1
}
\pgfplotsset{compat=1.13}
\begin{document}

\definecolor{qqzzff}{rgb}{0,0.6,1}
\definecolor{ududff}{rgb}{0.30196078431372547,0.30196078431372547,1}
\definecolor{xdxdff}{rgb}{0.49019607843137253,0.49019607843137253,1}
\definecolor{ffzzqq}{rgb}{1,0.6,0}
\definecolor{qqzzqq}{rgb}{0,0.6,0}
\definecolor{ffqqqq}{rgb}{1,0,0}
\definecolor{uuuuuu}{rgb}{0.26666666666666666,0.26666666666666666,0.26666666666666666}
\newcommand{\vi}[1]{\textcolor{blue}{#1}}
\newif\ifcomment
\commentfalse
\commenttrue
\newcommand{\comment}[3]{%
\ifcomment%
	{\color{#1}\bfseries\sffamily#3%
	}%
	\marginpar{\textcolor{#1}{\hspace{3em}\bfseries\sffamily #2}}%
	\else%
	\fi%
}
\newcommand{\victor}[1]{
	\comment{blue}{V}{#1}
}
\definecolor{oucrimsonred}{rgb}{0.6, 0.0, 0.0}
\newcommand{\jean}[1]{
	\comment{oucrimsonred}{J}{#1}
}
\definecolor{cadmiumgreen}{rgb}{0.0, 0.42, 0.24}
\newcommand{\hoang}[1]{
	\comment{cadmiumgreen}{H}{#1}
}

\title{Exploiting constant trace property in large-scale polynomial optimization}
\author{%
Ngoc Hoang Anh Mai\footremember{1}{CNRS; LAAS; 7 avenue du Colonel Roche, F-31400 Toulouse; France.} \and %
   Jean-Bernard Lasserre\footrecall{1} \footremember{2}{Institute of Mathematics, Universit\'e de Toulouse; F-31400 Toulouse, France.} \and  %
   Victor Magron\footrecall{1} \footrecall{2}   \and  %
   Jie Wang\footrecall{1}
  }
\maketitle
\begin{abstract}
\small
We prove that every semidefinite moment relaxation of a polynomial optimization problem (POP) with a ball constraint can be reformulated as a semidefinite program involving a matrix with constant trace property (CTP).
As a result such moment relaxations can be solved efficiently by first-order methods that exploit CTP, e.g., the conditional gradient-based augmented Lagrangian method.
We also extend this CTP-exploiting framework to large-scale POPs with different sparsity structures. The efficiency and scalability of our framework are illustrated 
on second-order moment relaxations for various randomly generated quadratically constrained quadratic programs.
\end{abstract}
{\small \tableofcontents}
\section{Introduction}
This paper is in the line of recent efforts to promote first-order methods as a viable alternative to interior-point methods (IPM) for solving large-scale conic optimization problems, in particular large-scale semidefinite programming (SDP) relaxations of polynomial optimization problems (POPs).
We show that a wide class of POPs have a nice property, namely the constant trace property (CTP), and that this property can be exploited in combination with  first-order methods to solve large-scale SDP relaxations associated with a POP.
So far, this property has been exploited only in a few cases, the most prominent examples being the Shor's  relaxation of Max-Cut \cite{yurtsever2019scalable}, in which the authors are able to handle SDP matrices of huge size, and equality constrained POPs on the sphere  \cite{mai2020hierarchy}.

Given polynomials $f,g_i,h_j$, let us consider the following POP with $n$ variables, $m$ inequality constraints and $l$ equality constraints:
\begin{equation}\label{eq:POP.def.intro}
    f^\star:=\min\{f(\mathbf x)\,:\quad g_i(\mathbf x)\ge 0\,,\,i\in[m]\,,\quad h_j(\mathbf x)= 0\,,\,j\in[l]\}\,,
\end{equation}
where $[m]:=\{1,\dots,m\}$ and $[l]:=\{1,\dots,l\}$. 
In general POP \eqref{eq:POP.def.intro} is non-convex, NP-hard. It is well known that under some mild condition, the optimal value $f^\star$ of POP \eqref{eq:POP.def.intro} can be approximated as closely as desired by the so-called Moment-Sums of squares (Moment-SOS) hierarchy \cite{lasserre2001global}. There are a lot of important applications of POP \eqref{eq:POP.def.intro} as well as the Moment-SOS hierarchy; the interested readers are referred to the monograph \cite{henrion2020}.

\paragraph{Computational cost of moment relaxations.}
The $k$-th order moment relaxation for POP \eqref{eq:POP.def.intro}
can be rewritten in compact form as the following standard SDP:
\begin{equation}\label{eq:SDP.form.intro}
\tau = \inf _{\mathbf X \in \mathcal S^+} \{ \left< \mathbf C,\mathbf X\right>\,:\,\left< \mathbf A_j, \mathbf X\right>= b_j\,,\,j\in [\zeta]\}\,,
\end{equation}
where $\mathcal{S}^+$ is the set of positive semidefinite (psd) matrices in a block diagonal form:
$\mathbf X=\diag(\mathbf X_1,\dots,\mathbf X_\omega)$
with $\mathbf X_j$ being a block of size $s^{(j)}$, $j\in [\omega]$ and $\zeta$ is the number of affine constraints.
We denote the largest block size by $s^{\max}:=\max_{j\in[\omega]}s^{(j)}$.

We say that SDP \eqref{eq:SDP.form.intro} has \emph{constant trace property} (CTP) if there exists a positive real number $a$ such that $\trace(\mathbf X) = a$, for all feasible solution $\mathbf X$ of SDP \eqref{eq:SDP.form.intro}.
We also say that POP \eqref{eq:POP.def.intro} has CTP when every moment relaxation of POP \eqref{eq:POP.def.intro} has CTP.

Table \ref{tab:comparison.SDP.solver} lists several available methods for solving SDP \eqref{eq:SDP.form.intro}.
In particular, observe that two of them, CGAL and SBM, are first-order methods that exploit CTP. 
In \cite{yurtsever2019scalable}, the authors combined CGAL with the Nystr\"om sketch (named SketchyCGAL), which require dramatically less storage than other methods and is very efficient for solving Shor's relaxation of large-scale MAX-CUT instances.

\begin{table}
    \caption{\small Complexity comparison of several methods for solving SDP. IP: interior point methods; ADMM: the alternating direction method of multipliers; SBM: spectral bundle methods; CGAL: conditional gradient-based augmented Lagrangian.}
    \label{tab:comparison.SDP.solver}
\small
\begin{center}
\begin{tabular}{|m{1.6cm}|m{1.8cm}|m{1.2cm}|m{1.8cm}|m{3.8cm}|}
\hline
Method & Software & SDP type & Convergence rate & The most expensive parts per iteration\\
\hline
IP \cite{helmberg1996interior}

(second-order)& Mosek \cite{mosek}& Arbitrary & $\mathcal{O}(\log(1/\varepsilon))$   \cite{vandenberghe1996semidefinite} &  System of linear equations solving with $\mathcal{O}((s^{\max})^6)$ \cite[Table 1]{vandenberghe2005interior}\\
\hline
ADMM \cite{boyd2011distributed} 

(first-order)& SCS\cite{o2016conic}, COSMO\cite{garstka2019cosmo} & Arbitrary &  $\mathcal{O}(\varepsilon^{-1})$  \cite{hong2017linear}& Positive definite system  of linear equations solving by $LDL^\top$-decomposition with $\mathcal{O}((s^{\max})^6)$\\

\hline
SBM \cite{helmberg2000spectral} 

(first-order)& ConicBundle \cite{helmberg2014spectral} & with CTP & $\mathcal{O}({\log(1/\varepsilon)}/{\varepsilon})$   \cite{ding2020revisit} & 
Positive definite linear system solving with $\mathcal{O}((s^{\max})^6)$\\
\hline
CGAL \cite{yurtsever2019conditional} 

(first-order) & SketchyCGAL \cite{yurtsever2019scalable}& with CTP & $\mathcal{O}(\varepsilon^{-1/2})$ & Smallest eigenvalue computing by the Arnoldi iteration with $\mathcal{O}(s^{\max})$ \cite{lee2009k}\\
\hline
\end{tabular}    
\end{center}
\end{table}

Note that SDP-relaxation \eqref{eq:SDP.form.intro} of POP \eqref{eq:POP.def.intro} at step $k$ of the Moment-SOS hierarchy has $\omega=m+1$ blocks whose largest  size  is $s^{\max}=\binom{n+k}{n}$ while the number of affine constraints is  $\zeta=\mathcal{O}(\binom{n+k}{n}^2)$.
Thus the computational cost for solving SDP \eqref{eq:SDP.form.intro} grows very rapidly with $k$.
Fortunately, it is usually possible to reduce the size of this SDP relaxation by exploiting certain structures of POP \eqref{eq:POP.def.intro}. Table \ref{tab:exploting.sparsity} lists some of these structures.

\begin{itemize}
\item Correlative sparsity (CS), term sparsity (TS) and their combination (CS-TS) are applied to POPs \eqref{eq:POP.def.intro} in case that the data $f,g_i,h_j$ are sparse polynomials.
The main idea of CS, TS and CS-TS is to break the moment matrices and localizing matrices (which are the psd matrices in the Moment-SOS relaxation) into a lot of blocks according to certain sparsity patterns derived from the POP. If the largest block size is relatively small (say $s^{\max}\le 100$), then the corresponding SDP can be solved efficiently. But if the largest block size is still large (say $s^{\max}\ge 200$), then the corresponding SDP remains hard to solve.

\item In the previous work \cite{mai2020hierarchy}, the first three
authors exploited CTP for equality constrained POPs on the sphere and converted the resulting SDP relaxations to spectral minimization problems which could be solved by LMBM efficiently.
This method returns approximate optimal values of SDP relaxations involving 
$2000\times 2000$ matrices for which Mosek encounters memory issues and SketchyCGAL is much less efficient. Importantly,
the moment SDP-relaxation of an equality constrained POP has a \emph{single} psd matrix. 
In contrast, for a POP involving a ball constraint (with possibly other inequality constraints), the resulting moment SDP-relaxations include several psd matrices.
{Unfortunately for such SDPs, LMBM usually returns inaccurate values even when CTP holds because of ill-conditioning issues. 
LMBM only updates the dual variables, so it is hard to ensure that the KKT conditions hold.}
{We can overcome the latter ill-conditioning issues by relying on a primal-dual algorithm such as CGAL.
It turns out that CGAL (without sketching) is suitable for this type of SDP.}
{For an SDP involving a single matrix, SketchyCGAL stores updated matrices by means of Nystr\"om sketch. In our experimental setting, we rather consider CGAL without sketching, which boils down to relying on implicit updated matrices.
It turns out that this strategy is much faster than the one based on Nystr\"om sketch, but does not provide the primal (matrix) solution.}

\end{itemize}

\begin{table}
    \caption{\small Several special structures for reducing complexity of the Moment-SOS relaxations.}
    \label{tab:exploting.sparsity}
\small
\begin{center}
\begin{tabular}{|m{1.6cm}|m{2cm}|m{7.3cm}|}
\hline
Structure & Software & POP type \\
\hline
 CS \cite{Waki06SparseSOS, lasserre2006convergent}& SparsePOP  \cite{waki2008algorithm}& $f=\sum_{j\in[p]} f_j$ 
and $f_j, (g_i)_{i\in J_j}, (h_i)_{i\in W_j}$ share the same variables for every $j\in[p]$ and $p>1$\\
\hline
TS \cite{wang2019tssos, wang2019second}& TSSOS & $f,g_i, h_j$ involve a few of terms\\
\hline
CS-TS \cite{wang2020cs}& TSSOS & Both CS and TS hold\\
\hline
CTP\cite{mai2020hierarchy} & SpectralPOP & Equality constrained POPs on a sphere ($m=0$ and $h_1:=R-\|\mathbf x\|_2^2$) 
\\
\hline
\end{tabular}    
\end{center}
\end{table}

\paragraph{SDP relaxations of non-convex quadratically constrained quadratic programs.}  A non-convex quadratically constrained quadratic (QCQP) program is a special instance of POP \eqref{eq:POP.def.intro} for which the degrees of the input polynomials are at most two. Famous instances of non-convex QCQPs include the MAX-CUT problem and the optimal power flow (OPF) problem \cite{josz2018lasserre}; in addition we recall that that LCQPs have an equivalent MAX-CUT formulation \cite{Lasserre-Maxcut}.
They also have applications in deep learning, e.g., the computation of Lipschitz constants \cite{chen2020semialgebraic}
and the stability analysis of recurrent neural networks
\cite{ebihara2020l_2}.
In practice, non-convex QCQPs usually involve a large number of variables (say $n\ge 1000$) and
their associated SDP relaxations \eqref{eq:SDP.form.intro} can be classified in two groups as follows:
\begin{itemize}
    \item \textbf{The first order relaxation}: $k=1$ (also known as Shor's relaxation in the literature).
    In this case the number of affine constraints in SDP \eqref{eq:SDP.form.intro} is typically not larger than the largest block size, i.e., $\zeta \le  s^{\max}$. 
    It can be efficiently solved by most SDP solvers, in particular with SketchyCGAL \cite{yurtsever2019scalable}.  
    Nevertheless the first order relaxation may only provide a lower bound for the optimal value of POP \eqref{eq:POP.def.intro}. 
    In this case, one needs to solve the second and perhaps even higher-order relaxations to obtain tighter bounds or achieve the global optimal value.
    \item \textbf{The second and higher-order relaxations}: $k\ge 2$. In this case the number of affine constraints in SDP \eqref{eq:SDP.form.intro} is typically much larger than the largest block size ($\zeta\gg s^{\max}$). Then unfortunately most SDP solvers cannot handle large-scale SDPs of this form. 
    In our previous work \cite{mai2020hierarchy}, we proposed a remedy for the particular case of second-order SDP relaxations of equality constrained POPs on the sphere, by relying on first-order solvers such as LMBM.
    \end{itemize}
    \paragraph{Common issues of solving large-scale SDP relaxations.} 
    When solving the second and higher-order SDP relaxations, SDP solvers often encounter the following issues:
    \begin{itemize}
        \item\textbf{Storage}: The interior-point methods (IPM) are often chosen by users because of their highly accurate output. 
        These methods are efficient for solving medium-scale SDPs.
        However they frequently fail due to lack of memory when solving large-scale SDPs (say $s^{\max}>500$ and $\zeta>2\times10^5$ on a standard laptop). Then first-order methods (e.g., ADMM, SBM, CGAL) provide an alternative to IPM to avoid the memory issue.
        This is due to the fact that the cost per iteration of first-order methods is much cheaper than that of IPM. 
        
        At the price of losing convexity one can also rely on heuristic methods and  replace the full matrix $\mathbf X$ in SDP \eqref{eq:SDP.form.intro} by a simpler one, in order to save memory. For instance, the Burer-Monteiro method \cite{burer2005local} considers a low rank factorization of $\mathbf X$. However, to get correct results the rank cannot be too low \cite{waldspurger2020rank} and therefore this limitation makes it useless for the second and higher-order relaxations of POPs.
        Not suffering from such a limitation, CGAL not only maintains the convexity of SDP \eqref{eq:SDP.form.intro} but also possibly runs with implicit matrix $\mathbf X$ as described in Remarks \ref{re:implicit} and \ref{re:implicit.block}.
        \item \textbf{Accuracy}: Nevertheless, first-order methods have low convergence rates compared to the interior-point methods.
        Their performance depends heavily on the problem scaling and conditioning.
        As a result, in solving large-scale SDPs with first-order methods it is often difficult to obtain results with high accuracy. In contrast the relative gap of the value returned by first-order SDP solvers w.r.t. the exact value is usually expected to be less than 1\%.
    \end{itemize}
    
The goal of this paper is to provide a method which returns the optimal value of the second-order moment SDP-relaxation and which is suitable for a class of large-scale non-convex QCQPs with CTP. 
Ideally (i) it should avoid the memory issue, and (ii) the resulting relative gap of the approximate value returned by this method w.r.t. the exact value, should be less than 1\%.

\paragraph{Contribution.}
We show that (i) (a large class of) POPs have a very nice \emph{constant trace property} and (ii) that this property can be exploited for solving their associated semidefinite relaxations via appropriate first-order methods. More precisely our contribution is threefold:
\begin{enumerate}
    \item In Section \ref{sec:sufficient.CTP} we show that if a positive real number belongs to the interior of every truncated quadratic module associated to the inequality constraints, then the corresponding POP has CTP.
    Moreover, we prove that this condition always holds when a ball constraint is present.
    \item In Section \ref{sec:obtain.CTP.SDP} we provide  a linear programming approach to check whether a POP has CTP. With this approach we prove in Section \ref{sec:special.POP.CTP} that several special classes of POPs (including POPs on a ball, annulus, simplex) have CTP.
    \item Our final contribution is to handle sparse large-scale POPs by integrating sparsity-exploiting techniques into the CTP-exploiting framework.
\end{enumerate}
For practical implementation we have provided
a software library called ctpPOP.
It consists of modeling each moment SDP-relaxation of POPs as a standard SDP with CTP and then solving this SDP by CGAL or the spectral method (SM) with nonsmooth optimization solvers (LMBM or PBM).

In Section \ref{sec:benchmark} we provide extensive numerical experiments to illustrate the efficiency and scalability of ctpPOP with the CGAL solver. In all our randomly generated POPs with different sparsity structures, 
the relative gap of the optimal value provided by CGAL w.r.t. the optimal value provided by Mosek is below 1\%.
Because of its very cheap cost per iteration, CGAL is more suitable for particularly bulky SDPs (such as moment SDP-relaxations of POPs) than other solvers (e.g. COSMO).

For instance for  minimizing a \emph{dense} quadratic polynomial on the unit ball
with up to $100$ variables, CGAL returns the optimal value of the second-order moment SDP relaxation within $6$ hours on a standard laptop while Mosek (considered  state-of-the-art IPM SDP solver) runs out of memory. Similarly, for minimizing a \emph{sparse} quadratic polynomial involving thousand variables, with  a ball constraint on each clique of variables, CGAL spends around two thousand seconds to solve the second-order moment SDP-relaxation while Mosek runs again out of memory. The largest clique of this POP involves $41$ variables. 

Classical Optimal Power Flow (OPF) problem without constraints on current magnitudes (as in \cite{josz2016ac,godard2019novel}) can be formulated as a POP with ball and annulus constraints. In many instances Shor's relaxation usually provides the global optimum.
However, for illustration purpose we have compared  CGAL and Mosek for solving the 
second-order CS-TS relaxation for one instance ``case89\_pegase\_\_api" from the PGLib-OPF database\footnote{https://github.com/power-grid-lib/pglib-opf}.
The largest block size and the number of equality constraints of this SDP are around 1.7 thousand and 8 million, respectively.
While Mosek failed because of memory issue, CGAL still returns the optimal value in 2 days, and with the relative gap w.r.t. a local optimal value less than 0.6\%. 


\section{Notation and preliminary results}

With $\x = (x_1,\dots,x_n)$, let $\R[\x]$ stand for the ring of real polynomials and let $\Sigma[\x]\subseteq\R[\x]$ be the subset of sum of squares (SOS) polynomials.
Their restrictions to polynomials of degree at most $d$ and $2d$ are denoted by $\R[\x]_d$ and $\Sigma[\x]_d$ respectively. 
For $\alpha = (\alpha_1,\dots,\alpha_n) \in \N^n$, let $|\alpha|:= \alpha_1 + \dots + \alpha_n$.
Let $\N^n_d:=\{\alpha\in\N^n : |\alpha|\le d\}$.
Let $(\x^{ \alpha})_{\alpha\in\N^n}$ 
be the canonical monomial basis of $\R[\x]$ (sorted w.r.t. the graded lexicographic order) and 
$\mathbf v_d(\x)$ be the vector of monomials of degree up to $d$, with length $\s(d,n) := \binom{n+d}{n}$. 
when it is clear from the context, we also write $\s(d)$ instead of $\s(d,n)$.
A polynomial $p\in\R[\x]_d$ can be written as  
$p(\x)\,=\,\sum_{\alpha\in\N^n_d} p_\alpha\,\x^\alpha\,=\,\mathbf{p}^\top\mathbf  v_d(\x)$, 
where $\mathbf{p}=(p_\alpha)\in\R^{\\s(d)}$ is the vector of coefficients in the canonical monomial basis. 
For $p\in\R[\x]$, let $\lceil p\rceil:=\lceil{\rm deg}(p)/2\rceil$. For a positive integer $m$, let $[m]:=\{1,2,\ldots,m\}$.
The $l_1$-norm of a polynomial $p$ is given by the $l_1$-norm of the vector of coefficients $\mathbf{p}$, that is $\|\mathbf{p}\|_1 := \sum_{\alpha} |p_\alpha|$. Given $\a\in\R^n$, the $l_2$-norm of $\a$ is  $\|\a\|_2:=(a_1^2+\dots+a_n^2)^{1/2}$ and the maximum norm of $\a$ is $\|\a\|_{\infty}:=\max\{|a_j|:j\in[n]\}$. Given a subset $\mathcal{S}$ of real symmetric matrices, let $\mathcal{S}^+:=\{\mathbf X\in \mathcal{S}\,:\, \mathbf X\succeq 0\}$. For $I\subseteq [n]$, let $\mathbf x(I):=\{x_j:j\in I\}$ and $\N^I_d:=\{\alpha\in\N^n_d:\supp(\alpha)\subseteq I\}$.
 
\paragraph{Polynomial optimization problem.}
A polynomial optimization problem (POP) is defined as
\begin{equation}\label{eq:POP.def}
f^\star:=\inf \{f(\mathbf x)\ :\ \mathbf x\in S(g)\cap V(h)\}\,,
\end{equation}
where $S(g)$ and $V(h)$ are a basic semialgebraic set and a real variety defined respectively by:
\begin{eqnarray}
    \nonumber
    S(g) &:=&\{\,\mathbf x\in\R^n:\: g_i(\mathbf x)\ge  0\,,\,i\in [m]\,\}\\
    \label{eq:semialg.set.real.variety}
    V(h) &:=&\{\,\mathbf x\in\R^n:\: h_j(\mathbf x)=0\,,\,i\in [l]\,\}\,,
\end{eqnarray}
for some polynomials $f,g_i,h_j\in\R[\mathbf x]$ with $g:=\{g_i\}_{i\in[m]}$, $h:=\{h_j\}_{j\in[l]}$. We will assume that POP \eqref{eq:POP.def} has at least one global minimizer.

\paragraph{Riesz linear functional.} Given  a real-valued sequence $\mathbf y=(y_\alpha)_{\alpha\in\N^n}$, define the 
Riesz linear functional $L_{\mathbf y}:\R[ \mathbf x ] \to \R$, $f\mapsto {L_{\mathbf y}}( f ) := \sum_{\alpha} f_\alpha y_\alpha$.  
Let $d$ be a positive integer. A real infinite (resp. finite) sequence $( y_\alpha)_{\alpha\in \N^n}$ (resp. $( y_\alpha)_{\alpha  \in \N^n_d}$) has a \emph{representing measure} if there exists a finite Borel measure $\mu$ such that $y_\alpha  = \int_{\R^n} {\mathbf x^\alpha d\mu(\mathbf x)}$ for every $\alpha  \in {\N^n}$ (resp. $\alpha  \in {\N^n_d}$). 
In this case, $( y_\alpha)_{\alpha  \in \N^n}$ is called be the moment sequence of $\mu$. 
We denote by $\supp(\mu)$ the support of a Borel measure $\mu$.

\paragraph{Moment/Localizing matrix.} The moment matrix of order $d$ associated with a real-valued sequence $\mathbf y=(y_\alpha)_{\alpha  \in \N^n}$ and $d\in \N^{>0}$,
is the real symmetric matrix $\mathbf M_d(\mathbf y)$ of size $s(d)$,  with entries
$( y_{\alpha  + \beta })_{\alpha,\beta\in \N^n_d} $. 
The localizing matrix of order $d$ associated with $\mathbf y=(y_\alpha)_{\alpha  \in \N^n}$ and $p = \sum_{\gamma} p_\gamma x^\gamma  \in \R[\mathbf x]$, is the real symmetric matrix $\mathbf M_d(p\,\mathbf y)$ of size $s(d)$ 
with entries $(\sum_\gamma  {{p_\gamma }{y_{\gamma  + \alpha  + \beta }}})_{\alpha, \beta\in \N^n_d}$.

\paragraph{Quadratic module.}
Given $g=\{g_i:i\in[m]\}\subseteq\R[\x]$, the \emph{quadratic module} associated with $g$ is defined by
$Q(g): = \{ \sigma_0+\sum_{i \in[m]}\sigma_ig_i  \ :\ \sigma_0 \in \Sigma[ \mathbf x]\,,\, \sigma_i \in \Sigma[ \mathbf x]\}$, and for a positive integer $k$, the set
$Q_k(g): = \{ \sigma_0+\sum_{i \in[m]}\sigma_ig_i  \,:\,
      \sigma_0 \in \Sigma[ \mathbf x]_k\,,\,
     \sigma_i \in \Sigma[ \mathbf x]_{k-\lceil g_i \rceil}\}$
is the truncation of $Q(g)$ of order $k$. 
\paragraph{Ideal.}
Given $h=\{h_i:i\in[l]\}\subseteq\R[\x]$, the set
$I(h): = \{ \sum_{j \in[l]}\psi_jh_j  \ :\  \psi_j \in \R[ \mathbf x]\}$
is the \emph{ideal} generated by $h$, and the set $I_k(h): = \{ \sum_{j \in[l]}\psi_jh_j   \,:\,
     \psi_j \in \R[ \mathbf x]_{2(k-\lceil h_j \rceil)}\}$
is the truncation of $I(h)$ of order $k$.

\paragraph{Archimedeanity.} 
Assume that there exists $R>0$ such that $R-\|\mathbf x\|_2^2\in Q(g)+I(h)$. 
As a consequence, $S(g)\cap V(h)\subseteq \mathcal B_R$, where $\mathcal{B}_R:=\{\mathbf x\in\R^n\,:\,\|\mathbf x\|_2\le \sqrt R\}$. 
In this case, we say that $Q(g)+I(h)$ is  \emph{Archimedean} \cite{lasserre2010moments}. 

\paragraph{The Moment-SOS hierarchy \cite{lasserre2001global}.}

Given a POP \eqref{eq:POP.def}, consider the following associated hierarchy of SOS relaxations indexed by $k\in\N^{\ge k_{\min}}$ with $k_{\min}:=\max\{\lceil f\rceil, \{\lceil g_i\rceil\}_{i\in[m]},\{\lceil h_j\rceil\}_{j\in[l]}\}$:
\begin{equation}\label{eq:sos.hierarchy.0}
\rho_k\,:=\,\sup \{\,\xi \in\R\ :\ f-\xi \in Q_k(g)+I_k(h)\}\,.
\end{equation}
For each $\sigma\in\Sigma[x]_d$, there exists $\mathbf G\succeq 0$ such that $\sigma=\mathbf v_d^\top\mathbf G\mathbf v_d$.
Thus for each $k\in\N^{\ge k_{\min}}$, \eqref{eq:sos.hierarchy.0} can be rewritten as an SDP:
\begin{equation}\label{eq:sos.hierarchy}
\rho_k = \sup \limits_{\xi,\mathbf G_i,\mathbf u_j} \left\{\xi\ \left|
\begin{array}{rl}
&\mathbf G_i \succeq 0\,,\,
f-\xi=\mathbf v_k^\top \mathbf G_0 \mathbf v_k\\
&\quad\quad\quad\qquad\qquad+\sum_{i\in[m]} g_i \mathbf v_{k-\lceil g_i\rceil}^\top \mathbf G_i \mathbf v_{k-\lceil g_i\rceil}\\
&\quad\quad\quad\qquad\qquad+ \sum_{j\in[l]} h_j \mathbf v_{2(k-\lceil h_j\rceil)}^\top \mathbf u_j
\end{array}\right. \right\}\,.
\end{equation}

For every $k\in \N^{\ge k_{\min}}$, the dual of \eqref{eq:sos.hierarchy} reads as
\begin{equation}\label{eq:moment.hierarchy}
\tau_k \,:= \,\inf \limits_{\mathbf y \in {\R^{\s({2k})} }} \left\{ L_{\mathbf y}(f)\ \left|\begin{array}{rl}
& \mathbf M_k(\mathbf y) \succeq 0\,,\:y_{\mathbf 0}\,=\,1\\
&\mathbf M_{k - \lceil g_i \rceil }(g_i\;\mathbf y)   \succeq 0\,,\,i\in[m]\\
&\mathbf M_{k - \lceil h_j \rceil }(h_j\;\mathbf y)   = 0\,,\,j\in[l]
\end{array}
\right. \right\}\,.
\end{equation}

If $Q(g)+V(h)$ is Archimedean, then both $(\rho_k)_{k\in\N^{\ge k_{\min}}}$ and $(\tau_k)_{k\in\N^{\ge k_{\min}}}$ converge to $f^\star$. For details on the Moment-SOS hierarchy and its various applications the interested reader is referred to
\cite{lasserre2010moments}.

\section{Exploiting CTP for dense POPs}
\label{sec:ctp.dense.pop}
This section is devoted to developing a framework to exploit CTP for dense POPs. 
We provide a sufficient condition for a POP to have CTP, as well as 
a series of linear programs to check whether the sufficient condition holds. In addition we show that several special classes of POPs have CTP. 

\subsection{CTP for dense POPs}
\label{sec:spectral.relax}
First let us define CTP for a POP. To simplify notation, for every $k\in\N^{\ge k_{\min}}$, denote by $\mathcal{S}_k$ the set of real symmetric matrices 

- of size $s_k:=\s(k)+\sum_{i\in[m]}\s(k-\lceil g_i\rceil)$,

- in a block diagonal form 
$\mathbf X=\diag(\mathbf X_0,\dots,\mathbf X_m)$, and such that

- $\mathbf X_0$ (resp. $\mathbf X_i$) is of size $\s(k)$ (resp. $\s(k-\lceil g_i\rceil)$ for $i\in [m]$).\\

Letting $\mathbf D_k(\mathbf y):=\diag(\mathbf M_k(\mathbf y),\mathbf M_{k-\lceil g_1\rceil}(g_1\mathbf y),\dots, \mathbf M_{k-\lceil g_m\rceil}(g_m\mathbf y))$, 
SDP \eqref{eq:moment.hierarchy} can be rewritten in the form:
\begin{equation}\label{eq:dual.diag.moment.mat}
\tau_k \,:= \,\inf\limits_{\mathbf y \in \R^{\s(2k)} }\left\{  L_{\mathbf y}(f)\ \left|\begin{array}{rl}
&\mathbf D_k(\mathbf y) \in \mathcal S_k^+\,,\,y_{\mathbf 0}=1\,,\\
&\mathbf M_{k-\lceil h_i\rceil}(h_i\mathbf y)=0\,,\,i\in[l]
\end{array}
\right. \right\}\,.
\end{equation}
\begin{definition}\label{def:ctp}
(CTP for a POP)
We say that POP \eqref{eq:POP.def} has CTP if for every $k\in \N^{\ge k_{\min}}$, there exists $a_k>0$ and a positive definite matrix $\mathbf P_{k}\in \mathcal{S}_k$ such that for all $\mathbf y \in \R^{\s(2k)}$,
\begin{equation}
    \left.
\begin{array}{rl}
&\mathbf M_{k-\lceil h_i\rceil}(h_i\mathbf y)=0\,,\,i\in[l]\,,\\
&y_{\mathbf 0}=1
\end{array}
\right\}\Rightarrow  \trace(\mathbf P_{k} \mathbf D_k(\mathbf y) \mathbf P_{k})=a_k\,.
\end{equation}
\end{definition}
In other words, we say that POP \eqref{eq:POP.def} has CTP if each moment relaxation \eqref{eq:dual.diag.moment.mat} has an equivalent form involving a psd matrix whose trace is constant. In this case,
we call $a_k$ the constant trace and $\mathbf P_k$ the basis transformation matrix. In the next subsection, we provide a sufficient condition for POP \eqref{eq:POP.def} to have CTP. 

\begin{example} (CTP for equality constrained POPs on a sphere \cite{mai2020hierarchy})  
If $g=\emptyset$ and $h_1=R-\|\mathbf x\|_2^2$ for some $R>0$, then POP \eqref{eq:POP.def} has CTP with $a_k=(R+1)^k$ and $\mathbf P_{k}:=\diag((\theta^{1/2}_{k,\alpha})_{\alpha\in\N^n_k})$,
where $(\theta_{k,\alpha})_{\alpha\in\N^n_k}\subseteq \R^{>0}$ satisfies $(1+\|\mathbf
x\|_2^2)^k=\sum_{\alpha\in\N^n_k}\theta_{k,\alpha}\mathbf x^{2\alpha}$, for all $k\in\N^{\ge k_{\min}}$.
\end{example}

We now provide a general method to solve a POP with CTP. We first convert the $k$-th order moment relaxation \eqref{eq:dual.diag.moment.mat} of this POP to a standard primal SDP problem with CTP and then leverage appropriate first-order algorithms that exploit CTP to solve the resulting SDP problem.

Suppose POP \eqref{eq:POP.def} has CTP. 
For every $k\in\N^{\ge k_{\min}}$, letting $\mathbf X=\mathbf P_{k}\mathbf D_k(\mathbf y)\mathbf P_{k}$,
\eqref{eq:dual.diag.moment.mat} can be rewritten as
\begin{equation}\label{eq:SDP.form}
\tau_k = \inf _{\mathbf X\in \mathcal{S}_k^+} \{ \left< \mathbf C_k,\mathbf X\right>\,:\,\mathcal{A}_k \mathbf X=\mathbf b_k\}\,,
\end{equation}
where $\mathcal{A}_k:\mathcal{S}_k\to \R^{\zeta_k}$ is a linear operator such that
$\mathcal{A}_k\mathbf X=(\left< \mathbf A_{k,1},\mathbf X\right>,\dots,\left< \mathbf A_{k,\zeta_k},\mathbf X\right>)$
with $\mathbf A_{k,i} \in \mathcal{S}_k$, $i\in[\zeta_k]$, $\mathbf C_k \in \mathcal{S}_k$ and $\mathbf b_k\in \R^{\zeta_k}$. 
Appendix \ref{sec:convert.standart.SDP} describes how to convert SDP \eqref{eq:dual.diag.moment.mat} to the form \eqref{eq:SDP.form}.

The dual of SDP \eqref{eq:SDP.form} reads as 
\begin{equation}\label{eq:SDP.form.dual}
\rho_k = \sup _{\mathbf z\in\R^{\zeta_k}} \,\{ \,\mathbf b_k^\top\mathbf z\,:\,
\mathcal{A}_k^\top \mathbf z-\mathbf C_k\in \mathcal S_k^+\}\,,
\end{equation}
where $\mathcal{A}_k^\top:\R^{\zeta_k}\to \mathcal{S}_k$ is the adjoint operator of $\mathcal{A}_k$, i.e., $\mathcal{A}_k^\top\mathbf z=\sum_{i\in[\zeta_k]} z_i\mathbf A_{k,i}$.




After replacing $(\mathcal{A}_k, \mathbf{A}_{k,i},  \mathbf b_k, \mathbf C_k, \mathcal{S}_k, \zeta_k, s_k, \tau_k, \rho_k, a_k)$ by $(\mathcal{A}, \mathbf{A}_{i}, \mathbf b, \mathbf C, \mathcal{S}, \zeta, s, \tau, \rho, a)$, the primal-dual \eqref{eq:SDP.form}-\eqref{eq:SDP.form.dual} has an equivalent formulation as the primal-dual \eqref{eq:SDP.form.0}-\eqref{eq:SDP.form.dual.0}; see also Appendix \ref{sec:sdp.ctp.dense} with $\omega=m+1$ and $s^{\max}=s(k)$.

Then two first-order algorithms (CGAL and SM) are leveraged for solving the primal-dual \eqref{eq:SDP.form.0}-\eqref{eq:SDP.form.dual.0}; see Appendix \ref{sec:sdp.ctp.dense} and Appendix \ref{sec:spectral.method.dense}. 

\subsection{A sufficient condition for a POP to have CTP}
\label{sec:sufficient.CTP}
In this section, we provide a sufficient condition for POP \eqref{eq:POP.def} to have CTP.

For $k\in\N^{\ge k_{\min}}$, let $Q_k^\circ(g)$ be the interior of the truncated quadratic module $Q_k(g)$, i.e.,
$Q_k^\circ(g):=\{\mathbf v_k^\top \mathbf G_0 \mathbf v_k+\sum_{i\in[m]} g_i \mathbf v_{k-\lceil g_i\rceil}^\top \mathbf G_i \mathbf v_{k-\lceil g_i\rceil} \,:\, \mathbf G_i\succ 0, \quad i\in\{0\}\cup[m]\}$.

\begin{theorem}\label{theo:suff.cond.CTP}
The following statements hold:
\begin{enumerate}
    \item If one the following equivalent conditions hold for all $k\in\N^{\ge k_{\min}}$:
    \begin{eqnarray}
        \nonumber
        \R^{>0}\subseteq Q_k^\circ(g)+I_k(h)  & \Leftrightarrow &  \forall \delta>0\,, \  \delta\in Q_k^\circ(g)+I_k(h)\\
        \label{eq:suffi.con.ideal}
        &\Leftrightarrow &   1\in Q_k^\circ(g)+I_k(h)\,,
    \end{eqnarray}
    then POP \eqref{eq:POP.def} has CTP, as in Definition \ref{def:ctp}.
    \item Assume that $h=\emptyset$ and $S(g)$ has nonempty interior. Then POP \eqref{eq:POP.def} has CTP if and only if 
    \begin{equation}\label{eq:pos.real.inter.quad}
        \R^{>0}\subseteq Q_k^\circ(g)\,,\, \forall k\in\N^{\ge k_{\min}}\,.
    \end{equation}
\end{enumerate}
\end{theorem}
The proof of Theorem \ref{theo:suff.cond.CTP} is postponed to Appendix \ref{proof:suff.cond.CTP}.


The following lemma will be used later on.
\begin{lemma}\label{lem:equality} 
Let $R>0$. For all $k\in\N^{\ge 1}$, one has
\begin{equation}
    (R+1)^k=(1+\|\mathbf x\|^2_2)^k+(R-\|\mathbf x\|^2_2)\sum_{j=0}^{k-1}(R+1)^j(1+\|\mathbf x\|^2_2)^{k-j-1}\,.
\end{equation}
\end{lemma}
\begin{proof}
Let $k\in\N^{\ge 1}$. 
Letting $a=R+1$ and $b=1+\|\mathbf x\|^2_2$, the desired equality follows from $a^k-b^k=(a-b)\sum_{j=0}^{k-1}a^jb^{k-1-j}$.
\end{proof}

The next result states that the sufficient condition in Theorem \ref{theo:suff.cond.CTP} holds whenever a ball constraint is present in the POP's description.
For a real symmetric matrix $\mathbf A$, denote the largest eigenvalue of $\mathbf A$ by $\lambda_{\max}(\mathbf A)$.
\begin{theorem}\label{theo:generic.ctp}
If $R-\|\mathbf x\|_2^2\in g$ for some $R>0$ then the inclusions \eqref{eq:pos.real.inter.quad} hold and therefore POP \eqref{eq:POP.def} has CTP.
\end{theorem}

\begin{proof}
Without loss of generality, set $g_m:=R-\|\mathbf x\|_2^2$ and let $k\in\N^{\ge k_{\min}}$ be fixed. 
By Lemma \ref{lem:equality}, $(R+1)^k=\Theta+g_m\Lambda$,
where $\Theta:=(1+\|\mathbf x\|^2_2)^k$ and $\Lambda:=\sum_{j=0}^{k-1}(R+1)^j(1+\|\mathbf x\|^2_2)^{k-j-1}$.
Note that:
\begin{itemize}
    \item $\Theta=\sum_{\alpha\in\N^n_{k}}\theta_{\alpha}\mathbf x^{2\alpha}=\mathbf v_k^\top \mathbf G_0\mathbf v_k$ for some $(\theta_{\alpha})_{\alpha\in\N^n_{k}} \subseteq \R^{>0}$;
    \item $\Lambda=\sum_{\alpha\in\N^n_{k-1}}\lambda_{\alpha}\mathbf x^{2\alpha}=\mathbf v_{k-1}^\top \mathbf G_m\mathbf v_{k-1}$ for some $(\lambda_{\alpha})_{\alpha\in\N^n_{k-1}} \subseteq \R^{>0}$.
\end{itemize}
Here $\mathbf G_0=\diag((\theta_{\alpha})_{\alpha\in\N^n_{k}})$ and $\mathbf G_m=\diag((\lambda_{\alpha})_{\alpha\in\N^n_{k}})$ are both positive definite.
Then we have
$(R+1)^k=\mathbf v_k^\top \mathbf G_0\mathbf v_k+g_m\mathbf v_{k-1}^\top \mathbf G_m\mathbf v_{k-1}$.
Denote by $\mathbf I_t$ the identity matrix of size $s(t)$ for $t\in\N$.

Let $\mathbf W$ be a real symmetric matrix such that
$\sum_{i\in[m-1]} g_i\mathbf v_{k-\lceil g_i \rceil}^\top\mathbf I_{k-\lceil g_i \rceil} \mathbf v_{k-\lceil g_i \rceil} = \mathbf v_k^\top \mathbf W\mathbf v_k$.
Since $\mathbf G_0\succ 0$, there exists $\delta>0$ such that $\mathbf G_0-\delta \mathbf W\succ 0$.
Indeed, 
\begin{equation}
    \mathbf G_0-\delta \mathbf W\succ 0\Leftrightarrow \mathbf I_k\succ\delta \mathbf G_0^{-1/2} \mathbf W \mathbf G_0^{-1/2}  \Leftrightarrow 1> \delta \lambda_{\max}(\mathbf G_0^{-1/2} \mathbf W \mathbf G_0^{-1/2})\,, 
\end{equation}
yielding the selection $\delta=1/(|\lambda_{\max}(\mathbf G_0^{-1/2} \mathbf W \mathbf G_0^{-1/2})|+1)$. Then
\begin{equation*}
   (R+1)^k=\mathbf v_k^\top (\mathbf G_0-\delta \mathbf W)\mathbf v_k+\delta \sum_{i\in[m-1]} g_i\mathbf v_{k-\lceil g_i \rceil}^\top\mathbf I_{k-\lceil g_i \rceil} \mathbf v_{k-\lceil g_i \rceil}+g_m\mathbf v_{k-1}^\top \mathbf G_m\mathbf v_{k-1}\,,
\end{equation*}
which implies $(R+1)^k\in Q^\circ_k(g)$, which in turn yields the desired conclusion.
\end{proof}
 
The next result is a consequence of Theorem \ref{theo:generic.ctp}. Its states 
that if a POP has a ball constraint then the corresponding Moment-SOS relaxations satisfy Slater's condition.
\begin{corollary}\label{coro:stric.fea.sol.sos}
Assume that $R-\|\mathbf x\|_2^2\in g$ for some $R>0$.  
Then Slater's condition holds for SDP \eqref{eq:sos.hierarchy} for all $k\ge k_{\min}$. As a consequence, strong duality holds for the primal-dual \eqref{eq:sos.hierarchy}-\eqref{eq:moment.hierarchy} for all $k\ge k_{\min}$.
\end{corollary}
\begin{proof}
It suffices to prove that SDP \eqref{eq:sos.hierarchy} has a strictly feasible solution for all $k\ge k_{\min}$. 
Let $k\ge k_{\min}$ be fixed. By \cite[Proposition 5.8]{marshall}, there exist $\sigma_0\in\Sigma[\x]_k$, $\sigma\in\Sigma[\x]_{k-1}$ and $\lambda\in\R$ such that $f+\lambda=\sigma_0+(R-\|\x\|_2^2)\sigma$.
Thus $f+\lambda\in Q_k(g)$.
By Theorem \ref{theo:generic.ctp}, $1\in Q_k^\circ(g)$ and therefore
$f+1+\lambda\in Q^\circ_k(g)$, which yields the desired conclusion.
\end{proof}

\begin{remark}\label{rem:direct.cons.trace}
From the proofs of Theorem \ref{theo:generic.ctp} and Theorem \ref{theo:suff.cond.CTP}, the constant trace $a_k$ and the basis transformation matrix $\mathbf P_k$ (Definition \ref{def:ctp}) can be taken as
\begin{equation*}
    a_k=(R+1)^k\quad\text{and}\quad \mathbf P_k=\diag((\mathbf G_0-\delta \mathbf W)^{1/2},\sqrt{\delta}\mathbf I_{k-\lceil g_1 \rceil},\dots,\sqrt{\delta}\mathbf I_{k-\lceil g_{m-1} \rceil},\mathbf G_m^{1/2})\,.
\end{equation*}
However, this choice has poor numerical properties. 
In the next section we provide a series of linear programs (LPs) inspired from the inclusions \eqref{eq:suffi.con.ideal}, to obtain the constant trace $a_k$ and the basis transformation matrix $\mathbf P_k$ which achieve a better numerical performance.
\end{remark}

\subsection{Verifying CTP for POPs by solving linear programs}
\label{sec:obtain.CTP.SDP}
For any $k\in\N^{\ge k_{\min}}$, let $\hat{\mathcal S}_{k}$ be the set of real diagonal matrices of size $\s(k)$ and consider the following LP:
\begin{equation}\label{eq:find.CTP.diag}
\inf \limits_{\xi,\mathbf G_i,\mathbf u_j} \left\{\xi\ \left|\begin{array}{rl}
&\mathbf G_0-\mathbf I_0 \in \hat{\mathcal S}_{k}^+\,,\,\mathbf G_i-\mathbf I_i \in \hat{\mathcal S}_{k-\lceil g_i\rceil}^+\,,\,i\in[m]\,,\\
&\xi=\mathbf v_k^\top \mathbf G_0 \mathbf v_k+\sum_{i\in[m]} g_i \mathbf v_{k-\lceil g_i\rceil}^\top \mathbf G_i \mathbf v_{k-\lceil g_i\rceil}\\
&\qquad\qquad\qquad\qquad+ \sum_{j\in[l]} h_j \mathbf v_{2(k-\lceil h_j\rceil)}^\top \mathbf u_j
\end{array}
\right.\right\}\,,
\end{equation}
where $\mathbf I_i$ is the identity matrix for $i\in\{0\}\cup [m]$.
\begin{lemma}\label{lem:feas.LP}
If LP \eqref{eq:find.CTP.diag} has a feasible solution $(\xi_k,\mathbf G_{i,k},\mathbf u_{j,k})$ for every $k\in\N^{\ge k_{\min}}$, then POP \eqref{eq:POP.def} has CTP with $a_k=\xi_k$ and $\mathbf P_k=\diag(\mathbf G_{0,k}^{1/2},\dots,\mathbf G_{m,k}^{1/2})$.
\end{lemma}
The proof of Lemma \ref{lem:feas.LP} is similar to that of Theorem \ref{theo:suff.cond.CTP} with $a_k=\xi_k$ and $\mathbf G_{i}=\mathbf G_{i,k}$, $i\in\{0\}\cup [m]$.

Since small constant traces are highly desirable for efficiency of first-order algorithms (e.g. CGAL), we search for an optimal solution of LP \eqref{eq:find.CTP.diag} instead of just a feasible solution.


\begin{remark}
One can extend the classes of diagonal matrices $\hat{\mathcal S}_{k}$, $\hat{\mathcal S}_{k-\lceil g_i\rceil}$ in \eqref{eq:find.CTP.diag} to obtain a smaller constant trace. For instance, one can define $\hat{\mathcal S}_{k}$, $\hat{\mathcal S}_{k-\lceil g_i\rceil}$ to be the classes of symmetric block diagonal matrices with block size $2$. 
As shown in \cite[Lemma 4.3]{wang2019second}, \eqref{eq:find.CTP.diag} then becomes a second-order cone program (SOCP) which can be also efficiently solved.
\end{remark}

\subsection{Special classes of POPs with CTP}
\label{sec:special.POP.CTP}
In this section we identify  two classes of POPs whose CTP can be verified by LP \eqref{eq:find.CTP.diag}.

\subsubsection{POPs with ball or annulus constraints on subsets of variables}


Consider the following assumption on the inequality constraints of POP \eqref{eq:POP.def}.
\begin{assumption}\label{ass:bound.cons}
There exists a nonnegative integer $r\le m/2$ and
\begin{itemize}
    \item $\overline R_{i}>\underline R_i>0$, $T_i\subseteq [n]$ for $i\in [r]$;
    \item $\overline R_{j}>0$, $T_{j}\subseteq [n]$ for $j\in [m]\backslash[2r]$
\end{itemize}
 such that 
 \begin{enumerate}[(1)]
     \item $(\cup_{i\in[r]} T_i) \cup (\cup_{j\in[m]\backslash[2r]} T_{j})=[n]$;
     \item $g_i:=\|\mathbf x(T_i)\|^2_2-\underline R_i$, $g_{i+r}:=\overline R_{i}-\|\mathbf x(T_i)\|^2_2$ for $i\in [r]$;
     \item $g_{j}:=\overline R_{j}-\|\mathbf x(T_{j})\|^2_2$ for $j\in [m]\backslash[2r]$.
 \end{enumerate}
\end{assumption}
Notice that if Assumption \ref{ass:bound.cons} holds then POP \eqref{eq:POP.def} has $r$ annulus constraints 
and $(m-2r)$ ball constraints on subsets of variables. 
Moreover, $Q(g)+I(h)$ is Archimedean due to (1) in Assumption \ref{ass:bound.cons}. 

\begin{example}
Assumption \ref{ass:bound.cons} holds in the following cases:
\begin{enumerate}[(1)]
    \item $m=1$, $r=0$ and $g_1:=\overline R_1-\|\mathbf x\|_2^2$, i.e., $S(g)$ is a ball;
    \item $m=n$, $r=0$ and $g_i:=\overline R_i-x_i^2$ for $i\in[n]$, i.e., $S(g)$ is a box;
    \item $m=2$, $r=1$ and $g_1:=\|\mathbf x\|_2^2 - \underline R_1$, $g_2:=\overline R_1-\|\mathbf x\|_2^2$ ($\overline R_1>\underline R_1>0$), i.e., $S(g)$ is an annulus.
\end{enumerate}
\end{example}

\begin{proposition}\label{prop:suff.cond.feas}
If Assumption \ref{ass:bound.cons} holds 
then LP \eqref{eq:find.CTP.diag} has a feasible solution for every $k\in\N^{\ge k_{\min}}$,
and therefore POP \eqref{eq:POP.def} has CTP.
\end{proposition}
The proof of Proposition \ref{prop:suff.cond.feas} is postponed to Appendix \ref{proof:suff.cond.feas}.

\subsubsection{POPs with inequality constraints of  equivalent degree}
\label{sec:CTP.equidegree}
We say that polynomials $p_1,\dots,p_t$ are of equivalent degree if $\lceil p_1\rceil =\dots=\lceil p_{t}\rceil$. 
\begin{assumption}\label{ass:equidegree}
Let $m\ge 3$ and $\{g_j\}_{j\in[m-2]}$ be of equivalent degree.
$L>0$ and $R>0$ are such that
$g_{m-1}=L-\sum\limits_{j\in[m-2]}g_j$ and $g_m=R-\|\mathbf x\|_2^2$.
\end{assumption}

\begin{proposition}\label{prop:CTP.equidegree}
If Assumption \ref{ass:equidegree} holds then LP \eqref{eq:find.CTP.diag} has a feasible solution for every $k\in\N^{\ge k_{\min}}$, and therefore POP \eqref{eq:POP.def} has CTP.
\end{proposition}

\begin{example}
Let $R,L>0$ satisfy $R\ge L^2$ and 
\begin{equation}\label{eq:stand.simplex}
    m=n+2\,,\, g_i=x_i \textrm{ for } i\in [n]\,,\,  g_{n+1}=L-\sum_{i\in[n]}x_i \text{ and }\ g_{n+2}=R-\|\mathbf x\|_2^2\,.
\end{equation}
Then Assumption \ref{ass:equidegree} holds and $S(g)$ is a simplex.
\end{example}

When $S(g)$ is compact, we can always reformulate POP \eqref{eq:POP.def} such that Assumption \ref{ass:equidegree} holds. Suppose $S(g)\subseteq \mathcal{B}_R$ for some $R$. Let $u=\max_{i\in[m]}\lceil g_i\rceil$. Set $\tilde g_i:=g_i(1+\|\mathbf x\|_2^2)^{u-\lceil g_i\rceil}$ for $i\in[m]$. Let $L$ be a positive number such that $\sum_{i\in[m]}\tilde g_i\le L$ on $S(g)$. Set $\tilde g_{m+1}:=L- \sum_{i\in[m]}\tilde g_i$ and $\tilde g_{m+2}:=R-\|\mathbf x\|_2^2$.
\begin{remark}\label{rem:choose.L}
For the latter case, one can choose any positive number $L\ge (R+1)^u\sum_{i\in[m]} \|g_i\|_1$. Indeed, for any $\mathbf z\in S(g)$, and since $\|\mathbf z\|_2^2\le R$:
\begin{equation*}
    |\mathbf z^\alpha|=\prod_{i\in[n]} |z_i|^{\alpha_i}\le \prod_{i\in[n]} (1+\|\mathbf z\|_2^2)^{\alpha_i/2}=(1+\|\mathbf z\|_2^2)^{|\alpha|/2}\le (1+R)^{t}\,,\,\forall \alpha\in \N^n_{2t}\,.
\end{equation*}
This implies that for every $i\in[m]$,
\begin{equation*}
    \tilde g_i(\mathbf z)\le (1+R)^{u-\lceil g_i\rceil} \sum_{\alpha\in\N^n_{2\lceil g_i\rceil}}|g_\alpha| |\mathbf z^\alpha|\le (1+R)^{u-\lceil g_i\rceil} (R+1)^{\lceil g_i\rceil} \|g_i\|_1 = (1+R)^{u}\|g_i\|_1\,.
\end{equation*}
Thus we have $\sum_{i\in[m]}\tilde g_i\le (1+R)^{u}\sum_{i\in[m]}\|g_i\|_1$ on $S(g)$.

\end{remark}

\begin{corollary}\label{coro:equi.pop.ctp}
With the above notation, $S(g\cup \{\tilde g_{m+1}, \tilde g_{m+2}\})=S(g)$ and LP \eqref{eq:find.CTP.diag} has a feasible solution when replacing $g$ by $g\cup \{\tilde g_{m+1},\tilde g_{m+2}\}$ for each $k\in\N^{\ge k_{\min}}$.
As a result, POP \eqref{eq:POP.def} is equivalent to the new POP
\begin{equation}
    f^\star:=\inf \{f(\mathbf x)\ :\ \mathbf x\in S(g\cup \{\tilde g_{m+1},\tilde g_{m+2}\})\cap V(h)\}
\end{equation} which has CTP.
\end{corollary}
The proof of Corollary \ref{coro:equi.pop.ctp} is postponed to Appendix \ref{proof:coro:equi.pop.ctp}.

In case where  POP \eqref{eq:POP.def} does not have CTP and $S(g)$ is compact, Corollary \ref{coro:equi.pop.ctp} provides a way to construct an equivalent POP by including two additional redundant  constraints. Then CTP of this new POP can be verified by LP.
\subsection{Main algorithm}
Algorithm \ref{alg:sol.nonsmooth.hier.B} below solves POP \eqref{eq:POP.def} whose CTP can be verified by LP. 

\begin{algorithm}
    \caption{SpecialPOP-CTP}
    \label{alg:sol.nonsmooth.hier.B} 
    \small
    \textbf{Input:} POP \eqref{eq:POP.def} and a relaxation order $k\in\N^{\ge k_{\min}}$\\
    \textbf{Output:} The optimal value $\tau_k$ of SDP \eqref{eq:SDP.form}
    \begin{algorithmic}[1]
    \State Solve LP \eqref{eq:find.CTP.diag} with an optimal solution $(\xi_k,\mathbf G_{i,k},\mathbf u_{j,k})$;
    \State Let $a_k=\xi_k$ and $\mathbf P_k=\diag(\mathbf G_{0,k}^{1/2},\dots,\mathbf G_{m,k}^{1/2})$;
    \State Compute the optimal value $\tau_k$ of SDP \eqref{eq:SDP.form} by 
    running an algorithm based on first-order methods, and which exploits CTP ;
    \end{algorithmic}
\end{algorithm}
    
Examples of algorithms based on first-order methods and which exploit CTP
are CGAL (Algorithm \ref{alg:CGAL} in Appendix \ref{sec:sdp.ctp.dense}) or SM (Algorithm \ref{alg:sol.SDP.CTP.0} in Appendix \ref{sec:spectral.method.dense}).

 
\section{Exploiting CTP for POPs with CS}
\label{sec:CTP.correlative.POP}

In this section, we extend  the CTP-exploiting framework to POPs with
sparsity. For clarity of exposition we only consider \emph{correlative sparsity} (CS).
However, in Appendix \ref{sec:sparse.POP.TS.CSTS} we also treat
\emph{term sparsity} (TS) \cite{wang2019tssos} as well as \emph{correlative-term sparsity} (CS-TS) \cite{wang2020cs}. Since the methodology is very similar
to that in the dense case described earlier, we omit details and only present the main results. 

To begin with, we recall some basic facts on exploiting CS for POP \eqref{eq:POP.def} initially proposed in \cite{waki2006sums} by Waki et al.

\subsection{POPs with CS}\label{sec:sparse.POP}
For $\alpha\in\N^n$, let $\supp(\alpha):=\{j\in[n]:\alpha_j>0\}$.
Assume $I\subseteq [n]$. Given $\mathbf y=(y_\alpha)_{\alpha\in\N^n_d}$, the moment (resp. localizing) submatrix associated to $I$ of order $d$ is defined by $\mathbf M_d(\mathbf y,I):=(y_{\alpha+\beta})_{\alpha,\beta\in \N^I_d}$ (resp. $\mathbf M_d(q\mathbf y, I):=(\sum_{\gamma}q_\gamma y_{\alpha+\beta+\gamma})_{\alpha,\beta\in \N^I_d}$ for $q\in\R[x(I)]$).
Let $\mathbf v_d^I:=(\mathbf x^\alpha)_{\alpha\in \N^I_d}$ with length $\s(|I|,d):=\binom{|I|+d}{n}$.

Assume that $\{I_j\}_{j\in[p]}$ (with $n_j:=|I_j|$) are the maximal cliques of (a chordal extension of) the correlative sparsity pattern (csp) graph associated to POP \eqref{eq:POP.def}, as defined in \cite{waki2006sums}.

Let $\{J_j\}_{j\in[p]}$ (resp. $\{W_j\}_{j\in[p]}$) be a partition of $[m]$ (resp. $[l]$) such that for all $i\in J_j$, $g_i\in \R[x(I_j)]$ (resp. $i\in W_j$, $h_i\in \R[x(I_j)]$), $j\in[p]$.
For each  $j\in [p]$, let $m_j:=|J_j|$, $l_j:=|W_j|$ and $g_{J_j}:=\{g_i\,:\,i\in J_j\}$, $h_{W_j}:=\{h_i\,:\,i\in W_j\}$.
Then $Q(g_{J_j})$ (resp. $I(h_{W_j})$) is a quadratic module (resp. an ideal) in $\R[x(I_j)]$, for $j\in[p]$.

For each $k\in \N^{\ge k_{\min}}$, consider the following sparse SOS relaxation:
\begin{equation}\label{eq:primal.cs.SOS}
\rho_{k}^{\text{cs}} := \sup\ \left\{\xi\,:\,f-\xi\in\sum_{j\in[p]} (Q_k(g_{J_j})+I_k(h_{W_j}))\right\}\,.
\end{equation}
It is equivalent to the SDP:
\begin{equation}\label{eq:primal.cs.SDP}
\rho_{k}^{\text{cs}} = \sup \limits_{\xi,\mathbf G_i^{(j)},\mathbf u_i^{(j)}} \left\{ \xi\ \left|\begin{array}{rl}
&{\mathbf G}_i^{(j)} \succeq 0\,,\,i\in\{0\}\cup J_j\,,\,j\in [p]\,,\\
&f-\xi=\sum_{j\in[p]}\left((\mathbf v_k^{I_j})^\top {\mathbf G}_0^{(j)} \mathbf v_k^{I_j}\right.\\
&\qquad\qquad+\sum_{i\in J_j} g_i (\mathbf v_{k-\lceil g_i\rceil}^{I_j})^\top {\mathbf G}_i^{(j)} \mathbf v_{k-\lceil g_i\rceil}^{I_j}\\
&\qquad\qquad\left.+\sum_{i\in W_j} h_i (\mathbf v_{2(k-\lceil h_i\rceil)}^{I_j})^\top {\mathbf u}_i^{(j)} \right)
\end{array}
\right. \right\}\,.
\end{equation}

The dual of \eqref{eq:primal.cs.SDP} reads:
\begin{equation}\label{eq:dual.cs.SDP} 
\tau_k^{\text{cs}} \,:= \,\inf \limits_{\mathbf y \in {\R^{\s({2k})} }} \left\{ L_{\mathbf y}(f)\ \left|\begin{array}{rl}
& \mathbf M_k(\mathbf y, I_j) \succeq 0\,,\,j\in[p]\,,\,y_{\mathbf 0}\,=\,1\,.\\
&\mathbf M_{k - \lceil g_i \rceil }(g_i\;\mathbf y,I_j)   \succeq 0\,,\,i\in J_j\,,\,j\in[p]\,,\\
&\mathbf M_{k - \lceil h_i \rceil }(h_i\;\mathbf y,I_j)   = 0\,,\,i\in W_j\,,\,j\in[p]
\end{array}\right. \right\}\,.
\end{equation}


It is shown in \cite[Theorem 3.6]{lasserre2006convergent} that convergence of the primal-dual \eqref{eq:primal.cs.SDP}-\eqref{eq:dual.cs.SDP} to $f^\star$ is guaranteed if there are additional ball constraints on each clique of variables.


\subsection{Exploiting CTP for POPs with CS}
\label{sec:def.CTP.each.clique}
Consider POP \eqref{eq:POP.def} with CS described in Section \ref{sec:sparse.POP}.
For every $j\in[p]$ and for every $k\in\N^{\ge k_{\min}}$, letting $\mathbf D_k(\mathbf y,I_j):=\diag(\mathbf M_k(\mathbf y,I_j),(\mathbf M_{k-\lceil g_i\rceil}(g_i\mathbf y,I_j))_{i\in  J_j})$ for $\mathbf y\in\R^{s(k)}$,
SDP \eqref{eq:dual.cs.SDP} can be rewritten as
\begin{equation}\label{eq:cs-ts.SDP.simple}
\tau_{k}^{\text{cs}} \,:= \,\inf \limits_{\mathbf y \in {\R^{\s({2k})} }} \left\{ L_{\mathbf y}(f)\ \left|\begin{array}{rl}
& \mathbf D_k(\mathbf y, I_j) \succeq 0\,,\,j\in[p]\,,\,y_{\mathbf 0}\,=\,1\,,\\
&\mathbf M_{k - \lceil h_i \rceil }(h_i\;\mathbf y,I_j)   = 0\,,\,i\in W_j\,,\,j\in[p]
\end{array}\right.\right\}\,.
\end{equation}
We define CTP for POP with CS as follows.
\begin{definition}\label{def:ctp.cs}(CTP for a POP with CS)
We say that POP \eqref{eq:POP.def} with CS has CTP if for every $k\in\N^{\ge k_{\min}}$ and for every $j\in[p]$, there exists a positive number $a_k^{(j)}$ and a positive definite matrix $\mathbf P_{k}^{(j)}\in \mathcal{S}_k$ such that  for all $\mathbf y \in \R^{\s(2k)}$,
\begin{equation}
    \left.
\begin{array}{rl}
&\mathbf M_{k-\lceil h_i\rceil}(h_i\mathbf y,I_j)=0\,,\,i\in W_j\,,\\
&y_{\mathbf 0}=1
\end{array}
\right\}\Rightarrow  \trace(\mathbf P_{k}^{(j)} \mathbf D_k(\mathbf y,I_j) \mathbf P_{k}^{(j)})=a_k^{(j)}\,.
\end{equation}
\end{definition}

The following result provides a sufficient condition for a POP with CS to have CTP. 
\begin{theorem}\label{theo:generic.ctp.cs}
Assume that there is a ball constraint on each clique of variables, i.e.,
\begin{equation}\label{eq:ball.on.clique}
    \forall\ j\in[p],\, R_j-\|\mathbf x(I_j)\|_2^2\in g \textrm{ for some }R_j>0\,.
\end{equation}
Then one has
$\R^{>0}\subseteq Q_k^\circ(g_{J_j})$, for all $k\in\N^{\ge k_{\min}}$ and for all $j\in [p]$.
As a consequence, POP \eqref{eq:POP.def} has CTP.
\end{theorem}

The proof of Theorem \ref{theo:generic.ctp.cs} being very similar to that 
of Theorem \ref{theo:generic.ctp} by considering each clique of variables, is omitted.

Again by considering each clique of variables, the following result can be obtained from Theorem \ref{theo:generic.ctp.cs} in the same way Corollary \ref{coro:stric.fea.sol.sos} was obtained. 

\begin{corollary}\label{coro:slater.con.cs}
If \eqref{eq:ball.on.clique} holds then Slater's condition for SDP \eqref{eq:primal.cs.SDP} holds for all $k\in\N^{\ge k_{\min}}$.
\end{corollary}

We are now in position to provide a general method to solve POPs with CS which have CTP.\\

Consider POP \eqref{eq:POP.def} with CS described in Section \ref{sec:sparse.POP}. Assume that POP \eqref{eq:POP.def} has CTP and let $k\in\N^{\ge k_{\min}}$ be fixed.
For every $j\in[p]$, we denote by $\mathcal{S}_{j,k}$ the set of real symmetric matrices of size $\s(k,n_j)+\sum_{i\in J_j}\s(k-\lceil g_i\rceil,n_j)$ in a block diagonal form:
$\mathbf X=\diag(\mathbf X_{0},(\mathbf X_i)_{i\in J_j})$
such that $\mathbf X_{0}$ is a block of size $\s(k,n_j)$ and $\mathbf X_i$ is a block of size $\s(k-\lceil g_i\rceil,n_j)$ for $i\in J_j$.

Letting 
\begin{equation}\label{eq:convert.momentmat.sparse}
\mathbf X_j=\mathbf P_{k}^{(j)}\mathbf D_k(\mathbf y,I_j)\mathbf P_{k}^{(j)}\,,\,j\in[p]\,,
\end{equation}
SDP \eqref{eq:cs-ts.SDP.simple} can be rewritten as:
\begin{equation}\label{eq:SDP.form.sparse}
\tau_{k}^{\text{cs}} = \inf _{\mathbf X_j\in \mathcal{S}_{j,k}^+} \left\{ \sum_{j\in[p]}\left< \mathbf C_{j,k},\mathbf X_j\right>\,:\,\sum_{j\in[p]}\mathcal{A}_{j,k} \mathbf X_j=\mathbf b_k\,,\,j\in[p]\right\}\,,
\end{equation}
where for every $j\in[p]$, $\mathcal{A}_{j,k}:\mathcal{S}_{j,k}\to \R^{\zeta_k}$ is a linear operator of the form
$\mathcal{A}_{j,k}\mathbf X=(\left< \mathbf A_{j,k,1},\mathbf X\right>,\dots,\left< \mathbf A_{j,k,\zeta_k},\mathbf X\right>)$
with $\mathbf A_{j,k,i} \in \mathcal{S}_{j,k}$, $i\in[\zeta_k]$, $\mathbf C_{j,k} \in \mathcal{S}_{j,k}$, $j\in[p]$ and $\mathbf b_k\in \R^{\zeta_k}$. 
See Appendix \ref{sec:convert.standart.SDP.cs} for the conversion of SDP \eqref{eq:cs-ts.SDP.simple} to the form \eqref{eq:SDP.form.sparse}.

The dual of SDP \eqref{eq:SDP.form.sparse} reads as:
\begin{equation}\label{eq:SDP.form.sparse.dual}
\rho^{\text{cs}}_k = \sup_{\mathbf y\in\R^\zeta}\, \left\{\, \mathbf b_k^\top\mathbf y\,:\,
\mathcal{A}_{j,k}^\top \mathbf y-\mathbf C_{j,k}\in \mathcal{S}^+_{j,k}\,,j\in[p]\,\right\}\,,
\end{equation}
where
$\mathcal{A}_{j,k}^\top:\R^{\zeta}\to \mathcal{S}_{j,k}$ is the adjoint operator of $\mathcal{A}_{j,k}$, i.e., $\mathcal{A}_{j,k}^\top\mathbf z=\sum_{i\in[\zeta]} z_i\mathbf A_{j,k,i}$, $j\in[p]$.
By Definition \ref{def:ctp.cs}, it holds that for every $k\in\N^{\ge k_{\min}}$,
\begin{equation}
\left.
\begin{array}{lr}
    \forall \ \mathbf X_j\in \mathcal S_{j,k}\,,\,j\in[p]\\
    \sum_{j\in[p]}\mathcal{A}_{j,k} \mathbf X_j=\mathbf b_k
        \end{array}
    \right\}
    \Rightarrow \trace(\mathbf X_j)=a_k^{(j)}\,,\,j\in[p]\,.
\end{equation}

After replacing $(\mathcal{A}_{j,k}, \mathbf{A}_{j,k,i},  \mathbf b_k, \mathbf C_{j,k}, \mathcal{S}_{j,k}, \zeta_k,  \tau_{k}^{\text{cs}}, a_k^{(j)})$ by $(\mathcal{A}_j, \mathbf{A}_{i,j}, \mathbf b, \mathbf C_j, \mathcal{S}_j, \zeta, \tau, a_j)$, SDP \eqref{eq:SDP.form.sparse} then becomes SDP \eqref{eq:SDP.form.0.blocks}; see Appendix \ref{sec:sdp.ctp.blocks} with $\omega_j=m_j+1$ and $s^{\max}=\max_{j\in[p]}\s(k,n_j)$.

If there is a ball constraint on each clique of variables then by Corollary \ref{coro:slater.con.cs}, strong duality holds for the 
pair \eqref{eq:SDP.form.sparse}-\eqref{eq:SDP.form.sparse.dual}, for every $k\in\N^{\ge k_{\min}}$.

The two algorithms (CGAL and SM) based on first-order methods are then leveraged to solve the primal-dual \eqref{eq:SDP.form.sparse}-\eqref{eq:SDP.form.sparse.dual}; see Appendix \ref{sec:sdp.ctp.blocks} and Appendix \ref{app:spectral.SDP.ctp}.


\subsection{Verifying CTP for POPs with CS via LP}
As in the dense case, we can verify CTP for a POP with CS via a series of LPs.

For every $k\in\N^{\ge k_{\min}}$ and for every $j\in[p]$, let $\hat{\mathcal  S}_{k,j}$ be the set of real diagonal matrices of size $\s(k,n_j)$ and consider the following LP:
\begin{equation}\label{eq:find.CTP.cliq}
\inf \limits_{\xi,\mathbf G_i,\mathbf u_i} \left\{ \xi\ \left|\begin{array}{rl}
&\mathbf G_0-\mathbf I_0 \in \hat{\mathcal S}_{k,j}^+\,,\,\mathbf G_i-\mathbf I_i \in \hat{\mathcal S}_{k-\lceil g_i\rceil,j}^+\,,\,i\in J_j\,,\\
&\xi=(\mathbf v_k^{I_j})^\top {\mathbf G}_0 \mathbf v_k^{I_j}+\sum_{i\in J_j} g_i (\mathbf v_{k-\lceil g_i\rceil}^{I_j})^\top {\mathbf G}_i \mathbf v_{k-\lceil g_i\rceil}^{I_j}\\
&\qquad\qquad\qquad\qquad\qquad+\sum_{i\in W_j} h_i (\mathbf v_{2(k-\lceil h_i\rceil)}^{I_j})^\top {\mathbf u}_i 
\end{array}
\right. \right\}\,,
\end{equation}
where $\mathbf I_i$ is the identity matrix, for every $i\in\{0\}\cup J_j$.

\begin{lemma}\label{lem:feas.LP.cs}
Let POP \eqref{eq:POP.def} with CS be described in Section \ref{sec:sparse.POP}. If LP \eqref{eq:find.CTP.cliq} has a feasible solution $(\xi_{k}^{(j)},\mathbf G_{i,k}^{(j)},\mathbf u_{i,k}^{(j)})$, for every $k\in\N^{\ge k_{\min}}$ and for every $j\in[p]$, then POP \eqref{eq:POP.def} has CTP with $\mathbf P_k^{(j)}=\diag(\mathbf G_{0,k}^{1/2},(\mathbf G_{i,k}^{1/2})_{i\in J_i})$ and $a_k^{(j)}=\xi_k^{(j)}$, for $k\in\N^{\ge k_{\min}}$ and for $j\in[p]$.
\end{lemma}
The proof of Lemma \ref{lem:feas.LP.cs} is similar to that of Lemma \ref{lem:feas.LP}.

For instance, for POPs with ball or annulus constraints on subsets of each clique of variables,
CTP can be verified by LP.



\begin{proposition}\label{prop:suff.cond.feas.cs}
Let POP \eqref{eq:POP.def} with CS be described in Section \ref{sec:sparse.POP}. Let $(T_i)_{i\in[r]\cup([m]\backslash [2r])}$ be as in Assumption \ref{ass:bound.cons} and further assume that for every $j\in[p]$, $(\cup_{q\in J_j\cap [r]} T_q)\cup (\cup_{q\in J_j\backslash [2r]} T_q)=I_j$. Then LP \eqref{eq:find.CTP.cliq} has a feasible solution for every $k\in\N^{\ge k_{\min}}$, and therefore POP \eqref{eq:POP.def} has CTP.
\end{proposition}
The proof of Proposition \ref{prop:suff.cond.feas.cs} is postponed to Appendix \ref{sec:proof.prop:suff.cond.feas.cs}.

\subsection{Main algorithm}

Algorithm \ref{alg:sol.nonsmooth.hier.B.cs} below solves POP \eqref{eq:POP.def} with CS and whose CTP can be verified by LP.

    \begin{algorithm}
    \caption{SpecialPOP-CTP-CS}
    \label{alg:sol.nonsmooth.hier.B.cs} 
    \small
    \textbf{Input:} POP \eqref{eq:POP.def} with CS and a relaxation order $k\in\N^{\ge k_{\min}}$\\
    \textbf{Output:} The optimal value $\tau_k^{\text{cs}}$ of SDP \eqref{eq:SDP.form.sparse}
    \begin{algorithmic}[1]
    
    \For {$j\in[p]$}{} 
    \State Solve LP
    \eqref{eq:find.CTP.cliq} to obtain an optimal solution $(\xi_k^{(j)},\mathbf G_{i,k}^{(j)},\mathbf u_{j,k}^{(j)})$;
    \State Let $a_k^{(j)}=\xi_k^{(j)}$ and $\mathbf P_k^{(j)}=\diag((\mathbf G_{0,k}^{(j)})^{1/2},\dots,(\mathbf G_{m,k}^{(j)})^{1/2})$;
    \EndFor
  \State Compute the optimal value $\tau_k^{\text{cs}}$ of SDP \eqref{eq:SDP.form.sparse}  by running an algorithm based on first-order methods and which exploits
  CTP.
    \end{algorithmic}
    \end{algorithm}

In Step 4 of Algorithm \ref{alg:sol.nonsmooth.hier.B.cs} 
the two algorithms CGAL (Algorithm \ref{alg:CGAL.blocks} in Appendix \ref{sec:sdp.ctp.blocks} or SM (Algorithm \ref{alg:sol.SDP.CTP.0.blocks} in Appendix \ref{app:spectral.SDP.ctp}) are good candidates.
    

\section{Numerical experiments}
\label{sec:benchmark}
In this section we report results of numerical experiments obtained by solving the 
second-order Moment-SOS relaxation of various randomly generated instances 
of QCQPs with CTP. 
The experiments are performed in Julia 1.3.1 with the following software packages:
\begin{itemize}
    \item {\tt SumOfSquare} \cite{weisser2019polynomial} is a modeling library for solving the Moment-SOS relaxations of dense POPs, based on JuMP (with Mosek 9.1 used as SDP solver).
    \item {\tt TSSOS} \cite{wang2019tssos,wang2020chordal,wang2020cs} is a modeling library for solving Moment-SOS relaxations of sparse POPs based on JuMP (with Mosek 9.1 used as SDP solver).
    \item {\tt LMBM} solves unconstrained non-smooth optimization with the limited-memory bundle method by Haarala et al. \cite{haarala2007globally,haarala2004new}
    and calls Karmitsa's Fortran implementation of the LMBM algorithm  \cite{karmitsa2007lmbm}.
    \item {\tt Arpack} \cite{lehoucq1998arpack} is used to compute the smallest eigenvalues and the corresponding eigenvectors of real symmetric matrices of (potentially) large size, which is based on the implicitly restarted Arnoldi method.
\end{itemize}

The implementation of algorithms \ref{alg:sol.nonsmooth.hier.B} and \ref{alg:sol.nonsmooth.hier.B.cs} is available online via the link:
\begin{center}
    \href{https://github.com/maihoanganh/SpectralPOP}{{\bf https://github.com/maihoanganh/ctpPOP}}.
\end{center}

We use a desktop computer with an Intel(R) Core(TM) i7-8665U CPU @ 1.9GHz $\times$ 8 and 31.2 GB of RAM. 
The notation for the numerical results is given in Table \ref{tab:nontation}.
\begin{table}
    \caption{\small The notation}
    \label{tab:nontation}
\small
\begin{center}
\begin{tabular}{|m{1.0cm}|m{10.3cm}|}
\hline
$n$&the number of variables of a POP\\
\hline
$m$&the number of inequality constraints of a POP\\
\hline
$l$&the number of equality constraints of a POP\\
\hline
$u^{\max}$& the largest size of variable cliques of a sparse POP\\
\hline
$p$& the number of variable cliques of a sparse POP\\
\hline
$k$&the relaxation order of the Moment-SOS hierarchy\\
\hline
$t$&the sparse order of the sparsity adapted Moment-SOS hierarchy (for TS and CS-TS)\\
\hline
$\omega$& the number of psd blocks in an SDP\\
\hline
$s^{\max}$&the largest size of psd blocks in an SDP\\
\hline
$\zeta$&the number of affine equality constraints in an SDP\\
\hline
$a^{\max}$& the largest constant trace\\
\hline
Mosek & the SDP relaxation modeled by {\tt SumOfSquares} (for dense POPs) or {\tt TSSOS} (for sparse POPs) and solved by Mosek 9.1\\
\hline
CGAL & the SDP relaxation modeled by our CTP-exploiting method and solved by the CGAL algorithm\\
\hline
LMBM &the SDP relaxation modeled by our CTP-exploiting method and 
solved by the SM algorithm with the {\tt LMBM} solver\\
\hline
val& the optimal value of the SDP relaxation\\
\hline
gap& the relative optimality gap w.r.t. the value returned by Mosek, i.e.,

$\text{gap}=|\text{val}-\text{val(Mosek)}|/{|\text{val(Mosek)}|}$\\
\hline
time & the running time in seconds (including  modeling and solving time)\\
\hline
$-$& the calculation runs out of space\\
\hline
\end{tabular}    
\end{center}
\end{table}

For the examples tested in this paper, the modeling time of {\tt SumOfSquares}, {\tt TSSOS} and {\tt ctpPOP} is typically negligible compared to the solving time of the 
packages Mosek, CGAL, and LMBM. Hence the total running time mainly depends on the solvers and we compare their performances below.
As mentioned in the introduction, the current framework differs from our previous work \cite{mai2020hierarchy}, where we exploited CTP for equality constrained POPs on a sphere, which could be solved by {\tt LMBM} efficiently.
The reason is that the SDP relaxations of such equality constrained POPs involve a single psd matrix. 
For the benchmarks of this section, we consider POPs involving ball/annulus constraints, so the resulting relaxations include several psd matrices.
Our numerical experiments confirm that for such SDPs, LMBM returns inaccurate values while CGAL (without sketching) performs better for this type of SDP in terms of accuracy and efficiency.

\subsection{Randomly generated dense QCQPs with a ball constraint}
\label{sec:experiment.single.ball}
\paragraph{Test problems:} We construct randomly generated dense QCQPs with a ball constraint as follows:
\begin{enumerate}
    \item Generate a dense quadratic polynomial objective function $f$ with random coefficients following the uniform probability distribution on $(-1,1)$.
        \item Let $m=1$ and $g_1:=1-\|\mathbf x\|_2^2$;
    \item Take a random point $\mathbf a$ in $S(g)$ w.r.t. the uniform distribution;
    \item For every $j\in[l]$, generate a dense quadratic polynomial $h_j$ by
    \begin{enumerate}[(i)]
        \item for each $\alpha\in\N^n_2\backslash \{\mathbf 0\}$, taking a random coefficient $h_{j,\alpha}$ for $h_j$ in $(-1,1)$ w.r.t. the uniform distribution;
        \item setting $h_{j,\mathbf 0}:=-\sum_{\alpha\in\N^n_2\backslash \{\mathbf 0\}} h_{j,\alpha} \mathbf a^\alpha$.
    \end{enumerate}
    Then $\mathbf a$ is a feasible solution of POP \eqref{eq:POP.def}.
\end{enumerate}

The numerical results are displayed in Table \ref{tab:quadratics.on.unit.ball} and \ref{tab:QCQP.on.unit.ball}. 
\begin{table}
    \caption{\small Numerical results for minimizing a dense quadratic polynomial on a unit ball}
    \label{tab:quadratics.on.unit.ball}
    {\small\begin{itemize}
        \item POP size: $m=1$, $l=0$; Relaxation order: $k=2$; SDP size: $\omega=2$, $a^{\max}=3$.
    \end{itemize}}
\small
\begin{center}
   \begin{tabular}{|c|c|c|c|c|c|c|c|c|c|}
        \hline
        \multicolumn{1}{|c|}{POP size}&
        \multicolumn{2}{c|}{SDP size} & \multicolumn{2}{c|}{Mosek }&      \multicolumn{2}{c|}{CGAL}&      \multicolumn{2}{c|}{LMBM} \\ \hline
$n$ &$s^{\max}$ & $\zeta$&
\multicolumn{1}{c|}{val}& \multicolumn{1}{c|}{time}&
\multicolumn{1}{c|}{val}&
\multicolumn{1}{c|}{time}&
\multicolumn{1}{c|}{val}&
\multicolumn{1}{c|}{time}\\
\hline
10 &66 &1277& -2.2181 & 0.3 &-2.2170  &0.2 & -2.2187 & 0.3\\
\hline
20 &231 &16402 & -3.7973 & 4 &-3.7947  &0.6 & -3.7096 & 7\\
\hline
30 &496 &77377 & -3.6876 & 3474 &-3.6858  &104 & -3.8530 & 59\\
\hline
40 &861 &236202 & $-$ & $-$ &-4.1718  &33 & -4.7730 & 179\\
\hline
50 &1326 &564877 & $-$ & $-$ &-6.3107  &1007 & -7.3874 & 139\\
\hline
60 &1891 &1155402 & $-$ & $-$ &-6.5326  &1085 & -7.4733 & 674\\
\hline
70 &2556 &2119777 & $-$ & $-$ &-7.3379  &1262 & -9.5223 & 1486\\
\hline
80 &3321 &3590002 & $-$ & $-$ &-7.9559  &4988 & -10.0260 & 1241\\
\hline
90 &4186 &5718077 & $-$ & $-$ &-7.3425  &5187 & -9.4477 & 5313\\
\hline
100 &5151 &8676002 & $-$ & $-$ &-7.7374  &22451 & -10.684 & 5355\\
\hline
\end{tabular}    
\end{center}
\end{table}
\begin{table}
    \caption{\small Numerical results for randomly generated dense QCQPs with a ball constraint}
    \label{tab:QCQP.on.unit.ball}
    {\small\begin{itemize}
        \item POP size: $m=1$, $l=\lceil n/4\rceil $; Relaxation order: $k=2$; SDP size: $\omega=2$, $a^{\max}=3$.
    \end{itemize}}
\small
\begin{center}
   \begin{tabular}{|c|c|c|c|c|c|c|c|c|c|}
        \hline
        \multicolumn{2}{|c|}{POP size}&
        \multicolumn{2}{c|}{SDP size} & \multicolumn{2}{c|}{Mosek}&      \multicolumn{2}{c|}{CGAL}&      \multicolumn{2}{c|}{LMBM} \\ \hline
$n$&$l$ &$s^{\max}$ & $\zeta$&
\multicolumn{1}{c|}{val}& \multicolumn{1}{c|}{time}&
\multicolumn{1}{c|}{val}&
\multicolumn{1}{c|}{time}&
\multicolumn{1}{c|}{val}&
\multicolumn{1}{c|}{time}\\
\hline
10&3   & 66 & 1475 & -2.0686 & 1.7 &-2.0674 & 0.8 & -2.0874 & 0.3\\
\hline
20&5   & 231 & 17557 & -3.0103 & 61 &-3.0075 & 7 & -3.0750 & 18\\
\hline
30&8   & 496 & 81345 & -3.3293 & 4573 &-3.3249 &80 & -3.6863 & 123\\
\hline
40&10   & 861 & 244812 & $-$ & $-$ & -4.6977 & 194 & -5.3488  & 488 \\
\hline
50&13   & 1326 & 582115 & $-$ & $-$ & -4.2394 &951  & -6.1325 & 837\\
\hline
60&15   & 1891 & 1183767 & $-$ & $-$ & -5.7793 & 1387 & -7.5718 & 3781\\
\hline
70&18   & 2556 & 2165785 & $-$ & $-$ & -6.1278 & 4335 & -8.1181 & 15854\\
\hline
\end{tabular}    
\end{center}
\end{table}

\paragraph{Discussion:}
As one can see from Table \ref{tab:quadratics.on.unit.ball} and \ref{tab:QCQP.on.unit.ball},
CGAL is typically the fastest solver and returns an optimal value of gap within 1\% w.r.t. the one returned by Mosek when $n\le 30$. Mosek runs out of memory when $n\ge 40$ while CGAL works well up to $n=100$.
We should point out that LMBM is less accurate or even fails to converge to the optimal value when $n\ge 20$. The reason might be that LMBM only solves the dual problem and hence looses information of the primal problem.

\subsection{Randomly generated dense QCQPs with annulus constraints}
\label{sec:experiment.single.annulus}
\paragraph{Test problems:} We construct randomly generated dense QCQPs
as in Section \ref{sec:experiment.single.ball}, where the ball constraint is now replaced by annulus constraints. Namely, in Step 2 we take $m=2$, $g_1:=\|\mathbf x\|_2^2-1/2$ and $g_2:=1-\|\mathbf x\|_2^2$.
The numerical results are displayed in Table \ref{tab:quadratics.on.annulus} and \ref{tab:QCQP.on.annulus}. 
\begin{table}
    \caption{\small Numerical results for minimizing a dense quadratic polynomial on an annulus}
    \label{tab:quadratics.on.annulus}
    {\small\begin{itemize}
        \item POP size: $m=2$, $l=0$; Relaxation order: $k=2$; SDP size: $\omega=3$, $a^{\max}=4$.
    \end{itemize}}
\small
\begin{center}
   \begin{tabular}{|c|c|c|c|c|c|c|c|c|}
        \hline
        \multicolumn{1}{|c|}{POP size}&
        \multicolumn{2}{c|}{SDP size} & \multicolumn{2}{c|}{Mosek}&      \multicolumn{2}{c|}{CGAL}&      \multicolumn{2}{c|}{LMBM} \\ \hline
$n$ &$s^{\max}$ & $\zeta$&
\multicolumn{1}{c|}{val}& \multicolumn{1}{c|}{time}&
\multicolumn{1}{c|}{val}&
\multicolumn{1}{c|}{time}&
\multicolumn{1}{c|}{val}&
\multicolumn{1}{c|}{time}\\
\hline
10 &66 &1343 & -3.0295 & 0.5 &-3.0278 &1 &-3.0311&0.8  \\
\hline
20 &231 &16633 & -3.6468 & 69 &-3.6458 &5 &-3.7814&16  \\
\hline
30 &496 &77873 & -3.9108 & 2546 &-3.9079 &9 &-3.8941&51  \\
\hline
40 &861 &237063 & $-$ & $-$ &-4.7469 &28 &-6.9780&119  \\
\hline
50 &1326 &566203 & $-$ & $-$ &-6.4170 &112 &-11.1028&258  \\
\hline
60 &1891 &1157293 & $-$ & $-$ &-5.5841 &226 &-9.2142&473  \\
\hline
70 &2556 &2122333 & $-$ & $-$ &-7.9325 &730 & -12.7862& 1669 \\
\hline
80 &3321 &3593323 & $-$ & $-$ &-7.6164 &1355 & -10.068 & 317 \\
\hline
90 &4186 &5722263 & $-$ & $-$ &-8.1900 &3563 & -12.439 & 8751 \\
\hline
\end{tabular}    
\end{center}
\end{table}
\begin{table}
    \caption{\small Numerical results for randomly generated dense QCQPs with annulus constraints}
    \label{tab:QCQP.on.annulus}
    {\small\begin{itemize}
        \item POP size: $m=2$, $l=\lceil n/4\rceil$; Relaxation order: $k=2$; SDP size: $\omega=3$, $a^{\max}=4$.
    \end{itemize}}
\small
\begin{center}
   \begin{tabular}{|c|c|c|c|c|c|c|c|c|c|}
        \hline
        \multicolumn{2}{|c|}{POP size} &
        \multicolumn{2}{c|}{SDP size} & \multicolumn{2}{c|}{Mosek}&      \multicolumn{2}{c|}{CGAL}&      \multicolumn{2}{c|}{LMBM} \\ \hline
$n$&$l$ &$s^{\max}$ & $\zeta$&
\multicolumn{1}{c|}{val}& \multicolumn{1}{c|}{time}&
\multicolumn{1}{c|}{val}&
\multicolumn{1}{c|}{time}&
\multicolumn{1}{c|}{val}&
\multicolumn{1}{c|}{time}\\
\hline
10&3   & 66 & 1541 & -2.7950 & 0.5 &-2.7934 & 2&-2.7829&7\\
\hline
20&5   & 231 & 17788 & -3.5048  & 95  &-3.5027 & 10& -4.4491 & 46\\
\hline
30&8   & 496 & 81841 & -3.3964  &  4237  &-3.3937 & 45& -4.9592 & 111\\
\hline
40&10  & 861 & 245673 & $-$  &  $-$  &-4.6573 & 140 & -6.7683 & 648\\
\hline
50&13  & 1326 & 583441 & $-$ & $-$ & -3.8236 & 437 & -6.9930 & 519\\
\hline
60&15 & 1891 & 1185658 & $-$ & $-$ & -4.5246 & 1076 & -7.5845 &2917 \\
\hline
70&18 & 2556 & 2168341 & $-$ & $-$ & -6.2924 & 4783 & -9.6145 & 2644 \\
\hline
\end{tabular}    
\end{center}
\end{table}

\paragraph{Discussion:} Same remarks as in Section \ref{sec:experiment.single.ball}.

\subsection{Randomly generated dense QCQPs with box constraints}
\label{sec:experiment.box}
\paragraph{Test problems:} We construct randomly generated dense QCQPs
as in Section \ref{sec:experiment.single.ball}, where the ball constraint is now replaced by box constraints. Namely, in Step 2 we take $m=n$, $g_j:=-x_j^2+1/n$, $j\in [n]$. 

The numerical results are displayed in Table
\ref{tab:quadratics.on.box} and \ref{tab:QCQP.on.box}. 
\begin{table}
    \caption{\small Numerical results for minimizing a dense quadratic polynomial on a box}
    \label{tab:quadratics.on.box}
    {\small\begin{itemize}
        \item POP size: $m=n$, $l=0$; Relaxation order: $k=2$; SDP size: $\omega=n+1$, $a^{\max}=3$.
    \end{itemize}}
\small
\begin{center}
   \begin{tabular}{|c|c|c|c|c|c|c|c|c|c|}
        \hline
        \multicolumn{1}{|c|}{POP size} &
        \multicolumn{2}{c|}{SDP size} & \multicolumn{2}{c|}{Mosek}&      \multicolumn{2}{c|}{CGAL}&      \multicolumn{2}{c|}{LMBM} \\ \hline
$n$ &$s^{\max}$ & $\zeta$&
\multicolumn{1}{c|}{val}& \multicolumn{1}{c|}{time}&
\multicolumn{1}{c|}{val}&
\multicolumn{1}{c|}{time}&
\multicolumn{1}{c|}{val}&
\multicolumn{1}{c|}{time}\\
\hline
10& 66 &1871 & -2.7197 & 0.5 &-2.7189 & 1 &-2.7327 &0.7  \\
\hline
20& 231 & 20791 & -3.3560 & 98 & -3.3501 & 57 & -4.2987 & 18 \\
\hline
30& 496 & 91761 & -4.6372 & 5150 & -4.6242 & 285 & -5.8805 & 156 \\
\hline
40& 861 &269781 & $-$ & $-$ & -4.5788 & 409 & -6.5857 & 188 \\
\hline
50& 1326 &629851 & $-$ & $-$ & -4.2313 & 2083 & -6.6163  & 323 \\
\hline
60& 1891 &1266971 & $-$ & $-$ & -4.0135 & 5525 & -6.5792  & 814 \\
\hline
70& 2556 &2296141 & $-$ & $-$ & -5.4019 & 15172 &  -8.7669 & 1434 \\
\hline
\end{tabular}    
\end{center}
\end{table}
\begin{table}
    \caption{\small Numerical results for randomly generated dense QCQPs with box constraints}
    \label{tab:QCQP.on.box}
    {\small\begin{itemize}
        \item POP size: $m=n$, $l=\lceil n/7\rceil $; Relaxation order: $k=2$; SDP size: $\omega=n+1$, $a^{\max}=3$.
    \end{itemize}}
\small
\begin{center}
   \begin{tabular}{|c|c|c|c|c|c|c|c|c|c|}
        \hline
        \multicolumn{2}{|c|}{POP size}&
        \multicolumn{2}{c|}{SDP size} & \multicolumn{2}{c|}{Mosek}&      \multicolumn{2}{c|}{CGAL}&      \multicolumn{2}{c|}{LMBM} \\ \hline
$n$&$l$ &$s^{\max}$ & $\zeta$&
\multicolumn{1}{c|}{val}& \multicolumn{1}{c|}{time}&
\multicolumn{1}{c|}{val}&
\multicolumn{1}{c|}{time}&
\multicolumn{1}{c|}{val}&
\multicolumn{1}{c|}{time}\\
\hline
10& 2  & 66 & 2003 & -1.8320 & 0.6 & -1.8321 & 3 & -1.9692 & 4\\
\hline
20& 3  & 231 & 21484 & -3.1797 & 175 & -3.1781 & 106 & -4.0216 & 29\\
\hline
30& 5  & 496 & 94241 & -2.2949 & 6850 & -2.2982 & 528 & -3.9900 & 152\\
\hline
40& 6  & 861 & 274947 & $-$ & $-$ & -3.8651 & 933 & -6.1379 & 298\\
\hline
50& 8  & 1326 & 640459 & $-$ & $-$ & -3.6267 & 6159 & -6.3651 & 1494\\
\hline
\end{tabular}    
\end{center}
\end{table}

\paragraph{Discussion:} We observe similar behaviors of the solvers as in Section \ref{sec:experiment.single.ball}. The important point to note here is that solving a QCQP with box constraints is less efficient than solving the same one with ball constraints.
This is because the efficiency of CGAL depends on the number of psd blocks involved in SDP.
For instance, when $n=50$, CGAL takes around 1000 seconds to solve the second-order moment relaxation of a QCQP with a ball constraint while it takes around 2100 seconds to solve this relaxation for a QCQP with box constraints.

\subsection{Randomly generated dense QCQPs with simplex constraints}
\label{sec:experiment.simplex}
\paragraph{Test problems:} We construct randomly generated dense QCQPs
as in Section \ref{sec:experiment.single.ball}, where the ball constraint is now replaced by simplex constraints. Namely, in Step 2 we take $g$ such that \eqref{eq:stand.simplex} holds with $L=R=1$. 
The numerical results are displayed in Table \ref{tab:quadratics.on.simplex} and \ref{tab:QCQP.on.simplex}. 
\begin{table}
    \caption{\small Numerical results for minimizing a dense quadratic polynomials on a simplex}
    \label{tab:quadratics.on.simplex}
    {\small\begin{itemize}
        \item POP size: $m=n+2$, $l=0$; Relaxation order: $k=2$; SDP size: $\omega=n+3$, $a^{\max}=5$.
    \end{itemize}}
\small
\begin{center}
   \begin{tabular}{|c|c|c|c|c|c|c|c|c|c|}
        \hline
        \multicolumn{1}{|c|}{POP size}&
        \multicolumn{2}{c|}{SDP size} & \multicolumn{2}{c|}{Mosek}&      \multicolumn{2}{c|}{CGAL}&      \multicolumn{2}{c|}{LMBM} \\ \hline
$n$ &$s^{\max}$ & $\zeta$&
\multicolumn{1}{c|}{val}& \multicolumn{1}{c|}{time}&
\multicolumn{1}{c|}{val}&
\multicolumn{1}{c|}{time}&
\multicolumn{1}{c|}{val}&
\multicolumn{1}{c|}{time}\\
\hline
10 &66 &2003 & -1.9954 & 0.3 &-1.9950 &7 &-2.2800& 27 \\
\hline
20 &231 &21253 & -1.5078 & 58 &-1.5055 &116 &-2.7237& 32 \\
\hline
30 &496 &92753 & -2.0537 & 2804 &-2.0480 &377 & -3.3114 & 917 \\
\hline
40 &861 &271503 & $-$ & $-$ &-2.3034 &950 & -4.0971 & 577 \\
\hline
50 &1326 &632503 & $-$ & $-$ & -1.8366 & 9539 & -4.0541  & 13700 \\
\hline
\end{tabular}    
\end{center}
\end{table}
\begin{table}
    \caption{\small Numerical results for randomly generated dense QCQPs with simplex constraints}
    \label{tab:QCQP.on.simplex}
    {\small\begin{itemize}
        \item POP size: $m=n+2$, $l=\lceil n/7\rceil $; Relaxation order: $k=2$; SDP size: $\omega=n+3$, $a^{\max}=5$.
    \end{itemize}}
\small
\begin{center}
   \begin{tabular}{|c|c|c|c|c|c|c|c|c|c|}
        \hline
        \multicolumn{2}{|c|}{POP size}&
        \multicolumn{2}{c|}{SDP size} & \multicolumn{2}{c|}{Mosek}&      \multicolumn{2}{c|}{CGAL}&      \multicolumn{2}{c|}{LMBM} \\ \hline
$n$&$l$ &$s^{\max}$ & $\zeta$&
\multicolumn{1}{c|}{val}& \multicolumn{1}{c|}{time}&
\multicolumn{1}{c|}{val}&
\multicolumn{1}{c|}{time}&
\multicolumn{1}{c|}{val}&
\multicolumn{1}{c|}{time}\\
\hline
10 & 2 & 66 & 2135 & -1.0605 & 0.4 & -1.0606 & 176 & -2.2338 & 2 \\
\hline
20 & 3 & 231 & 21946 & -1.6629 & 72 & -1.6628 & 512 & -3.3538 & 93 \\
\hline
30 & 5 & 496 & 95233 & -1.0091 & 6206 & -1.0249 & 1089 & -2.9425 & 100 \\
\hline
40 & 6 & 861 & 276669 & $-$ & $-$ & -0.3256 & 2314 & -2.9564 & 4431\\
\hline
50 & 8 & 1326 & 643111 & $-$ & $-$ & -1.4200 & 10035 & -5.4284 & 1310\\
\hline
\end{tabular}    
\end{center}
\end{table}

\paragraph{Discussion:} Again we observe a behavior of the solvers similar to that
in Section \ref{sec:experiment.single.ball}.
One can also see that solving a QCQP with simplex constraints by CGAL is significantly slower than solving the same one with box constraints. 
For instance, when $n=50$, CGAL takes 2100 seconds to solve the second-order moment relaxation for a QCQP with box constraints while it takes 9500 seconds with simplex constraints.

\subsection{Randomly generated QCQPs with TS and ball constraints}
\label{sec:experiment.single.ball.term}
\paragraph{Test problems:} 
We construct randomly generated QCQPs with TS and a ball constraint as follows:
\begin{enumerate}
    \item Generate a quadratic polynomial objective function $f$  such that for $\alpha\in\N^n_2$ with $|\alpha|\ne 2$, $f_\alpha=0$ and for $\alpha\in\N^n_2$ with $|\alpha|=2$, the coefficient $f_\alpha$ is randomly generated in $(-1,1)$ w.r.t. the uniform distribution;
    \item Take $m=1$ and $g_1:=1-\|\mathbf x\|_2^2$;
    \item Take a random point $\mathbf a$ in $S(g)$ w.r.t. the uniform distribution;
    \item For every $j\in[l]$, generate a quadratic polynomial $h_j$ by
    \begin{enumerate}[(i)]
        \item setting $h_{j,\alpha}=0$ for each $\alpha\in\N^n_2\backslash \{\mathbf 0\}$ with $|\alpha|\ne 2$;
        \item for each $\alpha\in\N^n_2\backslash \{\mathbf 0\}$ with $|\alpha|=2$, taking a random coefficient $h_{j,\alpha}$ for $h_j$ in $(-1,1)$ w.r.t. the uniform distribution;
        \item setting $h_{j,\mathbf 0}:=-\sum_{\alpha\in\N^n_2\backslash \{\mathbf 0\}} h_{j,\alpha} \mathbf a^\alpha$.
    \end{enumerate}
    Then $\mathbf a$ is a feasible solution of POP \eqref{eq:POP.def}.
\end{enumerate}

The numerical results are displayed in Table \ref{tab:quadratics.on.unit.ball.term} and \ref{tab:QCQP.on.unit.ball.term}. 
\begin{table}
    \caption{\small Numerical results for minimizing a random quadratic polynomial with TS on the unit ball}
    \label{tab:quadratics.on.unit.ball.term}
    {\small\begin{itemize}
        \item POP size: $m=1$, $l=0$;  Relaxation order: $k=2$;  Sparse order: $t=1$; SDP size: $\omega=4$, $a^{\max}=3$.
    \end{itemize}}
\small
\begin{center}
   \begin{tabular}{|c|c|c|c|c|c|c|c|c|c|c|}
        \hline
        \multicolumn{1}{|c|}{POP size}&
        \multicolumn{2}{c|}{SDP size} & \multicolumn{2}{c|}{Mosek}&      \multicolumn{2}{c|}{CGAL}&      \multicolumn{2}{c|}{LMBM}\\ \hline
$n$ &$s^{\max}$ & $\zeta$&
\multicolumn{1}{c|}{val}& \multicolumn{1}{c|}{time}&
\multicolumn{1}{c|}{val}&
\multicolumn{1}{c|}{time}&
\multicolumn{1}{c|}{val}&
\multicolumn{1}{c|}{time}\\
\hline
10 & 56 & 937 & -1.5681 & 4 & -1.5527 & 0.7 & -1.5711  & 0.07\\
\hline
20 & 211 & 13722 & -2.4275 & 36 & -2.3996 & 1 & -2.7301 & 0.6 \\
\hline
30 & 466 & 68357 & -3.0748 & 1930 & -3.0577 & 8 & -3.5188 & 8  \\
\hline
40 & 821 & 214842 & $-$ & $-$ & -3.6999 & 20 & -4.9033 & 40 \\
\hline
50 & 1276 & 523177 & $-$ & $-$ & -4.1603 & 128 & -5.3416 & 59 \\
\hline
60 & 1831 & 1083362 & $-$ & $-$ & -4.1914 & 655 & -5.6983 & 303 \\
\hline
70 & 2486 & 2005397 & $-$ & $-$ & -4.9578 & 1461 & -7.1968 & 1040 \\
\hline
80 & 3241 & 3419282 & $-$ & $-$ & -5.6452 & 7253 & -7.9133 & 5759 \\
\hline
\end{tabular}    
\end{center}
\end{table}
\begin{table}
    \caption{\small Numerical results for randomly generated QCQPs with TS and a ball constraint}
    \label{tab:QCQP.on.unit.ball.term}
    {\small\begin{itemize}
        \item POP size: $m=1$, $l=\lceil n/4\rceil $; Relaxation order: $k=2$;  Sparse order: $t=1$; SDP size: $\omega=4$, $a^{\max}=3$.
    \end{itemize}}
\small
\begin{center}
   \begin{tabular}{|c|c|c|c|c|c|c|c|c|c|}
        \hline
        \multicolumn{2}{|c|}{POP size}&
        \multicolumn{2}{c|}{SDP size} & \multicolumn{2}{c|}{Mosek}&      \multicolumn{2}{c|}{CGAL}& \multicolumn{2}{c|}{LMBM} \\ \hline
$n$&$l$ &$s^{\max}$ & $\zeta$&
\multicolumn{1}{c|}{val}& \multicolumn{1}{c|}{time}&
\multicolumn{1}{c|}{val}&
\multicolumn{1}{c|}{time}&
\multicolumn{1}{c|}{val}&
\multicolumn{1}{c|}{time}\\
\hline
10 &3 & 56 & 1105 & -0.60612 & 0.7 & -0.60550 & 2  & -0.60611 & 0.8\\
\hline
20 &5 & 211 & 14777 & -2.3115 & 47 & -2.3097 & 17& -2.3952 & 3 \\
\hline
30 &8 & 466 & 72085 & -2.8344 & 3102 &  -2.8321 & 112 &  -3.7588 & 128 \\
\hline
40 &10 & 821 & 223052 & $-$ & $-$ & -3.4081  & 476 & -4.4239  & 673 \\
\hline
50 &13 & 1276 & 539765 & $-$ & $-$ & -3.3552  & 1845 & -5.2568  & 729 \\
\hline
60 &15 & 1831 & 1110827 & $-$ & $-$ & -3.5620  & 2992 & -5.9898 & 1702 \\
\hline
\end{tabular}    
\end{center}
\end{table}
\paragraph{Discussion:}
The behavior of solvers is  similar to that in the dense case.

\subsection{Randomly generated QCQPs with TS and box constraints}
\label{sec:experiment.box.term}
\paragraph{Test problems:} We construct randomly generated QCQPs with TS as in Section \ref{sec:experiment.single.ball.term}, where the ball constraint is now replaced by box constraints. 
The numerical results are displayed in Table \ref{tab:quadratics.on.box.term} and \ref{tab:QCQP.on.box.term}. 
\begin{table}
    \caption{\small Numerical results for minimizing a random quadratic polynomial with TS on a box}
    \label{tab:quadratics.on.box.term}
    {\small\begin{itemize}
        \item POP size: $m=n$, $l=0$; Relaxation order: $k=2$;  Sparse order: $t=1$; SDP size: $a^{\max}=3$.
    \end{itemize}}
\small
\begin{center}
   \begin{tabular}{|c|c|c|c|c|c|c|c|c|c|}
        \hline
        \multicolumn{1}{|c|}{POP size} &
        \multicolumn{3}{c|}{SDP size} & \multicolumn{2}{c|}{Mosek}&      \multicolumn{2}{c|}{CGAL}&      \multicolumn{2}{c|}{LMBM} \\ \hline
$n$ &$\omega$&$s^{\max}$ & $\zeta$&
\multicolumn{1}{c|}{val}& \multicolumn{1}{c|}{time}&
\multicolumn{1}{c|}{val}&
\multicolumn{1}{c|}{time}&
\multicolumn{1}{c|}{val}&
\multicolumn{1}{c|}{time}\\
\hline
10 & 22 & 56 &  1441 & -1.0539 & 3 & -1.0519 & 14 & -1.11671 & 1 \\
\hline
20 & 42 & 211 &  17731 & -1.3925 & 93 & -1.3802 & 161 & -2.2978 & 2   \\
\hline
30 & 62 & 466 &  81871 & -2.2301 & 4392 & -2.2128 & 567 & -2.4544 &  533 \\
\hline
40 & 82 & 821 &  246861 & $-$ & $-$ & -2.5209 &  1602 & -4.6159 & 1036\\
\hline
50 & 102 & 1276 &  585701 & $-$ & $-$ & -3.0282 &  2583 & -4.9146 & 376\\
\hline
60 & 122 & 1831 &  1191391 & $-$ & $-$ & -3.0470 & 10858  & -5.7882 & 353\\
\hline
\end{tabular}    
\end{center}
\end{table}
\begin{table}
    \caption{\small Numerical results for randomly generated QCQPs with TS and box constraints}
    \label{tab:QCQP.on.box.term}
    {\small\begin{itemize}
        \item POP size: $m=n$, $l=\lceil n/7\rceil $; Relaxation order: $k=2$;  Sparse order: $t=1$.; SDP size: $\omega=n+1$, $a^{\max}=3$.
    \end{itemize}}
\small
\begin{center}
   \begin{tabular}{|c|c|c|c|c|c|c|c|c|c|c|}
        \hline
        \multicolumn{2}{|c|}{POP size}&
        \multicolumn{3}{c|}{SDP size} & \multicolumn{2}{c|}{Mosek}&      \multicolumn{2}{c|}{CGAL}&      \multicolumn{2}{c|}{LMBM} \\ \hline
$n$&$l$&$\omega$ &$s^{\max}$ & $\zeta$&
\multicolumn{1}{c|}{val}& \multicolumn{1}{c|}{time}&
\multicolumn{1}{c|}{val}&
\multicolumn{1}{c|}{time}&
\multicolumn{1}{c|}{val}&
\multicolumn{1}{c|}{time}\\
\hline
10&2 & 22&56 & 1553 & -0.77189 & 0.2 & -0.77214& 9 & -0.78092 & 1\\
\hline
20&3 & 42& 211 & 18364 & -1.7962 & 71 & -1.8009 & 150 & -2.7771 & 4\\
\hline
30&5 & 62& 466 & 84201 & -1.8529 & 5814 & -1.8625 & 650 & -3.5891 & 268 \\
\hline
40&6 & 82& 821 & 251787 & $-$ & $-$ & -2.1930 & 2994 & -4.5890 & 317\\
\hline
50&8 & 102& 1276 & 595909 & $-$ & $-$ & -2.4655 &  8397 & -5.1811 & 883 \\
\hline
\end{tabular}    
\end{center}
\end{table}
\paragraph{Discussion:}
Again the behavior of solvers is similar to that  in the dense case.

\subsection{Randomly generated QCQPs with CS and ball constraints on each clique of variables}
\label{sec:experiment.single.ball.sparse}

\paragraph{Test problems:} We construct randomly generated QCQPs with CS and ball constraints on each clique of variables as follows:
\begin{enumerate}
    \item Take a positive integer $u$, $p:=\lfloor n/u\rfloor +1$ and let
    \begin{equation}\label{eq:index.vars}
        I_j=\begin{cases}
        [u],&\text{if }j=1\,,\\
        \{u(j-1),\dots,uj\},&\text{if }j\in\{2,\dots,p-1\}\,,\\
        \{u(p-1),\dots,n\},&\text{if }j=p\,;
        \end{cases}
    \end{equation}
    \item Generate a quadratic polynomial objective function $f=\sum_{j\in[p]}f_j$ such that for each $j\in[p]$, $f_j\in\R[\mathbf x(I_j)]_2$, and the coefficient $f_{j,\alpha},\alpha\in\N^{I_j}_2$ of $f_j$ is randomly generated in $(-1,1)$ w.r.t. the uniform distribution;
    \item Take $m=p$ and $g_j:=-\|\mathbf x(I_j)\|_2^2+1$, $j\in[m]$;
    \item Take a random point $\mathbf a$ in $S(g)$ w.r.t. the uniform distribution;
    \item Let $r:=\lfloor l/p\rfloor$ and \begin{equation}\label{eq:assign.eqcons}
        W_j:=\begin{cases} \{(j-1)r+1,\dots,jr\},&\text{if }j\in[p-1]\,,\\
    \{(p-1)r+1,\dots,l\},&\text{if } j=p\,.
    \end{cases}
    \end{equation}
     For every $j\in[p]$ and every $i\in W_j$, generate a quadratic polynomial $h_i\in\R[\mathbf x(I_j)]_2$ by
    \begin{enumerate}
        \item for each $\alpha\in\N^{I_j}_2\backslash \{\mathbf 0\}$, taking a random coefficient $h_{i,\alpha}$ of $h_i$ in $(-1,1)$ w.r.t. the uniform distribution;
        \item setting $h_{i,\mathbf 0}:=-\sum_{\alpha\in\N^{I_j}_2\backslash \{\mathbf 0\}} h_{j,\alpha} \mathbf a^\alpha$.
    \end{enumerate}
    Then $\mathbf a$ is a feasible solution of POP \eqref{eq:POP.def}.
\end{enumerate}

The numerical results are displayed in Table \ref{tab:experiment.single.ball.sparse} and \ref{tab:single.ball.sparse.QCQP}. 
\begin{table}
    \caption{\small Numerical results for minimizing a random  quadratic polynomial with CS and ball constraints on each clique of variables}
    \label{tab:experiment.single.ball.sparse}
    {\small\begin{itemize}
        \item POP size: $n=1000$, $m=p$, $l=0$, $u^{\max}=u+1$;  Relaxation order: $k=2$;
        SDP size: $\omega=2p$, $a^{\max}=3$.
    \end{itemize}}
\small
\begin{center}
   \begin{tabular}{|c|c|c|c|c|c|c|c|c|c|c|}
        \hline
        \multicolumn{2}{|c|}{POP size}&
        \multicolumn{3}{c|}{SDP size} & \multicolumn{2}{c|}{Mosek}&      \multicolumn{2}{c|}{CGAL}&      \multicolumn{2}{c|}{LMBM} \\ \hline
$u$&$p$ &$\omega$&$s^{\max}$ & $\zeta$&
\multicolumn{1}{c|}{val}& \multicolumn{1}{c|}{time}&
\multicolumn{1}{c|}{val}&
\multicolumn{1}{c|}{time}&
\multicolumn{1}{c|}{val}&
\multicolumn{1}{c|}{time}\\
\hline
 11 & 91& 182 & 91 & 222712 & -240.54 & 124 & -240.37 & 98 & -508.35 & 15\\
\hline
16 & 63& 126 & 171 & 550692 & -205.45 & 1389 & -205.19 & 280 & -429.93 & 83\\
\hline
 21 & 48 & 96 & 276 & 1107682 & $-$ & $-$ & -175.60 & 321 & -365.91 & 269\\
\hline
 26 & 39 & 78 & 406 & 1955879 & $-$ & $-$ & -165.65 & 559 & -338.00 & 225 \\
\hline
 31 & 33 & 66 & 561 & 3167072 & $-$ & $-$ & -149.10 & 973 & -305.33 & 5280 \\
\hline
 36 & 28 & 56 & 741 & 4758727 & $-$ & $-$ & -140.21 & 1315 & -285.69 & 737 \\
\hline
41 & 25 & 50 & 946 & 6839993 & $-$ & $-$ & -126.55 & 1926 & -265.28 & 622 \\
\hline
\end{tabular}    
\end{center}
\end{table}
\begin{table}
    \caption{\small Numerical results for randomly generated QCQPs with CS and ball constraints on each clique of variables}
    \label{tab:single.ball.sparse.QCQP}
    {\small\begin{itemize}
        \item POP size: $n=1000$, $m=p$, $l=143$, $u^{\max}=u+1$;  Relaxation order: $k=2$; SDP size: $\omega=2p$, $a^{\max}=3$.
    \end{itemize}}
\small
\begin{center}
   \begin{tabular}{|c|c|c|c|c|c|c|c|c|c|c|}
        \hline
        \multicolumn{2}{|c|}{POP size}&
        \multicolumn{3}{c|}{SDP size} & \multicolumn{2}{c|}{Mosek}&      \multicolumn{2}{c|}{CGAL}&      \multicolumn{2}{c|}{LMBM} \\ \hline
$u$&$p$ &$\omega$&$s^{\max}$ & $\zeta$&
\multicolumn{1}{c|}{val}& \multicolumn{1}{c|}{time}&
\multicolumn{1}{c|}{val}&
\multicolumn{1}{c|}{time}&
\multicolumn{1}{c|}{val}&
\multicolumn{1}{c|}{time}\\
\hline
11 & 91 & 182 & 91 & 235023 & -224.15 & 163 & -224.09 & 204 & -500.34 & 13 \\
\hline
16 & 63 & 126 & 171 & 572905 & -192.45 & 1830 & -192.30 & 335 & -420.87 & 50 \\
\hline
21 & 48 & 96 & 276 & 1139460 & $-$ & $-$ & -162.79 & 537 & -363.28 & 103 \\
\hline
26 & 39 & 78 & 406 & 2005124 & $-$ & $-$ & -148.77 & 1014 & -336.42 & 263 \\
\hline
31 & 33 & 66 & 561 & 3239573 & $-$ & $-$ & -142.38 & 2115 & -313.80 & 3679 \\
\hline
36 & 28 & 56 & 741 & 4862292 & $-$ & $-$ & -124.97 & 5304 & -263.77 & 6598 \\
\hline
\end{tabular}    
\end{center}
\end{table}

\paragraph{Discussion:}
The number of variables is fixed as $n=1000$. 
We increase the clique size $u$ so that the number of variable cliques $p$ decreases accordingly.
Again results in Table \ref{tab:experiment.single.ball.sparse} and \ref{tab:single.ball.sparse.QCQP} show that CGAL is the fastest solver 
and returns an optimal value of gap within 1\% w.r.t. the one returned by Mosek (for $u\le 16$). Moreover Mosek runs out of memory when $u\ge 21$, and LMBM fails to converge to the optimal value.

\subsection{Randomly generated QCQPs with CS and box constraints on each clique of variables}
\label{sec:experiment.single.box.sparse}

\paragraph{Test problems:} We construct randomly generated QCQPs
with CS as in Section \ref{sec:experiment.single.ball.sparse}, where ball constraints are now replaced by box constraints. Namely, in Step 3 we take $m=n$, $g_j:=-x_j^2+1/u$, $j\in [n]$. 

The numerical results are displayed in Table  \ref{tab:experiment.single.box.sparse} and \ref{tab:single.box.sparse.QCQP}. 
\begin{table}
    \caption{\small Numerical results for minimizing a random  quadratic polynomial with CS and box constraints on each clique of variables}
    \label{tab:experiment.single.box.sparse}
    {\small\begin{itemize}
        \item POP size: $n=m=1000$, $l=0$,  $u^{\max}=u+1$;   Relaxation order: $k=2$; Constant trace: $a^{\max}\in [3,4]$.
    \end{itemize}}
\small
\begin{center}
   \begin{tabular}{|c|c|c|c|c|c|c|c|c|c|c|}
        \hline
        \multicolumn{2}{|c|}{POP size}&
        \multicolumn{3}{c|}{SDP size} & \multicolumn{2}{c|}{Mosek}&      \multicolumn{2}{c|}{CGAL}&      \multicolumn{2}{c|}{LMBM} \\ \hline
$u$&$p$ &$\omega$&$s^{\max}$ & $\zeta$& 
\multicolumn{1}{c|}{val}& \multicolumn{1}{c|}{time}&
\multicolumn{1}{c|}{val}&
\multicolumn{1}{c|}{time}&
\multicolumn{1}{c|}{val}&
\multicolumn{1}{c|}{time}\\
\hline
11 & 91 & 1181 & 91 & 313361 &  -204.89 & 443 & -204.69 & 753 & -555.51 & 223 \\
\hline
16 & 63 & 1125 & 171 & 720323 & -163.11 & 3082 & -162.88 & 3059 & -438.22 & 119 \\
\hline
21 & 48 & 1095 & 276 & 1380918  & $-$ & $-$ & -147.92 & 5655 & -387.42 & 2213 \\
\hline
26 & 39 & 1077 & 406 & 2357161  & $-$ & $-$ & -131.00 & 8889 & -340.04 & 5346 \\
\hline
\end{tabular}    
\end{center}
\end{table}
\begin{table}
    \caption{\small Numerical results for QCQPs with CS and box constraints on each clique of variables}
    \label{tab:single.box.sparse.QCQP}
    {\small\begin{itemize}
        \item POP size: $n=m=1000$, $l=143$, $u^{\max}=u+1$; Relaxation order: $k=2$;
         Constant trace: $a^{\max}\in [3,4]$.
    \end{itemize}}
\small
\begin{center}
   \begin{tabular}{|c|c|c|c|c|c|c|c|c|c|c|}
        \hline
        \multicolumn{2}{|c|}{POP size} &
        \multicolumn{3}{c|}{SDP size} & \multicolumn{2}{c|}{Mosek}&      \multicolumn{2}{c|}{CGAL}&      \multicolumn{2}{c|}{LMBM} \\ \hline
$u$&$p$&$\omega$&$s^{\max}$ & $\zeta$&
\multicolumn{1}{c|}{val}& \multicolumn{1}{c|}{time}&
\multicolumn{1}{c|}{val}&
\multicolumn{1}{c|}{time}&
\multicolumn{1}{c|}{val}&
\multicolumn{1}{c|}{time}\\
\hline
11 &  91 & 1181 & 91 & 325672 &  -187.01  & 402 &-186.98 & 1915 & -570.68 & 184 \\
\hline
16 & 63 & 1125 & 171 & 742536  & -142.16  & 4323  & -142.27 & 4126 & -442.51 &57 \\
\hline
21 & 48 & 1095 & 276 & 1412696  & $-$  & $-$  & -131.14 & 5334 & -382.89 & 618 \\
\hline
26 & 39 & 1077 & 406 & 2406406  & $-$  & $-$  & -113.44 & 8037 &  -336.11 & 901 \\
\hline
\end{tabular}    
\end{center}
\end{table}

\paragraph{Discussion:}
The number of variables is fixed as $n=1000$.
We increase the clique size $u$ so that the number of variable cliques $p$ decreases accordingly.
From results in Table \ref{tab:experiment.single.ball.sparse} and \ref{tab:single.ball.sparse.QCQP}, one observes that
when the largest size of variable cliques is relatively small (say $u\le 11$), Mosek is the fastest solver. However when the largest size of variable cliques is relatively large (say $u\le 21$), Mosek runs out of memory while CGAL still works well. 


\subsection{Randomly generated QCQPs with CS-TS and ball constraints on each clique of variables}
\label{sec:experiment.single.ball.sparse.mix}
\paragraph{Test problems:} 
We construct randomly generated QCQPs with CS-TS and ball constraints on each clique of variables as follows:
\begin{enumerate}
    \item Take a positive integer $u$, $p:=\lfloor n/u\rfloor +1$ and let $(I_j)_{j\in[p]}$ be defined as in \eqref{eq:index.vars}.
    \item Generate a quadratic polynomial objective function $f=\sum_{j\in[p]}f_j$ such that for each $j\in[p]$, $f_j\in\R[\mathbf x(I_j)]_2$ and the nonzero coefficient $f_{j,\alpha}$ with $\alpha\in\N^{I_j}_2$ and $|\alpha|= 2$ is randomly generated
    in $(-1,1)$ w.r.t. the uniform distribution;
    \item Take $m=p$ and $g_j:=-\|\mathbf x(I_j)\|_2^2+1$, $j\in[m]$;
    \item Take a random point $\mathbf a$ in $S(g)$ w.r.t. the uniform distribution;
    \item Let $r:=\lfloor l/p\rfloor$ and $(W_j)_{j\in[p]}$ be as in \eqref{eq:assign.eqcons}.
     For every $j\in[p]$ and every $i\in W_j$, generate a quadratic polynomial $h_i\in\R[\mathbf x(I_j)]_2$ by
    \begin{enumerate}
        \item for each $\alpha\in\N^{I_j}_2\backslash\{\mathbf 0\}$ with $|\alpha|\ne2$, taking $h_{i,\alpha}=0$;
        \item for each $\alpha\in\N^{I_j}_2$ with $|\alpha|=2$, taking a random coefficient $h_{i,\alpha}$ of $h_i$ in $(-1,1)$ w.r.t. the uniform distribution;
        \item setting $h_{i,\mathbf 0}:=-\sum_{\alpha\in\N^{I_j}_2\backslash \{\mathbf 0\}} h_{j,\alpha} \mathbf a^\alpha$.
    \end{enumerate}
    Then $\mathbf a$ is a feasible solution of POP \eqref{eq:POP.def}.
\end{enumerate}

The numerical results are displayed in Table \ref{tab:experiment.single.ball.sparse.mix} and \ref{tab:single.ball.sparse.QCQP.mix}. 
\begin{table}
    \caption{\small Numerical results for minimizing a random quadratic polynomial with CS-TS and ball constraints on each clique of variables}
    \label{tab:experiment.single.ball.sparse.mix}
    {\small\begin{itemize}
        \item POP size: $n=1000$, $m=p$, $l=0$, $u^{\max}=u+1$;  Relaxation order: $k=2$;
        Sparse order: $t=1$;
        SDP size: $a^{\max}=3$.
    \end{itemize}}
\small
\begin{center}
   \begin{tabular}{|c|c|c|c|c|c|c|c|c|c|c|}
        \hline
        \multicolumn{2}{|c|}{POP size}&
        \multicolumn{3}{c|}{SDP size} & \multicolumn{2}{c|}{Mosek}&      \multicolumn{2}{c|}{CGAL}&      \multicolumn{2}{c|}{LMBM} \\ \hline
$u$&$p$ &$\omega$&$s^{\max}$ & $\zeta$&
\multicolumn{1}{c|}{val}& \multicolumn{1}{c|}{time}&
\multicolumn{1}{c|}{val}&
\multicolumn{1}{c|}{time}&
\multicolumn{1}{c|}{val}&
\multicolumn{1}{c|}{time}\\
\hline
 11 &91 & 364 & 79 & 169654  & -160.05 & 163 & -160.01 & 498 & -489.87 & 98 \\
 \hline
 16 & 63 & 252 & 154 & 448354  & -135.78 & 1422 & -135.74  & 768 & -413.24 & 186 \\
\hline
 21 &48 & 192 & 254 & 939619  & $-$  & $-$ & -117.17 & 1605 & -351.65 & 299 \\
\hline
 26 &39 & 156 & 379 & 1705763  & $-$  & $-$ & -106.26& 3150 & -318.15 & 347 \\
\hline
\end{tabular}    
\end{center}
\end{table}
\begin{table}
    \caption{\small Numerical results for QCQPs with CS-TS and ball constraints on each clique of variables}
    \label{tab:single.ball.sparse.QCQP.mix}
    {\small\begin{itemize}
        \item POP size: $n=1000$, $m=p$, $l=143$, $u^{\max}=u+1$;  Relaxation order: $k=2$;
        Sparse order: $t=1$;
        SDP size: $a^{\max}=3$.
    \end{itemize}}
\small
\begin{center}
   \begin{tabular}{|c|c|c|c|c|c|c|c|c|c|c|}
        \hline
        \multicolumn{2}{|c|}{POP size}&
        \multicolumn{3}{c|}{SDP size} & \multicolumn{2}{c|}{Mosek}&      \multicolumn{2}{c|}{CGAL} &      \multicolumn{2}{c|}{LMBM} \\ \hline
$u$&$p$ &$\omega$&$s^{\max}$ & $\zeta$&
\multicolumn{1}{c|}{val}& \multicolumn{1}{c|}{time}&
\multicolumn{1}{c|}{val}&
\multicolumn{1}{c|}{time}&
\multicolumn{1}{c|}{val}&
\multicolumn{1}{c|}{time}\\
\hline
11 & 91 & 364  & 79 & 180303 & -155.91 & 158 & -155.87 & 604 & -500.47 & 83\\
\hline
16 & 63 & 252  & 154 & 468290 & 127.42 & 1707 & -127.36 & 1053 & -412.42 & 236 \\
\hline
21 & 48 & 192  & 254 & 939619 & $-$ & $-$ & -114.85 & 2877 & -363.23 & 128\\
\hline
26 & 39 & 156  & 379 & 1751556 & $-$ & $-$ & -102.30 & 6878 & -329.16 & 749 \\
\hline
\end{tabular}    
\end{center}
\end{table}

\paragraph{Discussion:} The behavior of solvers is similar to that in Section \ref{sec:experiment.single.box.sparse}. Here, we also emphasize that our framework is less efficient than interior-point methods for most benchmarks presented in \cite{cstssos}.
The two underlying reasons are that (1) the block size of the resulting SDP relaxations is small, in which case Mosek performs more efficiently, e.g., for the benchmarks from \cite[Section 5.2]{cstssos}, and (2) it is harder to find the constant trace, e.g., for the benchmarks from \cite[Section 5.4]{cstssos}. 
Thus our proposed method complements that in \cite{cstssos} when the block size of the SDP relaxations is large and/or when CTP can be efficiently verified.

\subsection{Randomly generated QCQPs with CS-TS and box constraints on each clique of variables}
\label{sec:experiment.single.box.sparse.mix}

\paragraph{Test problems:} We construct randomly generated QCQPs
with CS-TS as in Section \ref{sec:experiment.single.ball.sparse.mix}, where ball constraints are now replaced by box constraints. Namely, in Step 3 we take $m=n$, $g_j:=-x_j^2+1/u$, $j\in [n]$. 
The numerical results are displayed in Table  \ref{tab:experiment.single.box.sparse.mix} and \ref{tab:single.box.sparse.QCQP.mix}. 
\begin{table}
    \caption{\small Numerical results for minimizing a random quadratic polynomial with CS-TS and box constraints on each clique of variables}
    \label{tab:experiment.single.box.sparse.mix}
    {\small\begin{itemize}
        \item POP size: $n=m=1000$, $l=0$,  $u^{\max}=u+1$; Relaxation order: $k=2$;
        Sparse order: $t=1$;
        Constant trace: $a^{\max}\in [3,4]$.
    \end{itemize}}
\small
\begin{center}
   \begin{tabular}{|c|c|c|c|c|c|c|c|c|c|c|}
        \hline
        \multicolumn{2}{|c|}{POP size}&
        \multicolumn{3}{c|}{SDP size} & \multicolumn{2}{c|}{Mosek}&      \multicolumn{2}{c|}{CGAL}&      \multicolumn{2}{c|}{LMBM} \\ \hline
$u$&$p$ &$\omega$&$s^{\max}$ & $\zeta$&  
\multicolumn{1}{c|}{val}& \multicolumn{1}{c|}{time}&
\multicolumn{1}{c|}{val}&
\multicolumn{1}{c|}{time}&
\multicolumn{1}{c|}{val}&
\multicolumn{1}{c|}{time}\\
\hline
 11 & 91 & 2362 & 79 & 248335 & -126.15 & 151 &-126.04 & 1982 &  -541.78 & 140 \\
\hline
16 & 63 & 2250 & 154 & 601081 & -100.75 & 2225 &-100.64 & 7323  & -429.78 &  88   \\
\hline
21 & 48 & 2190 & 254 & 1191001  & $-$ & $-$ & -87.804 & 10734 & -363.88 & 157  \\
\hline
26 & 39 & 2154 & 379 & 2080265  & $-$ & $-$ & -81.908 & 20294 & -338.00 & 1129  \\
\hline
\end{tabular}    
\end{center}
\end{table}
\begin{table}
    \caption{\small Numerical results for QCQPs with CS-TS and box constraints on each clique of variables}
    \label{tab:single.box.sparse.QCQP.mix}
    {\small\begin{itemize}
        \item POP size: $n=m=1000$, $l=143$, $u^{\max}=u+1$;  Relaxation order: $k=2$;
         Sparse order: $t=1$;
        Constant trace: $a^{\max}\in [3,4]$.
    \end{itemize}}
\small
\begin{center}
   \begin{tabular}{|c|c|c|c|c|c|c|c|c|c|c|}
        \hline
        \multicolumn{2}{|c|}{POP size}&
        \multicolumn{3}{c|}{SDP size} & \multicolumn{2}{c|}{Mosek}&     \multicolumn{2}{c|}{CGAL}&     \multicolumn{2}{c|}{LMBM} \\ \hline
$u$&$p$ &$\omega$&$s^{\max}$ & $\zeta$& 
\multicolumn{1}{c|}{val}& \multicolumn{1}{c|}{time}&
\multicolumn{1}{c|}{val}&
\multicolumn{1}{c|}{time}&
\multicolumn{1}{c|}{val}&
\multicolumn{1}{c|}{time}\\
\hline
11 & 91 & 2362 & 79 & 258984 & -114.53 & 325 & -114.27 & 482 & -529.32 & 226 
\\ \hline
16 & 63 & 2250 & 154 & 621017 & -96.199 & 4450 & -96.079 & 1245 & -433.34 &  519
\\ \hline
21 & 48 & 2190 & 254 &1220027 & $-$ & $-$& -83.013 & 8204 & -372.97 & 554
\\ \hline
26 & 39 & 2154 & 379 &2126058 & $-$ & $-$& -74.532 & 27600 & -258.90 & 764
\\ \hline
\end{tabular}    
\end{center}
\end{table}

\paragraph{Discussion:} The behavior of solvers is similar to that  in Section \ref{sec:experiment.single.box.sparse}.

\section{Conclusion}
In this paper, we have proposed a general framework 
for exploiting the constant trace property (CTP) in solving large-scale
SDPs, typically SDP-relaxations from the Moment-SOS hierarchy applied to  POPs.
Extensive numerical experiments strongly suggest that with this CTP formulation, the 
CGAL solver based on first-order methods
is more efficient and more scalable than Mosek (based on IPM) without exploiting CTP, especially when the block size is large.
In addition, the optimal value returned by CGAL is typically within 1\% w.r.t. the one returned by Mosek. 

We have also integrated sparsity-exploiting techniques into the CTP framework 
in order to handle larger size POPs. For SDP-relaxations of large-scale POPs with a term- and/or correlative-sparsity pattern, and in applications 
for which only a medium accuracy of optimal solutions is enough, we believe that our framework should be very  useful.

As a topic of further investigation, we would like to improve the LP-based formulation for verifying CTP, for instance by relying on more general second-order cone programming.
We also would like to generalize the CTP-exploiting framework to  noncommutative POPs \cite{burgdorf16,klep2019sparse,wang2020exploiting} which have attracted a lot of attention in the quantum information community.
Another line of research would be to investigate whether CTP could be efficiently exploited by interior-point solvers.

\paragraph{\textbf{Acknowledgements}.} 
The first author was supported by the MESRI funding from EDMITT.
The second author was supported by the FMJH Program PGMO (EPICS project) and  EDF, Thales, Orange et Criteo, as well as from the Tremplin ERC Stg Grant ANR-18-ERC2-0004-01 (T-COPS project).
This work has benefited from the Tremplin ERC Stg Grant ANR-18-ERC2-0004-01 (T-COPS project), the European Union's Horizon 2020 research and innovation programme under the Marie Sklodowska-Curie Actions, grant agreement 813211 (POEMA) as well as from the AI Interdisciplinary Institute ANITI funding, through the French ``Investing for the Future PIA3'' program under the Grant agreement n$^{\circ}$ANR-19-PI3A-0004.
The third and fourth authors were supported by the European Research Council (ERC) under the European's Union Horizon 2020 research and innovation program (grant agreement 666981 TAMING).

\appendix
\section{Appendix}
\label{sec:Appendix}
%
\subsection{Sparse POPs}
\label{sec:sparse.POP.TS.CSTS}

For matrices $\mathbf A$ and $\mathbf B$ of same sizes,  the Hadamard product of $\mathbf A$ and $\mathbf B$, denoted by $\mathbf A\circ \mathbf B$, is the matrix with entries $[\mathbf A\circ \mathbf B]_{i,j}=A_{i,j}B_{i,j}$.

\subsubsection{Term sparsity (TS)}
\label{sec:TS}
Fix a relaxation order $k\in\N^{\ge k_{\min}}$ and a  sparse order $t\in\N\backslash\{0\}$. 
We compute as in \cite[Section 5]{wang2019tssos} the following block diagonal (up to permutation) $(0,1)$-binary matrices:
$\mathbf G_{k,t}^{(0)}$ of size $\s(k)$;
    $\mathbf G_{k,t}^{(i)}$ of size $\s(k-\lceil g_i\rceil)$, $i\in [m]$;
$\mathbf H_{k,t}^{(i)}$ of size $\s(k-\lceil h_i\rceil)$, $i\in [l]$.
Then we consider the following sparse moment relaxation of POP \eqref{eq:POP.def}:
\begin{equation}\label{eq:ts.SDP}
\tau_{k,t}^{\text{ts}} \,:= \,\inf \limits_{\mathbf y \in {\R^{\s({2k})} }} \left\{ L_{\mathbf y}(f)\ \left| \begin{array}{rl}
& \mathbf G_{k,t}^{(0)}\circ \mathbf M_k(\mathbf y) \succeq 0\,,\,y_{\mathbf 0}\,=\,1\,,\\
&\mathbf G_{k,t}^{(i)}\circ \mathbf M_{k - \lceil g_i \rceil }(g_i\;\mathbf y)   \succeq 0\,,\,i\in[m]\,,\\
&\mathbf H_{k,t}^{(i)}\circ \mathbf M_{k - \lceil h_i \rceil }(h_i\;\mathbf y)   = 0\,,\,i\in[l]
\end{array}
\right. \right\}\,.
\end{equation}
One has
$\tau_{k,t-1}^{\text{ts}}\le
    \tau_{k,t}^{\text{ts}}\le \tau_{k}\le f^\star$, for all $(k,t)$.
Moreover, we have the following theorem.
\begin{theorem}
(Wang et al. \cite[Theorem 5.1]{wang2019tssos}) For each $k\in\N^{\ge k_{\min}}$, the sequence $(\tau_{k,t}^{\text{ts}})_{t\in\N\backslash\{0\}}$ converges to $\tau_{k}$ (the optimal value of SDP \eqref{eq:moment.hierarchy}) in finitely many steps.
\end{theorem}
The dual of \eqref{eq:ts.SDP} reads as:
\begin{equation}\label{eq:dual.ts.SDP}
\rho_{k,t}^{\text{ts}} = \sup \limits_{\xi,\mathbf Q_i,\mathbf U_i} \left\{ \xi\ \left| \begin{array}{rl}
&\bar{\mathbf Q}_i=\mathbf G_{k,t}^{(i)}\circ{\mathbf Q}_i \succeq 0\,,\,i\in\{0\}\cup [m]\,,\\
&\bar{\mathbf U}_i=\mathbf H_{k,t}^{(i)}\circ{\mathbf U}_i\,,\,i\in [l]\,,\\
&f-\xi=\mathbf v_k^\top \bar{\mathbf Q}_0 \mathbf v_k+\sum_{i\in [m]} g_i \mathbf v_{k-\lceil g_i\rceil}^\top \bar{\mathbf Q}_i \mathbf v_{k-\lceil g_i\rceil}\\
&\qquad\qquad\qquad\qquad+\sum_{i\in [l]} h_i \mathbf v_{k-\lceil h_i\rceil}^\top \bar{\mathbf U}_i \mathbf v_{k-\lceil h_i\rceil}
\end{array}
\right. \right\}\,.
\end{equation}

\subsubsection{Correlative-Term sparsity (CS-TS)}
\label{sec:CS-TS}
The basic idea of correlative-term sparsity is to exploit term sparsity for each clique. 
The clique structure of the initial set of variables is derived from correlative sparsity (Section \ref{sec:sparse.POP}).

Fix a relaxation order $k\in\N^{\ge k_{\min}}$.
For every sparse order $t\in\N\backslash\{0\}$ and for every $j\in [p]$, we compute the following block diagonal (up to permutation) $(0,1)-$binary matrices (see \cite{wang2020cs}):
$\mathbf G_{k,t,j}^{(0)}$ of size $\s(n_j,k)$;
 $\mathbf G_{k,t,j}^{(i)}$ of size $\s(n_j,k-\lceil g_i\rceil)$, $i\in J_j$; 
 $\mathbf H_{k,t,j}^{(i)}$ of size $\s(n_j,k-\lceil h_i\rceil)$, $i\in W_j$. Then let us consider the following CS-TS moment relaxation:
\begin{equation}\label{eq:cs-ts.SDP}
\tau_{k,t}^{\text{cs-ts}} \,:= \,\inf \limits_{\mathbf y \in {\R^{\s({2k})} }} \left\{ L_{\mathbf y}(f)\ \left| \begin{array}{rl}
& \mathbf G_{k,t,j}^{(0)}\circ \mathbf M_k(\mathbf y, I_j) \succeq 0\,,\,j\in[p]\,,\,y_{\mathbf 0}\,=\,1\,,\\
&\mathbf G_{k,t,j}^{(i)}\circ \mathbf M_{k - \lceil g_i \rceil }(g_i\;\mathbf y,I_j)   \succeq 0\,,\,i\in J_j\,,\,j\in[p]\,,\\
&\mathbf H_{k,t,j}^{(i)}\circ \mathbf M_{k - \lceil h_i \rceil }(h_i\;\mathbf y,I_j)   = 0\,,\,i\in W_j\,,\,j\in[p]
\end{array}\right.\right\}\,.
\end{equation}
One has
$\tau_{k,t-1}^{\text{cs-ts}}\le
    \tau_{k,t}^{\text{cs-ts}}\le
    \tau_{k}^{\text{cs}}\le \tau_{k}\le f^\star$, for all $(k,t)$.
Moreover, we have the following theorem.
\begin{theorem}
(Wang et al. \cite{wang2020cs}) 
For each $k\in\N^{\ge k_{\min}}$, the sequence $(\tau_{k,t}^{\text{cs-ts}})_{t\in\N\backslash\{0\}}$ converges to $\tau_{k}^{\text{cs}}$ (the optimal value of SDP \eqref{eq:dual.cs.SDP}) in finitely many steps.
\end{theorem}
The dual of \eqref{eq:cs-ts.SDP} reads as:
\begin{equation}\label{eq:dual.cs-ts.SDP}
\rho_{k,t}^{\text{cs-ts}} = \sup \limits_{\xi,\mathbf Q_i^{j},\mathbf U_i^{j}} \left\{ \xi\ \left|\begin{array}{rl}
&\bar{\mathbf Q}_i^{(j)}=\mathbf G_{k,t,j}^{(i)}\circ{\mathbf Q}_i^{(j)} \succeq 0\,,\,i\in\{0\}\cup J_j\,,\,j\in [p]\,,\\
&\bar{\mathbf U}_i^{(j)}=\mathbf H_{k,t,j}^{(i)}\circ{\mathbf U}_i^{(j)}\,,\,i\in W_j\,,\,j\in [p]\,,\\
&f-\xi=\sum_{j\in[p]}\left((\mathbf v_k^{I_j})^\top \bar{\mathbf Q}_0^{(j)} \mathbf v_k^{I_j}\right.\\
&\qquad\qquad+\sum_{i\in J_j} g_i (\mathbf v_{k-\lceil g_i\rceil}^{I_j})^\top \bar{\mathbf Q}_i^{(j)} \mathbf v_{k-\lceil g_i\rceil}^{I_j}\\
&\qquad\qquad\left.+\sum_{i\in W_j} h_i (\mathbf v_{k-\lceil h_i\rceil}^{I_j})^\top \bar{\mathbf U}_i^{(j)} \mathbf v_{k-\lceil h_i\rceil}^{I_j}\right)
\end{array}
\right.\right\}\,.
\end{equation}

\subsection{Conditional gradient-based augmented Lagrangian (CGAL)}
\label{sec:sdp.ctp}
\subsubsection{SDP with CTP}
\label{sec:sdp.ctp.dense}
Let $s,l, s^{(j)}\in\N^{\ge 1}$, $j\in[\omega]$, be fixed such that $s=\sum_{j=1}^\omega s^{(j)}$. 
Let $\mathcal{S}$ be the set of real symmetric matrices of size $s$ in a block diagonal form:
$\mathbf X=\diag(\mathbf X_1,\dots,\mathbf X_\omega)$,
such that $\mathbf X_j$ is a block of size $s^{(j)}$, $j\in [\omega]$.
Let $s^{\max}:=\max_{j\in[\omega]}s^{(j)}$.
Let $\mathcal{S^+}$ be the set of all $\mathbf X\in \mathcal S$ such that $\mathbf X\succeq 0$, i.e., $\mathbf X$ has only nonnegative eigenvalues.
Then $\mathcal S$ is a Hilbert space with scalar product $\left<\mathbf A, \mathbf B\right> = \trace(\mathbf B^\top \mathbf A)$ and $\mathcal S^+$ is a self-dual cone.

Let us consider the following SDP:
\begin{equation}\label{eq:SDP.form.0}
\tau = \inf_{\mathbf X\in \mathcal{S}^+} \,\{ \,\left< \mathbf C,\mathbf X\right>\,:\,\mathcal{A} \mathbf X=\mathbf b\}\,,
\end{equation}
where $\mathcal{A}:\mathcal{S}\to \R^{\zeta}$ is a linear operator of the form
$\mathcal{A}\mathbf X=[\left< \mathbf A_{1},\mathbf X\right>,\dots,\left< \mathbf A_{\zeta},\mathbf X\right>]$,
with $\mathbf A_{i} \in \mathcal{S}$, $i \in[\zeta]$, $\mathbf C \in \mathcal{S}$ is the cost matrix and $\mathbf b\in \R^{\zeta}$ is a vector. 

The dual of SDP \eqref{eq:SDP.form.0} reads as:
\begin{equation}\label{eq:SDP.form.dual.0}
\rho = \sup_{\mathbf y\in\R^\zeta}\, \{\, \mathbf b^\top\mathbf y\,:\,
\mathcal{A}^\top \mathbf y-\mathbf C\in \mathcal{S}^+\,\}\,,
\end{equation}
where $\mathcal{A}^\top:\R^{\zeta}\to \mathcal{S}$ is the adjoint operator of $\mathcal{A}$, i.e., $\mathcal{A}^\top\mathbf y=\sum_{i\in[\zeta]} y_i\mathbf A_{i}$. 

The following assumption will be used later on.
\begin{assumption}\label{ass:general.assump.sdp}
Consider the following conditions:
\begin{enumerate}
\item Strong duality of primal-dual \eqref{eq:SDP.form.0}-\eqref{eq:SDP.form.dual.0} holds, i.e., $\rho=\tau$ and $\rho\in\R$.
\item Constant trace property (CTP): $\exists a>0\ :\ \forall \ \mathbf X\in \mathcal S\,,\,\mathcal{A} \mathbf X=\mathbf b\Rightarrow 
\trace(\mathbf X)=a$.
\end{enumerate}
\end{assumption}

For $\mathbf X\in\mathcal S$, the Frobenius norm of $\mathbf X$ is defined by 
$\|\mathbf X\|_F:=\sqrt{\left<\mathbf X, \mathbf X\right>}$.
We denote by $\|\mathcal A\|$ the operator norm of $\mathcal A$, i.e., 
   $ \|\mathcal A\|:=\max_{\mathbf X\in \mathcal S}\|\mathcal A \mathbf X\|_2/\|\mathbf X\|_F$.
The smallest eigenvalue of a real symmetric matrix $\mathbf D$ is denoted by $\lambda_{\min}(\mathbf D)$.

\paragraph{Algorithm.}

In \cite{yurtsever2019scalable}, Yurtsever et al. state
Algorithm \ref{alg:CGAL} (see below) to solve SDP \eqref{eq:SDP.form.0} with CTP. This procedure is based on the augmented Lagrangian paradigm combined together with the conditional gradient method.

    \begin{algorithm}
    \small
    \caption{CGAL-SDP-CTP}
    \label{alg:CGAL} 
    \textbf{Input:} SDP \eqref{eq:SDP.form.0} such that Assumption \ref{ass:general.assump.sdp} holds; Parameter $K>0$.\\
    \textbf{Output:} $(\mathbf X_t)_{t\in\N}$.
    \begin{algorithmic}[1]
    \State Set $\mathbf X_0:=\mathbf 0_{\mathcal S}$ and $\mathbf y_0:=\mathbf 0_{\R^\zeta}$.
 \For{$t\in\N$}{}
 \State Set $\beta_t:=\sqrt{t+1}$ and $\eta_t:=2/(t+1)$;
 \State Take an eigenvector $\mathbf u_t$ corresponding to $\lambda_{\min}(\mathbf C +\mathcal A^\top (\mathbf y_{t-1}+\eta_t(\mathcal A \mathbf X_{t-1}-\mathbf b)))$;
 \State Set $\mathbf X_{t}:=(1-\eta_t)\mathbf X_{t-1}+\eta_t a \mathbf u_t\mathbf u_t^\top$;
 \State Select $\gamma_t$ as the largest $\gamma\in[0,1]$ such that: \\
    \hspace*{\algorithmicindent}\hspace*{\algorithmicindent} $\gamma \|\mathcal A \mathbf X_{t}-\mathbf b\|_2^2 \le \beta_t\eta_t^2a^2\|\mathcal A\|^2$ and $\|\mathbf y_{t-1}+\gamma(\mathcal A \mathbf X_{t}-\mathbf b)\|_2\le K$;

 \State Set $\mathbf y_{t}=\mathbf y_{t-1}+\gamma_t(\mathcal A \mathbf X_{t}-\mathbf b)$.
 \EndFor
    \end{algorithmic}
    \end{algorithm}
    
The convergence of the sequence $(\mathbf X_t)_{t\in\N}$ in Algorithm \ref{alg:CGAL} to the set of optimal solutions of SDP \eqref{eq:SDP.form.0} is guaranteed as follows:
    \begin{theorem}\cite[Fact 3.1 ]{yurtsever2019scalable}
    Consider SDP \eqref{eq:SDP.form.0} such that Assumption \ref{ass:general.assump.sdp} holds. Let $(\mathbf X_t)_{t\in\N}$ be in the output of Algorithm \ref{alg:CGAL}. Then $\mathbf X_t\succeq 0$, for all $t\in\N$ and 
    $\|\mathcal A \mathbf X_{t}-\mathbf b\|_2\to 0$, $ |\left<\mathbf C, \mathbf X_t\right>-\tau| \to 0$ as $ t\to \infty$,
    with the rate of order $\mathcal O(\sqrt t)$.
    \end{theorem}
\begin{remark}
In order to achieve the best 
convergence rate for Algorithm \ref{alg:CGAL}, we scale the problem's input as follows:
$\|\mathbf C\|_F=\|\mathcal A\|=a=1$ and $\|\mathbf A_1\|_F=\dots=\|\mathbf A_\zeta\|_F$.
\end{remark}

\begin{remark}
Given $\varepsilon>0$, the for loop in Algorithm \ref{alg:CGAL} terminates when:
\begin{equation}\label{eq:gap.primal.dual}
    \frac{|\left<\mathbf C, \mathbf X_{t-1}\right> - (a\lambda_{\min}(\mathbf C +\mathcal A^\top (\mathbf y_{t-1}+\eta_t(\mathcal A \mathbf X_{t-1}-\mathbf b))) - \mathbf b^\top \mathbf y_{t-1})|}{1+\max\{|\left<\mathbf C, \mathbf X_{t-1}\right>|, |a\lambda_{\min}(\mathbf C +\mathcal A^\top (\mathbf y_{t-1}+\eta_t(\mathcal A \mathbf X_{t-1}-\mathbf b))) - \mathbf b^\top \mathbf y_{t-1}|\}}\le \varepsilon
\end{equation}
and $\|\mathcal A \mathbf X_{t-1}-\mathbf b\|_2/\max\{1,\|b\|_2\} \le \varepsilon$. 
In our experiments, we choose $\varepsilon=10^{-3}$.
Note that the left hand side in \eqref{eq:gap.primal.dual} is the relative gap between the  primal and dual approximate values obtained at each iteration.
\end{remark}

\begin{remark}\label{re:implicit}
To save memory at each iteration, we can run Algorithm \ref{alg:CGAL} with an implicit $\mathbf X_t$ by setting $\mathbf w_t:=\mathcal A \mathbf X_{t}-\mathbf b$.
In this case, Step 5 becomes $\mathbf w_t:=(1-\eta_t)\mathbf w_{t-1}+\eta_t [\mathcal A(a\mathbf u_t\mathbf u_t^\top)-b] $.
Thus we only obtain an approximate dual solution $\mathbf y_t$ of SDP \eqref{eq:SDP.form.0} when Algorithm \ref{alg:CGAL} terminates. 
To recover an approximate primal solution $\mathbf X$ of SDP \eqref{eq:SDP.form.0}, we rely on a process  similar to steps 2 and 3 of Algorithm \ref{alg:sol.SDP.CTP.0}, which will be presented later on. 
\end{remark}
In Appendix \ref{sec:sdp.ctp.blocks}, we provide an analogous method to solve an SDP with CTP on each subset of blocks.
\subsubsection{SDP with CTP on each subset of blocks}
\label{sec:sdp.ctp.blocks}
Let $p\in\N^{\ge 1}$, $s_j,\omega_j\in\N$, $j\in[p]$, and $s^{(i,j)}\in\N^{\ge 1}$, $i\in[\omega_p]$, $j\in[p]$, be fixed such that $s_j=\sum_{i\in [\omega_j]} s^{(i,j)}$, $j\in[p]$. 
For every $j\in[p]$, let $\mathcal{S}_j$ be the set of real symmetric matrices of size $s_j$ in a block diagonal form:
$\mathbf X_j=\diag(\mathbf X_{1,j},\dots,\mathbf X_{\omega_j,j})$, 
such that $\mathbf X_{i,j}$ is a block of size $s^{(i,j)}$, $i\in [\omega_j]$.
Let $s^{\max}:=\max_{i\in[\omega_p],j\in[p]}s^{(i,j)}$.
For every $j\in[p]$, let $\mathcal{S}^+_j$ be the set of all $\mathbf X_j\in \mathcal S_j$ such that $\mathbf X_j\succeq 0$.
Then for every $j\in[p]$, $\mathcal S_j$ is a Hilbert space with scalar product $\left<\mathbf A, \mathbf B\right> = \trace(\mathbf B^\top \mathbf A)$ and $\mathcal S^+_j$ is a self-dual cone.

Let us consider the following SDP:
\begin{equation}\label{eq:SDP.form.0.blocks}
\tau = \inf_{\mathbf X_j\in \mathcal{S}_j^+} \,\left\{ \,\sum_{j\in[p]}\left< \mathbf C_j,\mathbf X_j\right>\,:\,\sum_{j\in[p]}  \mathcal A_{j}\mathbf X_j= \mathbf b \right\}\,,
\end{equation}
where $\mathcal{A}_j:\mathcal{S}_j\to \R^{\zeta}$ is a linear operator of the form $\mathcal{A}_j\mathbf X=[\left< \mathbf A_{1,j},\mathbf X\right>,\dots,\left< \mathbf A_{\zeta,j},\mathbf X\right>]$,
with $\mathbf A_{i,j} \in \mathcal{S}_j$, $i \in[\zeta]$, $\mathbf C_j \in \mathcal{S}_j$, $j\in[p]$, and $\mathbf b\in \R^\zeta$. 

The dual of SDP \eqref{eq:SDP.form.0.blocks} reads as:
\begin{equation}\label{eq:SDP.form.dual.0.blocks}
\rho = \sup_{\mathbf y\in\R^\zeta}\, \left\{\, \mathbf b^\top\mathbf y\,:\,
\mathcal{A}_j^\top \mathbf y-\mathbf C_j\in \mathcal{S}^+_j\,,j\in[p]\,\right\}\,,
\end{equation}
where
$\mathcal{A}_j^\top:\R^{\zeta}\to \mathcal{S}_j$ is the adjoint operator of $\mathcal{A}_j$, i.e., $\mathcal{A}_j^\top\mathbf z=\sum_{i\in[\zeta]} z_i\mathbf A_{i,j}$, $j\in[p]$.

The following assumption will be used later on:
\begin{assumption}\label{ass:general.assump.sdp.blocks}
Consider the following conditions:
\begin{enumerate}
\item Strong duality of primal-dual \eqref{eq:SDP.form.0.blocks}-\eqref{eq:SDP.form.dual.0.blocks} holds, i.e., $\rho=\tau$ and $\rho\in\R$.
\item Constant trace property (CTP): there exist $a_j>0$ and  $j\in[p]$, such that 
\begin{equation}\label{eq:constan.trace.prop.blocks}
\left.
\begin{array}{lr}
\forall \ \mathbf X_j\in \mathcal S_j\,,j\in[p]\,,\\
\sum_{j\in[p]}\mathcal A_{j}\mathbf X_j= \mathbf b
\end{array}
\right\}
\Rightarrow 
\trace(\mathbf X_j)=a_j\,,j\in[p]\,.
\end{equation}
\end{enumerate}
\end{assumption}

 Recall that 
$\lambda_{\min}(\mathbf D)$ stands for the smallest eigenvalue of a real symmetric matrix $\mathbf D$.
We denote by $\prod_{j\in[p]}\mathcal{S}_j$ the set of all $\mathbf X=\diag(\mathbf X_j)_{j\in[p]}$ such that $\mathbf X_j\in \mathcal{S}_j$, for $j\in[p]$.
Let $\mathbf C:=\diag(\mathbf C_j)_{j\in[p]}$ and
let $\mathcal{A}:\prod_{j\in[p]}\mathcal{S}_j\to \R^{\zeta}$ be a linear operator of the form:
$\mathcal{A}\mathbf X= \sum_{j\in[p]} \mathcal{A}_j\mathbf X_j$, for all $ \mathbf X=\diag(\mathbf X_j)_{j\in[p]}\in \prod_{j\in[p]}\mathcal{S}_j$.
Then for every $\mathbf X=\diag(\mathbf X_j)_{j\in[p]}\in \prod_{j\in[p]}\mathcal{S}_j$, we have
$\left< \mathbf C,\mathbf X\right>=\sum_{j\in[p]}\left< \mathbf C_j,\mathbf X_j\right>$ and $\mathcal{A}\mathbf X=\left[\left< \mathbf A^{(1)},\mathbf X\right>,\dots,\left< \mathbf A^{(\zeta)},\mathbf X\right>\right]$,
where $\mathbf A^{(i)}:=\diag((\mathbf A_{i,j})_{j\in[p]})$, for $i\in [\zeta]$.

SDP \eqref{eq:SDP.form.0.blocks} can be rewritten as 
$\tau = \inf_{\mathbf X\in \prod_{j\in[p]} \mathcal{S}_j^+} \,\left\{ \,\left< \mathbf C,\mathbf X\right>\,:\,  \mathcal A\mathbf X= \mathbf b \right\}$.

The dual operator $\mathcal{A}^\top: \R^\zeta\to \prod_{j\in[p]} \mathcal{S}_j$ of $\mathcal{A}$ reads $\mathcal{A}^\top \mathbf z=\diag((\mathcal{A}_j^\top \mathbf z)_{j\in[p]})$.
Note $\Delta_j:=\{\mathbf X_j\in  \mathcal{S}_j^+\,:\,\trace(\mathbf X_j)=a_j\}$\,,\, for $j\in[p]$.
\paragraph{Algorithm.}
In order to solve SDP \eqref{eq:SDP.form.0.blocks} with CTP on each subset of blocks, we use  \cite[Algorithm 1]{yurtsever2019conditional} due to Yurtsever et al. to describe
Algorithm \ref{alg:CGAL.blocks} with the following setting:
$\mathcal X\leftarrow \Delta:=\prod_{j\in[p]}\Delta_j$,
     $\mathcal K\leftarrow  \{\mathbf b\}$, $p\leftarrow \zeta$,
     $Ax \leftarrow \mathcal A \mathbf X$, $f(x)\leftarrow \left< \mathbf C,\mathbf X\right>$,
    $\lambda_0 \leftarrow 1$, $\lambda_k  \leftarrow \beta_k$, $\sigma_k\leftarrow \gamma_k$.
     $D_{\mathcal{Y}_{k+1}}  \leftarrow K$,
     $L_f \leftarrow 0$, $\bar r_{k+1} \leftarrow \mathbf b$,
     $D_{\mathcal{X}}^2 \leftarrow 2\sum_{j\in[p]}a_j^2$,
     $v_k \leftarrow \mathbf C +\mathcal A^\top \mathbf z_k$,
    $\argmin_{x\in \mathcal X} \left<v_k,x\right> \leftarrow \argmin_{X\in \Delta} \left<\mathbf C +\mathcal A^\top \mathbf z_k,\mathbf X\right>$.

With fixed $\mathbf z_k$, we have:
\begin{equation*}
\begin{array}{rl}
     \min\limits_{\mathbf X\in \Delta} \left<\mathbf C +\mathcal A^\top \mathbf z_k,\mathbf X\right>& = \min\limits_{\diag((\mathbf X_j)_{j\in[p]})\in \prod_{j\in[p]}\Delta_j} \sum\limits_{j\in[p]}\left<\mathbf C_j +\mathcal A_j^\top \mathbf z_k,\mathbf X_j\right>  \\
     & = \sum\limits_{j\in[p]}\min\limits_{\mathbf X_j\in \Delta_j} \left<\mathbf C_j +\mathcal A_j^\top \mathbf z_k,\mathbf X_j\right> = \sum\limits_{j\in[p]} a_j\lambda_{\min}(\mathbf C_j +\mathcal A_j^\top \mathbf z_k)\,.
\end{array}
\end{equation*}
Let $\mathbf u_k^{(j)}$ be a uniform eigenvector corresponding to $\lambda_{\min}(\mathbf C_j +\mathcal A_j^\top \mathbf z_k)$, for $j\in[p]$.
Then one has
$\diag((a_j\mathbf u_k^{(j)}(\mathbf u_k^{(j)})^{\top})_{j\in[p]})\in \argmin_{\mathbf X\in \Delta} \left<\mathbf C +\mathcal A^\top \mathbf z_k,\mathbf X\right>$.
Thus we can set $s_k\leftarrow \diag((a_j\mathbf u_k^{(j)}(u_k^{(j)})^{\top})_{j\in[p]})$ in \cite[Algorithm 1]{yurtsever2019conditional}.

    \begin{algorithm}
    \small
    \caption{CGAL-SDP-CTP-Blocks}
    \label{alg:CGAL.blocks} 
    \textbf{Input:} SDP \eqref{eq:SDP.form.0.blocks} such that Assumption \ref{ass:general.assump.sdp.blocks} holds; Parameter $K>0$.\\
    \textbf{Output:} $((\mathbf X_j^{(t)})_{j\in [p]})_{t\in\N}$.
    \begin{algorithmic}[1]
    \State Set $(\mathbf X_j^{(0)})_{j\in[p]}:=(\mathbf 0_{\mathcal S})_{j\in[p]}$ and $\mathbf y_0:=\mathbf 0_{\R^\zeta}$.
 \For{$t\in\N$}{}
 \State Set $\beta_t:=\sqrt{t+1}$ and $\eta_t:=2/(t+1)$;
 \State Set $\mathbf z_t:=\mathbf y_{t-1}+\eta_t(\sum_{j\in[p]}\mathcal A_j \mathbf X_j^{(t-1)}-\mathbf b)$;
 \For{$j\in [p]$}{}
 \State Take a uniform eigenvector $\mathbf u_t^{(j)}$ corresponding to $\lambda_{\min}(\mathbf C_j +\mathcal A_j^\top \mathbf z_t)$;
 \State Set $\mathbf X_j^{(t)}:=(1-\eta_t)\mathbf X_j^{(t-1)}+\eta_t a_j \mathbf u_t^{(j)}(\mathbf u_t^{(j)})^\top$;
  \EndFor
 \State Select $\gamma_t$ as the larges $\gamma\in[0,1]$ such that: \\
    \hspace*{\algorithmicindent}\hspace*{\algorithmicindent} $\gamma \|\sum_{j\in[p]}\mathcal A_j \mathbf X_j^{(t)}-\mathbf b\|_2^2 \le \beta_t\eta_t^2(\sum_{j\in[p]}a_j^2)\|\mathcal A\|^2$ and  $\|\mathbf y_{t-1}+\gamma(\sum_{j\in[p]}\mathcal A_j \mathbf X_j^{(t)}-\mathbf b)\|_2\le K$;
 \State Set $\mathbf y_{t}=\mathbf y_{t-1}+\gamma_t(\sum_{j\in[p]}\mathcal A_j \mathbf X_j^{(t)} -\mathbf b)$.
 \EndFor
    \end{algorithmic}
    \end{algorithm}
    
Relying on \cite[Theorem 3.1]{yurtsever2019conditional}, we guarantee the convergence of the sequence $((\mathbf X_j^{(t)})_{j\in [p]})_{t\in\N}$ in Algorithm \ref{alg:CGAL.blocks} to the set of optimal solutions of SDP \eqref{eq:SDP.form.0.blocks} in the following theorem:
    \begin{theorem}
    Consider SDP \eqref{eq:SDP.form.0.blocks} such that Assumption \ref{ass:general.assump.sdp.blocks} holds. 
    Let $((\mathbf X_j^{(t)})_{j\in [p]})_{t\in\N}$ be  the output of Algorithm \ref{alg:CGAL.blocks}. 
    Then $\mathbf X_j^{(t)}\succeq 0$, for all $j\in[p]$ and for all $t\in\N$ and 
    $\left\|\sum_{j\in[p]}\mathcal A_j \mathbf X_j^{(t)}-\mathbf b\right\|_2\to 0$ and $ \left|\sum_{j\in[p]}\left<\mathbf C_j, \mathbf X_j^{(t)}\right>-\tau\right| \to 0$ as $t\to \infty$
    with the rate $\mathcal O(\sqrt t)$.
    \end{theorem}
\begin{remark}
Before running Algorithm \ref{alg:CGAL.blocks}, we  scale the problem's input as follows:
$\|\mathbf C\|_F=\|\mathcal A\|=a_1=\dots=a_p=1$ and $ \|\mathbf A^{(1)}\|_F=\dots=\|\mathbf A^{(\zeta)}\|_F$.
\end{remark}
\begin{remark}
Given $\varepsilon>0$, the for loop in Algorithm \ref{alg:CGAL.blocks} terminates when:
\begin{equation*}
    \frac{|\sum_{j\in[p]}\left<\mathbf C_j, \mathbf X_j^{(t-1)}\right> - \sum_{j\in[p]}(a_j\lambda_{\min}(\mathbf C_j +\mathcal A_j^\top \mathbf z_t) - \mathbf b^\top \mathbf y_{t-1})|}{1+\max\{|\sum_{j\in[p]}\left<\mathbf C_j, \mathbf X_j^{(t-1)}\right>|, |\sum_{j\in[p]}(a_j\lambda_{\min}(\mathbf C_j +\mathcal A_j^\top \mathbf z_t) - \mathbf b^\top \mathbf y_{t-1})|\}}\le \varepsilon
\end{equation*}
and $\|\sum_{j\in[p]}\mathcal A_j \mathbf X_j^{(t-1)}-\mathbf b\|_2/\max\{1,\|b\|_2\} \le \varepsilon$. 
In our experiments, we choose $\varepsilon=10^{-2}$.
\end{remark}

\begin{remark}\label{re:implicit.block}
To save memory at each iteration, we can run Algorithm \ref{alg:CGAL.blocks} with implicit $\mathbf X_j^{(t)}$, $j\in[p]$, by setting $\mathbf w_t:=\sum_{j\in[p]}\mathcal A_j \mathbf X_{j}^{(t)} -\mathbf b$.
In this case, Step 7 becomes $\mathbf w_t:=(1-\eta_t)\mathbf w_{t-1}+\eta_t[\sum_{j\in[p]} \mathcal A_j(a_j\mathbf u_t^{(j)}(\mathbf u_t^{(j)})^\top) -\mathbf b ] $.
Thus we only obtain an approximate dual solution $\mathbf y_t$ of SDP \eqref{eq:SDP.form.0.blocks} when Algorithm \ref{alg:CGAL.blocks} terminates. 
To recover an approximate primal solution $(\mathbf X_j)_{j\in[p]}$ of SDP \eqref{eq:SDP.form.0.blocks}, we do a process similar to steps 2 and 3 of Algorithm \ref{alg:sol.SDP.CTP.0.blocks} which will be presented later on.
\end{remark}
\subsection{Spectral method (SM)}
\label{sec:spectral.method}
\subsubsection{SDP with CTP}
\label{sec:spectral.method.dense}
Consider SDP with CTP described in Appendix \ref{sec:sdp.ctp}. The following assumption will be used later on:
\begin{assumption}\label{ass:general.assump.sdp.spectral}
Dual attainability: SDP \eqref{eq:SDP.form.dual.0} has an optimal solution.
\end{assumption}

\begin{lemma} \label{lem:obtain.dual.sol}
Let Assumption \ref{ass:general.assump.sdp} hold and let
$\varphi:\R^{\zeta}\to \R$ be a function defined by:
$\mathbf y\mapsto \varphi(\mathbf y)\,:=\,a\lambda_{\min}(\mathbf C-\mathcal{A}^\top\mathbf y)+\mathbf b^\top\mathbf y$.
Then:
\begin{equation}\label{eq:nonsmooth.hierarchy.0}
\tau=\sup_{\mathbf{y}\in \R^{\zeta}}\,\varphi(\mathbf y)\,.
\end{equation}
Moreover, if  Assumption \ref{ass:general.assump.sdp.spectral} holds, then problem \eqref{eq:nonsmooth.hierarchy.0} has an optimal solution.
\end{lemma}

Notice that  $\varphi$ in Lemma \ref{lem:obtain.dual.sol} is concave and continuous but not differentiable in general. The subdifferential of $\varphi$ at $\mathbf y$ reads:
$\partial \varphi(\mathbf y)=\{\mathbf b-a\mathcal{A}\mathbf U: \mathbf U\in \conv(\Gamma(\mathbf C-\mathcal{A}^\top\mathbf y))\}$,
where for each $\mathbf A\in \mathcal{S}$, 
$\Gamma(\mathbf A):=\{\mathbf u\mathbf u^\top\ :\ \mathbf A\mathbf u=\lambda_{\min}(\mathbf A)\mathbf u\ ,\ \|\mathbf u\|_2=1\}$.

Given $r\in\N^{\ge 1}$ and $\mathbf u_j\in \R^s$, $j\in[r]$, consider the following convex quadratic optimization problem (QP):
    \begin{equation}\label{eq:socp.sdp.sol}
    \begin{array}{rl}
        \min\limits_{ \xi\in\R^r} &\frac{1}2\left\|\mathbf b- a\mathcal{A}\left(\sum_{j\in[r]}\xi_j\mathbf u_j\mathbf u_j^\top\right)\right\|_2^2\\
        \text{s.t.}& \sum_{j\in[r]} \xi_j=1\,;\:\xi_j\ge 0\,,\,j\in [r]\,.
        \end{array}
    \end{equation}

Next, we describe 
Algorithm \ref{alg:sol.SDP.CTP.0} to solve  SDP \eqref{eq:SDP.form.0}, which is based on nonsmooth first-order optimization methods (e.g., LMBM \cite[Algorithm 1]{haarala2007globally}).

    \begin{algorithm}
    \small
    \caption{Spectral-SDP-CTP}
    \label{alg:sol.SDP.CTP.0} 
    \textbf{Input:} SDP \eqref{eq:SDP.form.0} with unknown optimal value and optimal solution;\\
    \hspace*{\algorithmicindent}\hspace*{\algorithmicindent} method (T) for solving convex nonsmooth unconstrained optimization problems (NSOP). \\
    \textbf{Output:} the optimal value $\rho$ and the optimal solution $\mathbf X^\star$ of SDP \eqref{eq:SDP.form.0}.
    \begin{algorithmic}[1]
    \State Compute the optimal value $\tau$ and an optimal solution $\mathbf{\bar y}$ of the  NSOP \eqref{eq:nonsmooth.hierarchy.0} by using method (T);
 \State Compute $\lambda_{\min}(\mathbf C-\mathcal{A}^\top\mathbf{\bar y})$ and its corresponding uniform eigenvectors $\mathbf u_1,\dots,\mathbf u_r$;
 \State Compute an optimal solution $(\bar \xi_1,\dots,\bar \xi_r)$ of QP \eqref{eq:socp.sdp.sol} and set $\mathbf X^\star=a\sum_{j=1}^r \bar \xi_j\mathbf u_j\mathbf u_j^\top$.
    \end{algorithmic}
    \end{algorithm}

\begin{corollary}\label{coro:well-defined.sdp.by.nonsmooth}
Let Assumption \ref{ass:general.assump.sdp} hold. Assume that the method (T) is globally convergent for NSOP \eqref{eq:nonsmooth.hierarchy.0} (e.g., (T) is LMBM).
Then output $\tau$ of Algorithm \ref{alg:sol.SDP.CTP.0} is well-defined.
Moreover, if Assumption \ref{ass:general.assump.sdp.spectral} holds, the vector $\bar {\mathbf y}$ mentioned at Step 1 of Algorithm \ref{alg:sol.SDP.CTP.0} exists and thus the output $\mathbf X^\star$ of Algorithm \ref{alg:sol.SDP.CTP.0} is well-defined.
\end{corollary}
\subsubsection{SDP with CTP on each subset of blocks}
\label{app:spectral.SDP.ctp}

Consider SDP with CTP on each subset of blocks described in Appendix \ref{sec:sdp.ctp.blocks}.

The following assumption will be used later on:
\begin{assumption}\label{ass:general.assump.sdp.blocks.spectral}
Dual attainability: SDP \eqref{eq:SDP.form.dual.0.blocks} has an optimal solution.
\end{assumption}

\begin{lemma} \label{lem:obtain.dual.sol.blocks}
Let Assumption \ref{ass:general.assump.sdp.blocks} hold and let
$\psi:\R^{\zeta}\to \R$ be a function defined by:
$\mathbf y\mapsto \psi(\mathbf y)\,:=\,\mathbf b^\top\mathbf y+\sum_{j\in[p]}a_j\lambda_{\min}(\mathbf C_j-\mathcal A_{j}^\top \mathbf y)$.
Then:
\begin{equation}\label{eq:nonsmooth.hierarchy.0.blocks}
\tau=\sup_{\mathbf{y}\in \R^{\zeta}}\,\psi(\mathbf y)\,.
\end{equation}
Moreover, if 
of Assumption \ref{ass:general.assump.sdp.blocks.spectral} holds, then problem \eqref{eq:nonsmooth.hierarchy.0.blocks} has an optimal solution.
\end{lemma}
\begin{proof}
From 
\eqref{eq:SDP.form.0.blocks} and Condition 4 
of Assumption \ref{ass:general.assump.sdp.blocks},
\begin{equation}
\tau = \inf_{\mathbf X_j\in \mathcal{S}_j^+} \,\left\{ \,\sum_{j\in[p]}\left< \mathbf C_j,\mathbf X_j\right>\,\left|\begin{array}{lr}\sum_{j\in[p]} \mathcal A_{j}\mathbf X_j=\mathbf  b\,,\\
\left< \mathbf I_j,\mathbf X_j\right>=a_j\,,\,j\in[p]
\end{array}\right.
\right\}\,,
\end{equation}
where $\mathbf I_j\in \mathcal S_j$ is the identity matrix, for $j\in[p]$.
Note that $\left< \mathbf I_j,\mathbf X_j\right>=\trace(\mathbf X_j)$, for $\mathbf X_j\in \mathcal{S}_j$, $j\in[p]$.
The dual of this SDP reads as:
\begin{equation}
\rho = \sup_{(\xi,\mathbf y)\in\R^{p+\zeta}}\, \left\{\, \sum_{j\in[p]}a_j\xi_j+\mathbf b^\top\mathbf y\,:\,
\mathbf C_j-\mathcal A_{j}^\top \mathbf y-\xi_j\mathbf I_j\in \mathcal{S}^+_j\,,j\in[p]\,\right\}\,.
\end{equation}
It implies that
$\rho=  \sup_{\xi,\mathbf y}\, \{\, \sum_{j\in[p]}a_j\xi_j+\mathbf b^\top\mathbf y\,:\,\xi_j\le \lambda_{\min}(\mathbf C_j-
\mathcal A_{j}^\top \mathbf y)\,,j\in[p]\,\}$.
From this, the result follows since $\rho=\tau$.
\end{proof}
\begin{proposition}\label{prop:properties.phi.0.blocks}
The function $\psi$ in Lemma \ref{lem:obtain.dual.sol.blocks} has the following properties:
\begin{enumerate}
\item  $\psi$ is concave and continuous but not differentiable in general.
\item The subdifferential of $\psi$ at $\mathbf y$ satisfies:
$\partial \psi(\mathbf y)= \mathbf b+\sum_{j\in[p]}a_j\partial \psi_j(\mathbf y)$,
where for every $j\in[p]$, $\psi_j:\R^{\zeta}\to \R$ is a function defined by $\psi_j(\mathbf y)=\lambda_{\min}(\mathbf C_j-\mathcal A_{j}^\top \mathbf y)$ and 
$\partial \psi_j(\mathbf y)=\{-\mathcal A_j\mathbf U
    \ :\ \mathbf U\in \conv(\Gamma(\mathbf C_j-\mathcal A_{j}^\top \mathbf y))\}$.
\end{enumerate}
\end{proposition}
\begin{proof}
It is not hard to prove the first statement.
Indeed, $\psi$ is a positive combination of $\mathbf z\mapsto \mathbf b^\top\mathbf z$, $\psi_j$, $j\in[p]$, which are convex, continuous functions.
The second statement follows by applying the subdifferential sum rule and notice that the domains of $\mathbf z\mapsto \mathbf b^\top\mathbf z$, $\psi_j$, $j\in[p]$, are both $\R^n$.
\end{proof}

\begin{lemma}\label{lem:extract.sdp.solu.blocks}
 If $\mathbf{\bar z}$ is an optimal solution of NSOP \eqref{eq:nonsmooth.hierarchy.0.blocks},
then:
\begin{enumerate}
    \item For each $j\in[p]$, there exists $(\mathbf X_j^\star)\in a_j\conv(\Gamma(\mathbf C_j-\mathcal A_{j}^\top \bar{\mathbf z}))$, such that  $\sum_{j\in[p]}\mathcal{A}_j\mathbf X_j^\star=\mathbf b$.
    \item For $j\in[p]$, $\mathbf X^\star_j=a_j\sum_{i\in [r_j]} \bar \xi_{i,j}\mathbf u_{i,j}\mathbf u_{i,j}^\top$, where $(\mathbf u_{i,j})_{i\in [r_j]}$ are all uniform eigenvectors  corresponding to $\lambda_{\min}(\mathbf C_j-\mathcal{A}_j^\top\mathbf{\bar z})$ and $((\bar \xi_{i,j})_{i\in[r_j]})_{j\in[r]}$ is an optimal solution of the convex quadratic problem:
    \begin{equation}\label{eq:socp.sdp.sol.blocks}
    \begin{array}{rl}
        \min\limits_{ \xi_{i,j}} &\frac{1}2\left\|\mathbf b- \sum_{j\in[p]}a_j\mathcal{A}_j\left(\sum_{j\in [r_j]}\xi_{i,j}\mathbf u_{i,j}\mathbf u_{i,j}^\top\right)\right\|_2^2\\
        \text{s.t.}& \sum_{i\in[r_j]} \xi_{i,j}=1\,;\:\xi_{i,j}\ge 0\,,\,i\in [r_j]\,,j\in[p]\,.
        \end{array}
    \end{equation}
    \item $(\mathbf X_j^\star)_{j\in[p]}$ is an optimal solution of SDP \eqref{eq:SDP.form.0.blocks}.
\end{enumerate}
 
\end{lemma}
\begin{proof}
By \cite[Theorem 4.2]{bagirov2014introduction}, $\mathbf 0\in \partial \psi(\mathbf{\bar z})$.
Combining this with Proposition \ref{prop:properties.phi.0.blocks}.2, 
the first statement follows, which in turn implies  the second statement.
We next prove the third statement.
For $j\in[p]$, one has $\mathbf X^\star_j\succeq 0$ since $\mathbf X^\star_j=a_j\sum_{i\in[r_j]} \bar \xi_{i,j}\mathbf u_{i,j}\mathbf u_{i,j}^\top$ with $\bar\xi_{i,j}\ge 0$, $i\in[r_j]$. 
From this and since $\sum_{j\in[p]}\mathcal{A}_j\mathbf X^\star_j=\mathbf b$, 
$(\mathbf X^\star_j)_{j\in[p]}$ is a feasible solution of SDP \eqref{eq:SDP.form.0.blocks}. 
Moreover,
\[\begin{array}{rl}
    \sum_{j\in[p]}\left< \mathbf C_j, \mathbf X^\star_j\right>&=\sum_{j\in[p]}\left< \mathbf C_j-\mathcal{A}_j^\top\mathbf{\bar z}, \mathbf X^\star_j\right>+\sum_{j\in[p]}\left< \mathcal{A}_j^\top\mathbf{\bar z}, \mathbf X^\star_j\right>\\
    &=\sum_{j\in[p]}a_j\sum_{i\in [r_j]} \bar \xi_{i,j}\left< \mathbf C_j-\mathcal{A}_j^\top\mathbf{\bar z},\mathbf u_{i,j}\mathbf u_{i,j}^\top\right>+\sum_{j\in[p]}\mathbf{\bar z}^\top( \mathcal{A}_j \mathbf X^\star_j)\\
    &=\sum_{j\in[p]}a_j\sum_{i\in[r_j]} \bar \xi_{i,j}\mathbf u_{i,j}^\top( \mathbf C_j-\mathcal{A}_j^\top\mathbf{\bar z})\mathbf u_{i,j}+\mathbf{\bar z}^\top\mathbf b\\
    &=\sum_{j\in[p]}a_j\lambda_{\min}(\mathbf C_j-\mathcal{A}_j^\top\mathbf{\bar z})\sum_{i\in[r_j]} \bar \xi_{i,j}\|\mathbf u_{i,j}\|_2^2+\mathbf{\bar z}^\top\mathbf b\\
    &=\sum_{j\in[p]}a_j\lambda_{\min}(\mathbf C_j-\mathcal{A}_j^\top\mathbf{\bar z})\sum_{i\in[r_j]} \bar \xi_{i,j}+\mathbf{\bar z}^\top\mathbf b\\
    &=\sum_{j\in[p]}a_j\lambda_{\min}(\mathbf C_j-\mathcal{A}_j^\top\mathbf{\bar z})+\mathbf{\bar z}^T\mathbf b=\psi(\mathbf{\bar z})=\tau\,.
\end{array}\]
Thus, $\sum_{j\in[p]}\left< \mathbf C_j, \mathbf X_j^\star\right>=\tau$, yielding the third statement.
\end{proof}

Next, we describe 
Algorithm \ref{alg:sol.SDP.CTP.0.blocks} to solve  SDP \eqref{eq:SDP.form.0.blocks}, which is based on nonsmooth first-order optimization methods (e.g., LMBM \cite[Algorithm 1]{haarala2007globally}).

    \begin{algorithm}
    \small
    \caption{Spectral-SDP-CTP-Blocks}
    \label{alg:sol.SDP.CTP.0.blocks} 
    \textbf{Input:} SDP \eqref{eq:SDP.form.0.blocks} with unknown optimal value and optimal solution;\\
    \hspace*{\algorithmicindent}\hspace*{\algorithmicindent} method (T) for solving NSOP. \\
    \textbf{Output:} the optimal value $\rho$ and the optimal solution $(\mathbf X^\star_j)_{j\in[p]}$ of SDP \eqref{eq:SDP.form.0.blocks}.
    \begin{algorithmic}[1]
    \State Compute the optimal value $\tau$ and an optimal solution $\mathbf{\bar y}$ of the  NSOP \eqref{eq:nonsmooth.hierarchy.0.blocks} by using method (T);
 \State For every $j\in[p]$, compute $\lambda_{\min}(\mathbf C_j-\mathcal{A}_j^\top\mathbf{\bar y})$ and its corresponding uniform eigenvectors $\mathbf u_{i,j}$, $i\in[r_j]$;
 \State Compute an optimal solution $((\bar \xi_{i,j})_{i\in[r_j]})_{j\in[r]}$ of QP \eqref{eq:socp.sdp.sol.blocks} and set $\mathbf X^\star_j=a_j\sum_{i\in[r_j]} \bar \xi_{i,j}\mathbf u_{i,j}\mathbf u_{i,j}^\top$, $j\in[p]$.
    \end{algorithmic}
    \end{algorithm}
The fact that Algorithm \ref{alg:sol.SDP.CTP.0.blocks} is well-defined
under certain conditions is a corollary of lemmas \ref{lem:obtain.dual.sol.blocks},  \ref{lem:extract.sdp.solu.blocks} and \cite[Lemma A.2]{mai2020hierarchy}.
\begin{corollary}\label{coro:well-defined.sdp.by.nonsmooth.blocks}
Let Assumption \ref{ass:general.assump.sdp.blocks} hold. Assume that the method (T) is globally convergent for NSOP \eqref{eq:nonsmooth.hierarchy.0.blocks} (e.g., (T) is LMBM).
Then output $\tau$ of Algorithm \ref{alg:sol.SDP.CTP.0.blocks} is well-defined.
Moreover, if Assumption \ref{ass:general.assump.sdp.blocks.spectral} holds, the vector $\bar {\mathbf y}$ involved at Step 1 of Algorithm \ref{alg:sol.SDP.CTP.0.blocks} exists and thus the output $(\mathbf X^\star_j)_{j\in[p]}$ of Algorithm \ref{alg:sol.SDP.CTP.0.blocks} is well-defined.
\end{corollary}
\subsection{Converting the moment relaxation to the standard SDP}
\subsubsection{The dense case}
\label{sec:convert.standart.SDP}
Let $k\in\N^{\ge k_{\min}}$ be fixed. 
We will present a way to transform SDP \eqref{eq:dual.diag.moment.mat} to the form \eqref{eq:SDP.form}.
By adding slack variables $\mathbf y^{(i)}\in \R^{s(2(k - \lceil g_i \rceil))}$, $i\in[m]$, SDP \eqref{eq:dual.diag.moment.mat} is equivalent to
\begin{equation}\label{eq:moment.hierarchy2}
\tau_k \,:= \,\inf \limits_{\mathbf y, \mathbf y^{(i)}} \left\{ L_{\mathbf y}(f)\ \left|\begin{array}{rl}
& \mathbf W_k(\mathbf y,\mathbf y^{(1)},\dots,\mathbf y^{(m)})\in \mathcal S_k^+\,,\\
&\mathbf M_{k - \lceil g_i \rceil }(\mathbf y^{(i)})=\mathbf M_{k - \lceil g_i \rceil }(g_i\;\mathbf y)\,,\,i\in[m]\,,\\
&\mathbf M_{k - \lceil h_j \rceil }(h_j\;\mathbf y)   = 0\,,\,j\in[l]
\end{array}
\right.\right\}\,,
\end{equation}
where $\mathbf W_k(\mathbf y,\mathbf y^{(1)},\dots,\mathbf y^{(m)}):=\diag(\mathbf M_k(\mathbf y),\mathbf M_{k-\lceil g_1\rceil}(\mathbf y^{(1)}),\dots, \mathbf M_{k-\lceil g_m\rceil}(\mathbf y^{(m)}))$.

Let $\mathcal{V}=\{\mathbf M_k(\mathbf z)\,:\, \mathbf z\in \R^{s(2k)}\}$ and $\mathcal{V}_i=\{\mathbf M_{k-\lceil g_i \rceil }(\mathbf z)\,:\, \mathbf z\in \R^{s(2(k-\lceil g_i \rceil ))}\}$, $i\in[m]$.
Then $\mathcal{V}$ and $\mathcal{V}_i$, $i\in[m]$, are the linear subspaces of the spaces of real symmetric matrices of size $s(k)$ and $s(k-\lceil g_i \rceil )$, $i\in[m]$, respectively.

Denote by $\mathcal{V}^\bot$, $\mathcal{V}_i^\bot$, $i\in[m]$, the orthogonal complements  of $\mathcal{V}$, $\mathcal{V}_i$, $i\in[m]$, respectively. 
In \cite[Appendix A.2 ]{mai2020hierarchy}, we show how to take a basis $\{\hat{\mathbf A}_j\}_{j\in[r]}$ of $\mathcal{V}^\bot$.
Similarly we can take a basis $\{\hat{\mathbf A}_j^{(i)}\}_{j\in[r_i]}$ of $\mathcal{V}^\bot_i$, $i\in[m]$.
Here $r=\dim(\mathcal{V}^\bot)$ and $r_i=\dim(\mathcal{V}_i^\bot)$, $i\in[m]$.

Notice that if $\mathbf X_0$ is a real symmetric matrix of size $s(k)$, then $\mathbf X_0=\mathbf M_k(\mathbf y)$ for some $\mathbf y\in \R^{s({2k})}$ if and only if  $\left< \hat{\mathbf A}_j, \mathbf X_0\right>=0$, $j\in[r]$. 
It implies that if $\mathbf X=\diag(\mathbf X_0,\dots,\mathbf X_m)\in\mathcal S_k$, then there exist $\mathbf y$ and $\mathbf y^{(i)}$, $i\in[m]$, such that 
$\mathbf X=\mathbf W_k(\mathbf y,\mathbf y^{(1)},\dots,\mathbf y^{(m)}) \Leftrightarrow \left< \bar{\mathbf A}, \mathbf X\right>=0\,,\, \bar{\mathbf A} \in \mathcal{B}_1$,
where $\mathcal{B}_1$ involves matrices $\bar{\mathbf A}$ defined as:
\begin{itemize}
    \item $\bar{\mathbf A}=\diag(\hat{\mathbf A}_j,\mathbf 0,\dots,\mathbf 0)$ for some $j\in[r]$;
    \item $\bar{\mathbf A}=\diag(\mathbf 0, \hat{\mathbf A}_j^{(1)},\dots,\mathbf 0)$ for some $j\in[r_1]$;
    \item $\dots$
    \item $\bar{\mathbf A}=\diag(\mathbf 0, \mathbf 0,\dots,\hat{\mathbf A}_j^{(m)})$ for some $j\in[r_m]$.
\end{itemize}
Notice that 
\begin{equation}
\begin{array}{rl}
     |\mathcal{B}_1|=r+\sum_{i\in[m]}r_i=&\displaystyle\frac{\s(k)(\s(k)+1)}{2}-\s(2k)  \\
     & +\displaystyle\sum_{i\in[m]}\left(\frac{\s(k-\lceil g_i\rceil)(\s(k-\lceil g_i\rceil)+1)}{2}-\s(2(k-\lceil g_i\rceil))\right)\,.
\end{array}
\end{equation}

The constraints $\mathbf M_{k - \lceil g_i \rceil }(\mathbf y^{(i)})=\mathbf M_{k - \lceil g_i \rceil }(g_i\;\mathbf y)$, $i\in[m]$, of SDP \eqref{eq:moment.hierarchy2} are equivalent to 
$\mathbf y^{(i)}_{\alpha}=\sum_{\gamma \in \N^n_{2\lceil g_i\rceil}} g_i\mathbf y_{\alpha+\gamma}$, $ \alpha\in \N^n_{2(k - \lceil g_i \rceil)}$, $i\in[m]$.
They can be written as
$    \left< \bar{\mathbf A}, \mathbf W_k(\mathbf y,\mathbf y^{(1)},\dots,\mathbf y^{(m)})\right>=0$, for $ \bar{\mathbf A} \in \mathcal{B}_2$,
where $\mathcal{B}_2$ involves matrices $\bar{\mathbf A}$ defined by 
$\bar{\mathbf A}=\diag(\tilde{\mathbf A},\mathbf 0,\dots, \mathbf 0,\tilde{\mathbf A}^{(i)},\mathbf 0,\dots,\mathbf 0)$,
with $\tilde {\mathbf A}=(\tilde A_{\mu,\nu})_{\mu,\nu\in\N^n_{k}}$ being defined as follows:
\begin{equation}
    \tilde A_{\mu,\nu}=\begin{cases}
g_{i,\gamma}&\text{ if }\mu=\nu\,,\,\mu+\nu=\alpha+\gamma\,,\\
\frac{1}2 g_{i,\gamma}&\text{ if }\mu\ne\nu\,,\,(\mu,\nu)\in \{(\mu_1,\nu_1),(\nu_1,\mu_1)\}\\
&\qquad\text{ with } (\mu_1,\nu_1)=\minimal(\{ (\bar\mu,\bar\nu)\in(\N^n_{k})^2\, :\,\bar\mu+\bar\nu=\alpha+\gamma\})\,,\\
0&\text{ otherwise,}
\end{cases}
\end{equation}
and 
$\tilde {\mathbf A}^{(i)}=(\tilde A_{\mu,\nu}^{(i)})_{\mu,\nu\in\N^n_{k-\lceil g_i \rceil}}$ being defined as follows:
\begin{equation}\label{eq:A.tilde}
    \tilde A_{\mu,\nu}^{(i)}=\begin{cases}
-1&\text{ if }\mu=\nu\,,\,\mu+\nu=\alpha\,,\\
-\frac{1}2 &\text{ if }\mu\ne\nu\,,\,(\mu,\nu)\in \{(\mu_1,\nu_1),(\nu_1,\mu_1)\}\\
&\qquad\text{ with } (\mu_1,\nu_1)=\minimal(\{ (\bar\mu,\bar\nu)\in(\N^n_{k})^2\, :\,\bar\mu+\bar\nu=\alpha\})\,,\\
0&\text{ otherwise,}
\end{cases}
\end{equation}
for some $\alpha\in \N^n_{2(k - \lceil g_i \rceil)}$ and $i\in[m]$.
Notice that $|\mathcal{B}_2|=\sum_{i\in[m]} 2(k - \lceil g_i \rceil)$.
Here $\minimal(T)$ is the minimal element of $T$, for every $T\subseteq \N^{2n}$ with respect to the graded lexicographic order.

The constraints $\mathbf M_{k - \lceil h_j \rceil }(h_j\;\mathbf y)   = 0$, $j\in[l]$, can be simplified as 
$\sum_{\gamma\in\N^n_{2\lceil h_j \rceil}}h_{j,\gamma}y_{\alpha+\gamma}=0$, $\alpha\in\N^n_{2(k-\lceil h_j \rceil)}$, $j\in [l]$.
They are equivalent to the following trace equality constraints:
$\left< \bar{\mathbf A}, \mathbf W_k(\mathbf y,\mathbf y^{(1)},\dots,\mathbf y^{(m)})\right>=0\,,\, \bar{\mathbf A} \in \mathcal{B}_3$,
where $\mathcal{B}_3$ involves matrices  $\bar{\mathbf A}=\diag(\tilde{\mathbf A},\mathbf 0,\dots, \mathbf 0)$,
with $\tilde {\mathbf A}=(\tilde A_{\mu,\nu})_{\mu,\nu\in\N^n_{k}}$ being defined as follows:
\begin{equation*}
    \tilde A_{\mu,\nu}=\begin{cases}
h_{j,\gamma}&\text{ if }\mu=\nu\,,\,\mu+\nu=\alpha+\gamma\,,\\

\frac{1}2 h_{j,\gamma}&\text{ if }\mu\ne\nu\,,\,(\mu,\nu)\in \{(\mu_1,\nu_1),(\nu_1,\mu_1)\}\\
&\qquad\text{ with } (\mu_1,\nu_1)=\minimal(\{ (\bar\mu,\bar\nu)\in(\N^n_{k})^2\, :\,\bar\mu+\bar\nu=\alpha+\gamma\})\,,\\
0&\text{ otherwise.}
\end{cases}
\end{equation*}
Notice that $|\mathcal{B}_3|=\sum_{j\in[l]} 2(k - \lceil h_j \rceil)$.

Let $\cup_{j\in [3]}\mathcal{B}_j=(\bar{\mathbf  A}_i)_{i\in [\zeta_k-1]}$, where 
\begin{equation*}
\begin{array}{rl}
    \zeta_k=1+\sum_{j\in [3]}|\mathcal{B}_j|=&1+\displaystyle\frac{\s(k)(\s(k)+1)}{2}-\s(2k)  \\
     & +\displaystyle\sum_{i\in[m]}\frac{\s(k-\lceil g_i\rceil)(\s(k-\lceil g_i\rceil)+1)}{2}+\sum_{j\in[l]}\s(2(k-\lceil h_j\rceil))\,.
     \end{array}
\end{equation*}
The final constraint $y_{\mathbf 0}=1$ can be rewritten as $\left<\bar{\mathbf A}_{\zeta_k}, \mathbf W_k(\mathbf y,\mathbf y^{(1)},\dots,\mathbf y^{(m)})\right>=1$ with $\bar{\mathbf A}_{\zeta_k}\in \mathcal S_k$ having zero entries  except the top left one $[\bar{ A}_{\zeta_k}]_{\mathbf 0,\mathbf 0}=1$.
Thus we select real vector $\mathbf b_k$ of length $t_k$ such that all entries of $\mathbf b_k$ are zeros except the final one $b_{\zeta_k}=1$.

The function $L_{\mathbf y}(f)=\sum_\gamma f_\gamma y_{\gamma}$ is equal to $\left<\bar{\mathbf C}, \mathbf W_k(\mathbf y,\mathbf y^{(1)},\dots,\mathbf y^{(m)})\right>$ with $\bar {\mathbf C}:=\diag(\tilde {\mathbf C},\mathbf 0,\dots,\mathbf 0)$, where $\tilde {\mathbf C}=(\tilde C_{\mu,\nu})_{\mu,\nu\in\N^n_k}$ is defined by:
\begin{equation*}
    \tilde C_{\mu,\nu}=\begin{cases}
f_{\gamma}&\text{ if }\mu=\nu\,,\,\mu+\nu=\gamma\,,\\
\frac{1}2 f_{\gamma}&\text{ if }\mu\ne\nu\,,\,(\mu,\nu)\in \{(\mu_1,\nu_1),(\nu_1,\mu_1)\}\\
&\qquad\text{ with } (\mu_1,\nu_1)=\minimal(\{ (\bar\mu,\bar\nu)\in(\N^n_{k})^2\, :\,\bar\mu+\bar\nu=\gamma\})\,,\\
0&\text{ otherwise.}
\end{cases}
\end{equation*}
By noting $\bar{\mathbf  X}=\mathbf W_k(\mathbf y,\mathbf y^{(1)},\dots,\mathbf y^{(m)})$, SDP \eqref{eq:moment.hierarchy2} has the standard form
\begin{equation}\label{eq:SDP.form.equivalent}
\tau_k = \inf_{\bar{\mathbf X}\in \mathcal{S}_k^+} \,\{ \,\left< \bar{\mathbf C},\bar{\mathbf X}\right>\,:\,\bar{\mathcal{A}} \bar{\mathbf X}=\mathbf b_k\}\,,
\end{equation}
where $\bar{\mathcal{A}}:\mathcal{S}_k\to \R^{\zeta_k}$ is a linear operator of the form $\bar{\mathcal{A}}\mathbf X=\left[\left< \bar{\mathbf A}_{1},\mathbf X\right>,\dots,\left< \bar{\mathbf A}_{\zeta_k},\mathbf X\right>\right]$.
Since 
$\left< \mathbf U,\mathbf V\right>=\left< \mathbf P_k^{-1}\mathbf U\mathbf P_k^{-1},\mathbf P_k\mathbf V\mathbf P_k\right>$, for all $\mathbf U,\mathbf V\in \mathcal{S}_k$,
by noting $\mathbf X=\mathbf P_k\bar{\mathbf X}\mathbf P_k$, SDP \eqref{eq:SDP.form.equivalent} can be written as \eqref{eq:SDP.form} with $\mathbf A_{k,i}=\mathbf P_k^{-1} \bar{\mathbf A}_i \mathbf P_k^{-1}$, $i\in[\zeta_k]$, and $\mathbf C_k=\mathbf P_k^{-1}\bar{\mathbf C}\mathbf P_k^{-1}$.
\subsubsection{The sparse case}
\label{sec:convert.standart.SDP.cs}
Let $k\in\N^{\ge k_{\min}}$ be fixed. 
We will present a way to transform SDP \eqref{eq:cs-ts.SDP.simple} to the form \eqref{eq:SDP.form.sparse}.
Doing a similar process as in Appendix \ref{sec:convert.standart.SDP} on every clique, by noting \eqref{eq:convert.momentmat.sparse}, for every $j\in[p]$, the constraints 
\begin{equation}
    \begin{cases}
    \mathbf D_k(\mathbf y, I_j) \succeq 0\,,\,y_{\mathbf 0}\,=\,1\,,\\
    \mathbf M_{k - \lceil h_i \rceil }(h_i\;\mathbf y,I_j)   = 0\,,\,i\in W_j\,,
    \end{cases}
\end{equation}
become $\hat {\mathcal A}_j \mathbf X_j=\hat{\mathbf b}_j$ for some linear operator $\hat {\mathcal A}_j:\mathcal{S}_{j,k}\to \R^{\hat \zeta_{j}}$ and vector $\hat{\mathbf b}_j\in \R^{\hat \zeta_{j}}$.
Moreover, $L_{\mathbf y}(f_j)=\left<\mathbf C_j, \mathbf X_j\right>$ for some matrix $\mathbf C_j\in \mathcal{S}_{j,k}$ since $f_j\in\R[x(I_j)]$, for every $j\in[p]$.
Then from \eqref{eq:convert.momentmat.sparse}, the objective function of SDP \eqref{eq:cs-ts.SDP.simple} is $L_{\mathbf y}(f)=\sum_{j\in[p]}\left<\mathbf C_j, \mathbf X_j\right>$.

Next we describe the constraints depending on common moments on cliques.
For every $\alpha\in \cup_{j\in[p]}\N^{I_j}_k$, note $T(\alpha):=\{j\in [p]\,:\,\alpha\in \N^{I_j}_k\}$. 
In other words, $T(\alpha)$ indices the cliques sharing the same moment $y_\alpha$.
For $\alpha\in \cup_{j\in[p]}\N^{I_j}_k$ such that $|T(\alpha)|\ge 2$, for every $j\in T(\alpha)$, let $\hat {\mathbf A}^{(\alpha)}_j\in \mathcal S_{j,k}$ be such that $\left<\hat {\mathbf A}^{(\alpha)}_j, \mathbf X_j\right>=y_\alpha$
It implies the constraints 
$\left<\hat {\mathbf A}^{(\alpha)}_{j_0}, \mathbf X_{j_0}\right>-\left<\hat {\mathbf A}^{(\alpha)}_i, \mathbf X_i\right>=0$, $ i\in T(\alpha)\backslash \{j_0\}$,
for every $\alpha\in \cup_{j\in[p]}\N^{I_j}_k$ such that $|T(\alpha)|\ge 2$, for some $j_0\in T(\alpha)$.
We denote by $\tilde {\mathcal A} \mathbf X =\mathbf 0_{\R^{\tilde \zeta}}$ all these constraints with $\mathbf X=\diag(\mathbf X_j)$.

Set $\zeta:=\sum_{j\in[p]}\hat \zeta_{j}+\tilde \zeta $ and $\mathbf b=[(\hat {\mathbf b}_j)_{j\in[p]},\mathbf 0_{\R^{\tilde \zeta}}] \in \R^\zeta$. 
Define the linear operator $\mathcal A: \prod_{j\in[p]} \mathcal S_{j,k}\to \R^{\zeta}$ such that
$\mathcal A\mathbf X=[(\hat{\mathcal A}_j \mathbf X_j)_{j\in[p]},\tilde {\mathcal A} \mathbf X]$, for all $ \mathbf X=\diag(\mathbf X_j)_{j\in[p]}\in \prod_{j\in[p]}\mathcal{S}_j$.
From \eqref{eq:convert.momentmat.sparse}, the affine constraints of SDP \eqref{eq:cs-ts.SDP.simple} are now equivalent to $\mathcal A\mathbf X=\mathbf b$.

Let $\mathbf A^{(i)}:=\diag((\mathbf A_{i,j})_{j\in[p]})\in \prod_{j\in[p]}\mathcal{S}_j $, $i\in [\zeta]$, be such that 
\[
\mathcal{A}\mathbf X=[\left< \mathbf A^{(1)},\mathbf X\right>,\dots,\left< \mathbf A^{(\zeta)},\mathbf X\right>]
\,,\] 
for all $ \mathbf X=\diag(\mathbf X_j)_{j\in[p]}\in \prod_{j\in[p]}\mathcal{S}_j$.
For every $j\in[p]$, define $\mathcal{A}_j:\mathcal{S}_j\to \R^{\zeta}$ as a linear operator of the form
$\mathcal{A}_j\mathbf X:=[\left< \mathbf A_{1,j},\mathbf X\right>,\dots,\left< \mathbf A_{\zeta,j},\mathbf X\right>]$.
Then $\mathcal{A}\mathbf X= \sum_{j\in[p]} \mathcal{A}_j\mathbf X_j$, for all $ \mathbf X=\diag(\mathbf X_j)_{j\in[p]}\in \prod_{j\in[p]}\mathcal{S}_j$.
Hence we obtain the data $(\mathbf C_{j,k},\mathcal A_{j,k}, \mathbf b_k,\zeta_k)=(\mathbf C_j,\mathcal A_j, \mathbf b,\zeta)$ of the standard form \eqref{eq:SDP.form.sparse} by plugging $k$.

\subsection{Proof of Theorem \ref{theo:suff.cond.CTP}}
\label{proof:suff.cond.CTP}
\begin{proof}
\begin{enumerate}
    \item Let $k\in\N^{\ge k_{\min}}$ and assume that $\R^{>0}\subseteq Q_k^\circ(g)+I_k(h)$. Then there exists $a_k>0$ such that
    \begin{equation}\label{eq:rep}
    a_k=\mathbf v_k^\top \mathbf G_0 \mathbf v_k+\sum_{i\in[m]} g_i \mathbf v_{k-\lceil g_i\rceil}^\top \mathbf G_i \mathbf v_{k-\lceil g_i\rceil}+ \sum_{j\in[l]} h_j \mathbf v_{2(k-\lceil h_j\rceil)}^\top \mathbf u_j \,,
\end{equation}
for some $\mathbf G_i\succ 0$, $i\in\{0\}\cup [m]$ and real vector $\mathbf u_j$, $j\in[l]$.
We denote by $\mathbf G_i^{1/2}$ the square root of $\mathbf G_i$, $i\in\{0\}\cup [m]$. 
Then $\mathbf G_i^{1/2}$ is well-defined and $\mathbf G_i^{1/2}\succ 0$. 
Set $\mathbf P_k=\diag(\mathbf G_0^{1/2},\dots,\mathbf G_m^{1/2})$.
Let $\mathbf y \in \R^{\s(2k)}$ such that $\mathbf M_k(h_j\mathbf y)=0$, $j\in[l]$, and $ y_{\mathbf 0}=1$.
Then 
\begin{equation}
    L_{\mathbf y}\left(\sum_{j\in[l]} h_j \mathbf v_{2(k-\lceil h_j\rceil)}^\top \mathbf u_j\right)=\sum_{j\in[l]}\sum_{\alpha\in\N^n_{2(k-\lceil h_j\rceil)}}u_{j,\alpha}L_{\mathbf y}(h_j\mathbf x^\alpha)=0\,.
    \end{equation}
From this and \eqref{eq:rep},
\begin{equation*}
\begin{array}{rl}
     a_k&=L_{\mathbf y}(\mathbf v_k^\top \mathbf G_0 \mathbf v_k + \sum_{i\in[m]}g_i\mathbf v_{k-\lceil g_i\rceil}^\top \mathbf G_i \mathbf v_{k-\lceil g_i\rceil} ) \\
     & =\trace(\mathbf M_k(\mathbf y)\mathbf G_0)+\sum_{i\in[m]}\trace(\mathbf M_{k-1}(g_i\mathbf y)\mathbf G_i)\\
     &=\trace(\mathbf G_{0}^{1/2}\mathbf M_k(\mathbf y)\mathbf G_{0}^{1/2})+\sum_{i\in[m]}\trace(\mathbf G_{i}^{1/2}\mathbf M_{k-1}(g_i\mathbf y)\mathbf G_{i}^{1/2})\\
     &=\trace(\mathbf P_{k} \mathbf D_k(\mathbf y)\mathbf P_{k})\,,
\end{array}
\end{equation*}
yielding the first statement.    
\item The ``if'' part comes from the first statement. 
Let us prove the ``only if'' part.
Assume that POP \eqref{eq:POP.def} has CTP (Definition \ref{def:ctp}). 
Let $\mathbf a\in S(g)$, $\mathbf y=(y_\alpha)_{\alpha\in\N^n}$ be the moment sequence of the Dirac measure $\delta_{\mathbf a}$. Let $k\in\N^{\ge k_{\min}}$ be fixed.
Since $\mathbf P_{k}\in \mathcal{S}_k$, 
$\mathbf P_k=\diag(\mathbf W_0,\dots,\mathbf W_m)$.
Then $\mathbf W_i^2\succ 0$, $i\in\{0\}\cup[m]$ since $\mathbf P_k\succ 0$.
Let us define the polynomial $w:=\mathbf v_k^\top \mathbf W_0^2 \mathbf v_k + \sum_{i\in[m]}g_i\mathbf v_{k-\lceil g_i\rceil}^\top \mathbf W_i^2 \mathbf v_{k-\lceil g_i\rceil}$.
By assumption,
\begin{equation*}
\begin{array}{rl}
     a_k&=\trace(\mathbf P_{k} \mathbf D_k(\mathbf y)\mathbf P_{k})\\
     &=\trace(\mathbf W_{0}\mathbf M_k(\mathbf y)\mathbf W_{0})+\sum_{i\in[m]}\trace(\mathbf W_{i}\mathbf M_{k-1}(g_i\mathbf y)\mathbf W_{i})\\
     & =\trace(\mathbf M_k(\mathbf y)\mathbf W_0^2)+\sum_{i\in[m]}\trace(\mathbf M_{k-1}(g_i\mathbf y)\mathbf W_i^2)\\
     &=L_{\mathbf y}(\mathbf v_k^\top \mathbf W_0^2 \mathbf v_k + \sum_{i\in[m]}g_i\mathbf v_{k-\lceil g_i\rceil}^\top \mathbf W_i^2 \mathbf v_{k-\lceil g_i\rceil} ) =\int_{\R^n} w \delta_{\mathbf a}=w(\mathbf a)\,,
\end{array}
\end{equation*}
It implies that $w-a_k$ vanishes on $S(g)$. Since $S(g)$ has nonempty interior, $w=a_k$, yielding the second statement.
\end{enumerate}
\end{proof}

\subsection{Proof of Proposition \ref{prop:suff.cond.feas}}
\label{proof:suff.cond.feas}
\begin{proof}
Let Assumption \ref{ass:bound.cons} hold. 
It is sufficient to show that \eqref{eq:find.CTP.diag} has a feasible solution for every $k\in\N^{\ge k_{\min}}$.

Let $\mathbf u=(u_j)_{j\in[n]}\subseteq\N^{\le m}$ be defined by
\begin{equation}
    u_j:=|\{i\in[r]\ :\ j\in T_i\}|+|\{i\in[m]\backslash[2r]\ :\ j\in T_{i}\}|\,,\,\qquad\forall\ j\in[n]\,.
\end{equation}
Since $(\cup_{i\in[r]} T_i) \cup (\cup_{i\in[m]\backslash[2r]} T_{i})=[n]$, one has  $u_j\in\N^{\ge 1}$, $j\in[n]$.
Moreover,
\begin{equation}
    \|\mathbf u\circ \mathbf x\|_2^2=\sum_{i\in[r]}\|\mathbf x(T_i)\|^2_2+\sum_{i\in[m]\backslash[2r]}\|\mathbf x(T_{i})\|^2_2\,.
\end{equation}
With 
$R:=\sum_{i\in[r]}(\underline R_i+\overline R_{i})+\sum_{i\in[m]\backslash[2r]}\overline R_{i}$,
by replacing $\mathbf x$ by $\mathbf u\circ \mathbf x$ in Lemma \ref{lem:equality}, one obtains that for all $k\in\N^{\ge k_{\min}}$,
\begin{equation}\label{eq:equality.multi}
    (R+1)^k=(1+\|\mathbf u\circ \mathbf x\|^2_2)^k+\Lambda_{k-1}\sum_{i\in[m]}\delta_ig_i\,,
\end{equation}
where $\Lambda_{k-1}:=\sum_{j=0}^{k-1}(R+1)^j(1+\|\mathbf u\circ \mathbf x\|^2_2)^{k-j-1}$ and
\begin{equation}
    \delta_i:=\frac{\underline R_i}{\overline R_i-\underline R_i}\,,\,\delta_{i+r}:=\frac{\overline R_i}{\overline R_i-\underline R_i}\,,\,i\in[r]\,,\,\text{ and }\delta_{q}=1\,,\,q\in[m]\backslash[2r].
\end{equation}
It is due to the fact that 
\begin{equation}
    R-\|\mathbf u\circ \mathbf x\|_2^2=\sum_{i\in[r]}(\underline R_i+\overline R_i-\|\mathbf x(T_i)\|^2_2)+\sum_{i\in[m]\backslash[2r]}(\overline R_{i}-\|\mathbf x(T_{i})\|^2_2)\,,
\end{equation}
and 
$\underline R_i+\overline R_i-\|\mathbf x(T_i)\|^2_2=\delta_ig_i+\delta_{i+r}g_{i+r}$, for all $ i\in[r]$.
For each $k\in\N^{\ge k_{\min}}$, let $(\theta_{k,\alpha})_{\alpha\in\N^n_k}\subseteq \R^{>0}$ and $(\eta_{k-1,\alpha})_{\alpha\in\N^n_{k-1}}\subseteq \R^{>0}$ be such that 
\[(1+\|\mathbf u \circ \mathbf
x\|_2^2)^k=\sum_{\alpha\in\N^n_k}\theta_{k,\alpha}\mathbf x^{2\alpha}\quad \text{and}\quad \Lambda_{k-1}=\sum_{\alpha\in\N^n_{k-1}}\eta_{k-1,\alpha}\mathbf x^{2\alpha}\,,\]
and define the diagonal matrices
\begin{equation}
\mathbf G_{k}^{(0)}:=\diag((\theta_{k,\alpha})_{\alpha\in\N^n_k})\quad\text{and}\quad\mathbf G_{k-1}^{(i)}:=\diag((\delta_i\eta_{k-1,\alpha})_{\alpha\in\N^n_{k-1}})\,,\,i\in[m]\,.
\end{equation}
Then \eqref{eq:equality.multi} yields that for every $k\in\N^{\ge k_{\min}}$:
\[(R+1)^k=\mathbf v_k^\top \mathbf G_{k}^{(0)} \mathbf v_k + \sum_{i\in[m]}g_i\mathbf v_{k-1}^\top \mathbf G_{k-1}^{(i)} \mathbf v_{k-1} \,. \]
Hence $((R+1)^k,\mathbf G_{k}^{(i)},\mathbf 0)$ is a feasible solution of \eqref{eq:find.CTP.diag}, for every $k\in\N^{\ge k_{\min}}$.
\end{proof}

\subsection{Proof of Proposition \ref{prop:CTP.equidegree}}
\label{sec:proof.equidegree}

\begin{proof}
Let Assumption \ref{ass:equidegree} hold with $u:=\lceil g_i\rceil$, $i\in[n+1]$.
For every $k\in\N^{\ge k_{\min}}$, letting $\Lambda_{k-1}:=\sum_{j=0}^{k-1}( R+1)^j(1+\|\mathbf x\|^2_2)^{k-j-1}$ and $\Theta_t:=(1+\|\mathbf x\|^2_2)^t$, for $t\in\N$, Lemma \ref{lem:equality} yields:
$(R+1)^k=\Theta_k+g_m\Lambda_{k-1}$.
It implies that for every $k\in\N^{\ge k_{\min}}$,
\begin{equation}\label{eq:pos.cert}
    (R+1)^k=(\Theta_k-\frac{L}{L+1}\Theta_{k-u})+\frac{1}{L+1}\Theta_{k-u}\sum_{i\in[m-1]}g_i+g_{m}\Lambda_{k-1}\,.
\end{equation}
It is due to the fact that $\sum_{i\in[m-1]}g_i=L$.
For each $k\in\N^{\ge k_{\min}}$, let us consider the following sequences:
\begin{itemize}
    \item $(\nu_{k,\alpha})_{\alpha\in\N^n_k}\subseteq \R^{>0}$ such that $\Theta_k-\frac{L}{L+1}\Theta_{k-u}=\sum_{\alpha\in\N^n_k}\nu_{k,\alpha}\mathbf x^{2\alpha}$;
    \item $(\theta_{k-u,\alpha})_{\alpha\in\N^n_{k-u}}\subseteq \R^{>0}$ such that $\frac{1}{L+1}\Theta_{k-u}=\sum_{\alpha\in\N^n_{k-u}}\theta_{k-u,\alpha}\mathbf x^{2\alpha}$;
    \item $(\eta_{k-1,\alpha})_{\alpha\in\N^n_{k-1}}\subseteq \R^{>0}$ such that $\Lambda_{k-1}=\sum_{\alpha\in\N^n_{k-1}}\eta_{k-1,\alpha}\mathbf x^{2\alpha}$.
\end{itemize}
For each $k\in\N^{\ge k_{\min}}$, define the diagonal matrices: $\mathbf G_{k}^{(0)}:=\diag((\nu_{k,\alpha})_{\alpha\in\N^n_k})$, 
\begin{equation*}
    \mathbf G_{k-u}^{(1)}:=\diag((\theta_{k-u,\alpha})_{\alpha\in\N^n_{k-u}})\,,\quad\text{and}\quad
    \mathbf G_{k-1}^{(2)}:=\diag((\eta_{k-1,\alpha})_{\alpha\in\N^n_{k-1}})\,.
\end{equation*}
Then \eqref{eq:pos.cert} yields that for every $k\in\N^{\ge k_{\min}}$,
\begin{equation}\label{eq:repres.simplex}
    (R+1)^k=\mathbf v_k^\top \mathbf G_{k}^{(0)} \mathbf v_k + \mathbf v_{k-u}^\top \mathbf G_{k-u}^{(1)} \mathbf v_{k-u}\sum_{i\in[m-1]}g_i+ \mathbf v_{k-1}^\top \mathbf G_{k-1}^{(2)} \mathbf v_{k-1}g_{m}\,.
\end{equation}
Hence $((R+1)^k,\mathbf G_{k}^{(i)},\mathbf 0)$ is a feasible solution of \eqref{eq:find.CTP.diag}, for every $k\in\N^{\ge k_{\min}}$.
By using Lemma \ref{lem:feas.LP}, the conclusion follows.
\end{proof}

\subsection{Proof of Corollary \ref{coro:equi.pop.ctp}}
\label{proof:coro:equi.pop.ctp}
\begin{proof}
Let $\tilde g:=\{\tilde g_{i}\}_{i\in[m+2]}$. 
Then $\{\tilde g_{i}\}_{i\in[m]}$ have the equivalent degree, i.e., there exists $u \in \N$ such that $\lceil \tilde g_{i} \rceil = u $, for all $i \in[m]$. 
Thus Assumption \ref{ass:equidegree} holds for $g \leftarrow \tilde g$, $m\leftarrow m+2$. 
By Proposition \ref{prop:CTP.equidegree}, \eqref{eq:find.CTP.diag} has a feasible solution with $g \leftarrow \tilde g$ for every order $k\in\N^{\ge k_{\min}}$.
It implies that for every $k\in\N^{\ge k_{\min}}$, there exist $\mathbf u^{(j)}_k\in \R^{\s(2(k-\lceil h_j\rceil))}$, $j\in[l]$, and
\begin{equation*}
    (\eta^{(0)}_{k,\alpha})_{\alpha\in \N^n_{k}}\subseteq \R^{>0}\,,\quad (\eta^{(i)}_{k-u,\alpha})_{\alpha\in \N^n_{k-u}}\subseteq \R^{>0}\,,\,i\in[m+1]\,,\quad  (\eta^{(m+2)}_{k-1,\alpha})_{\alpha\in \N^n_{k-1}}\subseteq \R^{>0}
\end{equation*}
 such that
 \begin{equation*}
 \begin{array}{rl}
     1=&\mathbf v_k^\top \diag((\eta^{(0)}_{k,\alpha})_{\alpha\in \N^n_{k}}) \mathbf v_k+\sum_{i\in[m+1]}\tilde g_i \mathbf v_{k-u}^\top \diag((\eta^{(i)}_{k-u,\alpha})_{\alpha\in \N^n_{k-u}}) \mathbf v_{k-u}\\
     &+\tilde g_{m+2} \mathbf v_{k-1}^\top \diag((\eta^{(m+2)}_{k-1,\alpha})_{\alpha\in \N^n_{k-1}}) \mathbf v_{k-1}+ \sum_{j\in[l]} h_j \mathbf v_{2(k-\lceil h_j\rceil)}^\top\mathbf u^{(j)}_k\,.
     \end{array}
 \end{equation*}
Let $k\in\N^{\ge k_{\min}}$ be fixed.
We define the following polynomials:
 \begin{itemize}
     \item $\sigma_0:=\mathbf v_k^\top \diag((\eta^{(0)}_{k,\alpha})_{\alpha\in \N^n_{k}}) \mathbf v_k=\sum_{\alpha\in \N^n_{k}} \eta^{(0)}_{k,\alpha}\mathbf x^{2\alpha}$,
     \item $\sigma_i:=\mathbf v_{k-u}^\top \diag((\eta^{(i)}_{k-u,\alpha})_{\alpha\in \N^n_{k-u}}) \mathbf v_{k-u}=\sum_{\alpha\in \N^n_{k-u}} \eta^{(i)}_{k-u,\alpha}\mathbf x^{2\alpha}$, $i\in[m+1]$,
     \item $\sigma_{m+2}:=\mathbf v_{k-1}^\top \diag((\eta^{(m+2)}_{k-1,\alpha})_{\alpha\in \N^n_{k-1}}) \mathbf v_{k-1}=\sum_{\alpha\in \N^n_{k-1}} \eta^{(m+2)}_{k-1,\alpha}\mathbf x^{2\alpha}$,
     \item $\psi_j:=\mathbf v_{2(k-\lceil h_j\rceil)}^\top\mathbf u^{(j)}_k$, $j\in[l]$.
 \end{itemize}
From these and since $\tilde g_i:=g_i(1+\|\mathbf x\|_2^2)^{u-\lceil g_i\rceil}$, for $i\in[m]$, one has
\begin{equation}\label{eq:rep.SOS.1}
\begin{array}{rl}
    1=&\sigma_0+\sum_{i\in[m]}\sigma_i\tilde g_i+ \sum_{j\in[l]} \psi_j h_j =\sigma_0+\sum_{i\in[m]}\sigma_i(1+\|\mathbf x\|_2^2)^{u-\lceil g_i\rceil} g_i\\
    &+\tilde g_{m+1} \sigma_{m+1}+\tilde g_{m+2} \sigma_{m+2} 
    + \sum_{j\in[l]} \psi_j h_j
    \,.
    \end{array}
\end{equation}

Then there exist  
$(\theta^{(i)}_{k-\lceil g_i\rceil ,\alpha})_{\alpha\in \N^n_{k-\lceil g_i\rceil}}\subseteq \R^{>0}$, $i\in[m]$,
 such that
 \begin{equation}
     \sigma_i(1+\|\mathbf x\|_2^2)^{u-\lceil g_i\rceil}=\sum_{\alpha\in \N^n_{k-\lceil g_i\rceil}} \theta^{(i)}_{k-\lceil g_i\rceil,\alpha}\mathbf x^{2\alpha}\,,\, i\in[m]\,.
 \end{equation}
Thus \eqref{eq:rep.SOS.1} becomes
 \begin{equation}
 \begin{array}{rl}
     1=&\mathbf v_k^\top \diag((\eta^{(0)}_{k,\alpha})_{\alpha\in \N^n_{k}}) \mathbf v_k+\sum_{i\in[m]}g_i \mathbf v_{k-\lceil g_i\rceil}^\top \diag((\theta^{(i)}_{k-\lceil g_i\rceil,\alpha})_{\alpha\in \N^n_{k-\lceil g_i\rceil}}) \mathbf v_{k-\lceil g_i\rceil}\\
     &+\tilde g_{m+1} \mathbf v_{k-u}^\top \diag((\eta^{(m+1)}_{k-u,\alpha})_{\alpha\in \N^n_{k-u}}) \mathbf v_{k-u}\\
     &+\tilde g_{m+2} \mathbf v_{k-1}^\top \diag((\eta^{(m+2)}_{k-1,\alpha})_{\alpha\in \N^n_{k-1}}) \mathbf v_{k-1}+ \sum_{j\in[l]} h_j \mathbf v_{2(k-\lceil h_j\rceil)}^\top\mathbf u^{(j)}_k\\
     &\in Q^\circ_k(g\cup\{\tilde g_{m+1},\tilde g_{m+2}\})+I_k(h)\,,
     \end{array}
 \end{equation}
 since 
 \begin{itemize}
     \item $\diag((\eta^{(0)}_{k,\alpha})_{\alpha\in \N^n_{k}})\succ 0$, $\diag((\theta^{(i)}_{k-\lceil g_i\rceil,\alpha})_{\alpha\in \N^n_{k-\lceil g_i\rceil}})\succ 0$, $i\in[m]$,
     \item $\diag((\eta^{(m+1)}_{k-u,\alpha})_{\alpha\in \N^n_{k-u}})\succ 0$, and $\diag((\eta^{(m+2)}_{k-1,\alpha})_{\alpha\in \N^n_{k-1}})\succ 0$.
 \end{itemize}
 It yields that  \eqref{eq:find.CTP.diag} has a feasible solution with $g \leftarrow g\cup \{\tilde g_{m+1},\tilde g_{m+2}\}$, for every order $k\in\N^{\ge k_{\min}}$.
 
\end{proof}
\subsection{Proof of Proposition \ref{prop:suff.cond.feas.cs}}
\label{sec:proof.prop:suff.cond.feas.cs}
\begin{proof}
To prove that POP \eqref{eq:POP.def} has CTP on each clique of variables, it is sufficient to show that \eqref{eq:find.CTP.cliq} has a feasible solution, for every $k\in\N^{\ge k_{\min}}$ and for every $j\in [p]$ due to Lemma \ref{lem:feas.LP.cs}.

For every $j\in[p]$, let $\mathbf u^{(j)}=(u_i^{(j)})_{i\in I_j}\subseteq\N^{\le |J_j|}$ be  defined by
\begin{equation}
    u_i^{(j)}=|\{q\in J_j\cap[r]\ :\ i\in T_q\}|+|\{q\in J_j\backslash[2r]\ :\ i\in T_q\}|\,,\,\quad i\in I_j\,.
\end{equation}
For every $j\in[p]$, one has $u_i^{(j)}\in\N^{\ge 1}$, $i\in I_j$, according to $(\cup_{q\in J_j\cap [r]} T_q)\cup (\cup_{q\in J_j\backslash [2r]} T_q)=I_j$.
Moreover,
\begin{equation}
    \|\mathbf u^{(j)}\circ \mathbf x(I_j)\|_2^2=\sum_{i\in J_j \cap[r]}\|\mathbf x(T_i)\|^2_2+\sum_{i\in J_j \backslash[2r]}\|\mathbf x(T_{i})\|^2_2\,,\,\quad\forall j\in[p]\,.
\end{equation}
 For every $j\in[p]$, with 
$R^{(j)}:=\sum_{i\in J_j \cap[r]}(\underline R_i+\overline R_{i})+\sum_{i\in J_j \backslash[2r]}\overline R_{i}$,
 by replacing $\mathbf x$ (resp. $R$) by $\mathbf u^{(j)}\circ \mathbf x(I_j)$ (resp. $R^{(j)}$) in Lemma \ref{lem:equality}, we obtain 
\begin{equation}\label{eq:rep.cs}
    (R^{(j)}+1)^k=(1+\|\mathbf u^{(j)}\circ \mathbf x(I_j)\|^2_2)^k+\Lambda_{ k-1}^{(j)}\sum_{i\in J_j}\delta_ig_i\,,\,\forall j\in [p]\,,\,\forall k\in\N^{\ge k_{\min}}\,,
\end{equation}
where $\Lambda_{ k-1}^{(j)}:=\sum_{r=0}^{k-1}(R^{(j)}+1)^r(1+\|\mathbf u^{(j)}\circ \mathbf x(I_j)\|^2_2)^{k-r-1}$ and
\begin{equation}
    \delta_i:=\frac{\underline R_i}{\overline R_i-\underline R_i}\,,\,\delta_{i+r}:=\frac{\overline R_i}{\overline R_i-\underline R_i}\,,\,i\in J_j\cap [r]\text{ and }\delta_{q}=1\,,\,q\in J_j\backslash [2r].
\end{equation}
It is due to the fact that 
\begin{equation}
    R^{(j)}-\|\mathbf u^{(j)}\circ \mathbf x(I_j)\|_2=\sum_{i\in J_j\cap [r]}(\underline R_i+\overline R_i-\|\mathbf x(T_i)\|^2_2)+\sum_{i\in J_j\backslash [2r]}(\overline R_{i}-\|\mathbf x(T_{i})\|^2_2)\,,
\end{equation}
and 
$\underline R_i+\overline R_i-\|\mathbf x(T_i)\|^2_2=\delta_ig_i+\delta_{i+r}g_{i+r}$, $i\in J_j\cap [r]$.
For every $j\in [p]$, for each $k\in\N^{\ge k_{\min}}$, let $(\theta_{k,\alpha}^{(j)})_{\alpha\in\N^{I_j}_k}\subseteq \R^{>0}$ and $(\eta^{(j)}_{k-1,\alpha})_{\alpha\in\N^{I_j}_{k-1}}\subseteq \R^{>0}$ be such that 
\[(1+\|\mathbf u^{(j)} \circ \mathbf
x(I_j)\|_2^2)^k=\sum_{\alpha\in\N^{I_j}_k}\theta_{k,\alpha}^{(j)}\mathbf x^{2\alpha}\quad \text{and}\quad \Lambda_{k-1}^{(j)}=\sum_{\alpha\in\N^{I_j}_{k-1}}\eta_{k-1,\alpha}^{(j)}\mathbf x^{2\alpha}\,,\]
and define the diagonal matrices:
\begin{equation}
\mathbf G_{k}^{(j,0)}:=\diag((\theta_{k,\alpha}^{(j)})_{\alpha\in\N^{I_j}_k})\quad\text{and}\quad\mathbf G_{k-1}^{(j,i)}:=\diag((\delta_i\eta^{(j)}_{k-1,\alpha})_{\alpha\in\N^{I_j}_{k-1}})\,,\,i\in J_j\,.
\end{equation}
For every $j\in[p]$, \eqref{eq:rep.cs} yields that for every $k\in\N^{\ge k_{\min}}$,
\begin{equation}
    (R^{(j)}+1)^k=(\mathbf v_k^{I_j})^\top \mathbf G_{k}^{(j,0)} \mathbf v_k^{I_j} + \sum_{i\in J_j}g_i(\mathbf v_{k-1}^{I_j})^\top \mathbf G_{k-1}^{(j,i)} \mathbf v_{k-1}^{I_j}\,.
\end{equation}
Hence $((R^{(j)}+1)^k,\mathbf G_{k}^{(j,i)},\mathbf 0)$ is  a feasible solution of \eqref{eq:find.CTP.cliq}, for every $k\in\N^{\ge k_{\min}}$ and for every $j\in [p]$.
\end{proof}

\small

\end{document}